\documentclass[12pt,reqno]{amsart}
\usepackage{amsfonts}
\usepackage{xfrac}    
\usepackage{lipsum}
\usepackage{amsmath,graphicx,amssymb,amsthm}
\usepackage{amsmath}
\usepackage{amssymb}
\usepackage{graphicx,color,xcolor}
\usepackage{hyperref}
\usepackage{enumerate}
\usepackage{mathrsfs}
\usepackage{blindtext}
\usepackage{scrextend}
\usepackage{enumitem}
\usepackage{bm} 
\addtokomafont{labelinglabel}{\sffamily}
\usepackage{color}
\usepackage[at]{easylist}
\providecommand{\U}[1]{\protect\rule{.1in}{.1in}}

\numberwithin{equation}{subsection}
\newcommand{\sett}[1]{\left\{#1\right\}}

\newcommand{\Nat}{\mathbb N}
\newcommand{\eps}{\varepsilon}

\newcommand{\Hil}{\mathcal{H}}

\newcommand{\N}{\mathcal{N}}

\newcommand{\ip}[2]{\left< #1, #2  \right>}
\newcommand{\votimes}{\overline{\otimes}}
\def\U{\mathcal U}

\def\N{\mathcal N}

\def\amslatex{$\mathcal{A}\kern-.1667em\lower.5ex\hbox{$M$}\kern-.125em\mathcal{S}$-\LaTeX}

\newtheorem{thmx}{Theorem}
\newtheorem{thm}{Theorem}[section]
\newtheorem{lem}[thm]{Lemma}
\newtheorem{cor}[thm]{Corollary}
\newtheorem{prop}[thm]{Proposition}
\newtheorem{definition}[thm]{Definition}
\theoremstyle{definition}
\newtheorem{rem}[thm]{Remark}
\newtheorem{remark}[thm]{Remarks}

\newtheorem{claim}[thm]{Claim}

\newtheorem*{opr}{Open Problem}

\newcommand{\cK}{\mathcal{K}}

\newcommand{\QN}{\mathcal{Q}\mathcal{N}}
\newcommand{\cH}{\mathcal H}
\newcommand{\ca}{\curvearrowright}

\newcommand{\Ad}{\operatorname{Ad}}

\newcommand{\diag}{\mathord{\text{\rm diag}}}

\newcommand{\rar}{\rightarrow}

\newcommand{\actson}{\curvearrowright}

\newcommand{\abs}[1]{\left\vert#1\right\vert}
\newcommand{\norm}[1]{\left\Vert#1\right\Vert}

\newtheorem{set}{set}[section]

\newtheorem{Lemma}[set]{Lemma}

\newtheorem{Remark}[set]{Remark}

\newcommand{\define}{\mathrel{\hbox{$\equiv$\hskip -.90em \lower .47ex \hbox{$\leftharpoondown$}}}}
\newcommand{\enifed}{\mathrel{\hbox{$\equiv$\hskip -.90em \lower .47ex \hbox{$\rightharpoondown$}}}}

\theoremstyle{plain}

\allowdisplaybreaks

\newcommand{\Cplx}{\mathbb C}
\newcommand{\nexer}[1]{}

\begin{document}
\title[ Quasinormalizers in crossed products]{Quasinormalizers in crossed products of von Neumann Algebras}
\author[Bannon]{Jon P. Bannon}
\address{Siena College Department of Mathematics, 515 Loudon Road, %Loudonville,
NY\ 12211 USA}
\email{jbannon@siena.edu}

\author[Cameron]{Jan Cameron}
\address{Department of Mathematics and Statistics, Vassar College, Poughkeepsie, NY 12604 USA}
\email{jacameron@vassar.edu}

\author[Chifan]{Ionu\c t Chifan}
\address{Department of Mathematics, The University of Iowa,  Iowa City, IA 52242 USA}
\email{ionut-chifan@uiowa.edu}

\author[Mukherjee]{Kunal Mukherjee}
\address{Indian Institute of Technology Madras, Chennai 600 036, India}
\email{kunal@iitm.ac.in}

\author[Smith]{Roger Smith}
\address{Department of Mathematics, Texas A\&M University, College Station, TX 77843 USA}
\email{rrsmith@tamu.edu}

\author[Wiggins]{Alan Wiggins}
\address{Department of Mathematics and Statistics, University of Michigan, Dearborn, MI 48626 USA }
\email{adwiggin@umich.edu}


\begin{abstract}  We study the relationship between the dynamics of the action $\alpha$ of a discrete group $G$ on a von Neumann algebra $M$, and structural properties of the associated crossed product inclusion $L(G) \subseteq M \rtimes_\alpha G$, and its intermediate  subalgebras.  This continues a thread of research originating in classical structural results for ergodic actions of discrete, abelian groups on probability spaces.  A key tool in the setting of a noncommutative dynamical system is the set of quasinormalizers for an inclusion of von Neumann algebras.  We show that the von Neumann algebra generated by the quasinormalizers captures analytical properties of the inclusion $L(G) \subseteq M \rtimes_\alpha G$ such as the Haagerup Approximation Property, and is essential to capturing ``almost periodic" behavior in the underlying dynamical system.  Our von Neumann algebraic point of view yields a new description of the Furstenberg-Zimmer distal tower for an ergodic action on a probability space, and we establish new versions of the Furstenberg-Zimmer structure theorem for general, tracial $W^*$-dynamical systems.  We present a number of examples contrasting the noncommutative and classical settings which also build on previous work concerning singular inclusions of finite von Neumann algebras.

\end{abstract}
\maketitle

\section{Introduction}

An ongoing thread of research in von Neumann algebras concerns the relationship between the structure of discrete groups and dynamical systems, and the structure of the von Neumann algebras they generate.  A natural site for such questions is the crossed product construction, which arises from the action $\alpha$ of a discrete group $G$ on a von Neumann algebra $M \subseteq \textbf{B}(\Hil)$.  The crossed product $M \rtimes_\alpha G$ is a von Neumann algebra on $\Hil \otimes \ell^2(G)$ which contains an isomorphic copy of $M$, as well as a copy of the von Neumann algebra $L(G)$ generated by the left regular representation of $G$ on $\ell^2(G)$.  The inclusions $M \subseteq M \rtimes_\alpha G$ and $L(G) \subseteq M\rtimes_\alpha G$ are key to understanding the relationship between the structure of $M \rtimes_\alpha G$, the group $G$, and the group action.

\vskip 0.08in
\noindent  
The main results of this paper relate the dynamics of the action $\alpha$ to structural properties of the inclusion $L(G) \subseteq M\rtimes_\alpha G$ and, more generally, inclusions of the form $N \rtimes_\alpha G \subseteq M \rtimes_\alpha G$ with $N\subseteq M.$ Interest in this area originated in the case where $M$ is an abelian von Neumann algebra --  in particular, the dynamics of an action on a probability space, in relation to the structure of the associated group-measure space construction. A principal object of interest at the von Neumann algebra level is the group of unitary normalizers of the inclusion, which has been shown to relate closely to the spectrum of the action.  For an inclusion $B \subseteq M$ of von Neumann algebras, the \emph{$($unitary$)$ normalizers} comprise the set $\mathcal{N}(B \subseteq M) = \sett{u \in \mathcal{U}(M) :uBu^*=B},$
a subgroup of the unitary group of $M$ which generates a von Neumann algebra between $B$ and $M$.  The inclusion $B \subseteq M$ is said to be \emph{singular} if the generated von Neumann algebra $\mathcal{N}(B \subseteq M)''$ is equal to $B$, and \emph{regular} if it is all of $M$.
\vskip 0.08in
\noindent  
An ergodic, measure-preserving action $\sigma$  of a discrete, abelian group $G$ on a probability space $X$ produces a masa (maximal abelian subalgebra) $L(G)$ in the crossed product $L^\infty(X)\rtimes_\sigma G$.  Nielsen \cite{Ni} showed that this masa is singular if and only if the action is weak mixing.  Packer \cite{Pac0} showed, more generally, that the von Neumann algebra of the normalizer of $L(G)$ in $L^\infty (X) \rtimes_\sigma G$ is the intermediate subalgebra $A_0 \rtimes_\sigma G$, where $A_0$ is the invariant subalgebra of $L^\infty(X)$ generated by the eigenfunctions of $\sigma$. In particular, we have that $A_0 = L^\infty(X)$ (in which case the action is \emph{compact}) if and only if the masa $L(G)$ is regular (in this case it is called a \emph{Cartan subalgebra} of $L^\infty(X)\rtimes_\sigma G$). 

\vskip 0.08in
\noindent  Similar characterizations are known beyond the case of an abelian acting group, though the situation is  more complicated.  
 As part of a larger work on profinite actions, Ioana \cite{Io} showed that, for an ergodic action  $G \actson^\sigma (X,\mu)$,  the von Neumann subalgebra $L^\infty(Y) \rtimes_\sigma G \subseteq L^\infty(X) \rtimes_\sigma G$ corresponding to the {maximal compact quotient} of the action is generated by the set  of {quasinormalizers} of $L(G)$ in $L^\infty(X)\rtimes_\sigma G$. These results were proved for countable groups and standard measure spaces. For an inclusion $B \subseteq M$ of von Neumann algebras, the set of \emph{quasinormalizers} \cite{Po} is the collection $\mathcal{QN}(B \subseteq M)$ of elements $x \in M$ with the property that there exist $x_1,\ldots, x_n \in M$ such that 
\[ xB \subseteq \sum_i B x_i, \text{ and } Bx \subseteq \sum_i x_i B.\]
A more general object is the set $\mathcal{QN}^{(1)}(B \subseteq M)$ of \emph{one-sided quasinormalizers}.  An element $x \in M$ is a one-sided quasinormalizer of $B$ if it satisfies the weaker condition that there exist $y_1, \ldots, y_n \in M$ such that
\[ Bx \subseteq \sum_i y_i B.\]
For a general inclusion $B \subseteq M$, the following relationship holds for the von Neumann algebras generated by normalizers and (one-sided) quasinormalizers:
\[ \mathcal{N}(B \subseteq M)'' \subseteq \mathcal{QN}(B \subseteq M)'' \subseteq vN(\mathcal{QN}^{(1)}(B \subseteq M)).\]
In what follows, we will refer to these objects (respectively) as the \emph{normalizing algebra}, the \emph{quasinormalizing algebra}, and the \emph{one-sided quasinormalizing algebra} for the inclusion.  They are generally not equal: an example of Grossman and the sixth named author \cite{GrWi} shows that the first inclusion may be proper; examples of Fang, Gao, and Smith \cite{FaGaSm} show that the second may also be proper (we discuss these further in Section \ref{section:normalizers}).  However, when $B$ is a masa, they coincide.  Thus, Ioana's result 
 \cite{Io} coincides with Packer's \cite{Pac0} when the acting group $G$ is abelian. An alternative proof of the same result was obtained in \cite{ChDa} in the context of progress on an extended  Neshveyev-St\o{}rmer rigidity conjecture. 

\vskip 0.08in
\noindent  This paper concerns analogous questions in the setting of a  \emph{$W^*$-dynamical system} $(M,G,\alpha,\rho)$, consisting of an ergodic action $\alpha$ of a (possibly uncountable) discrete group $G$  on a von Neumann algebra $M$ preserving a fixed, faithful, normal state $\rho$ on $M$.  As a starting point, in light of the results for abelian groups and von Neumann algebras, we consider two  questions:  first, whether the quasinormalizer (rather than the normalizer) for the inclusion $L(G) \subseteq M \rtimes_\alpha G$ is necessary to capture the dynamics of the compact part of the action; second, whether results analogous to those of \cite{Io} mentioned above hold in this more general context.  We answer both questions affirmatively, and explore numerous generalizations and applications.  An outline of the paper and summary of the results follows.
\vskip 0.08in
\noindent  
Any quasinormalizer for a subalgebra $B$ of  a von Neumann algebra $(N,\varphi)$ gives rise to a finitely generated $B$-module in $L^2(N,\varphi)$.  Section \ref{section:background} includes background and preliminaries on $W^*$-dynamical systems, and modules over finite von Neumann algebras. 
\vskip 0.1in
\noindent  Section \ref{section:normalizers} concerns the relationship between normalizers and quasinormalizers of the inclusion $L(G) \subseteq M \rtimes_\alpha G$.  Theorem \ref{example1} describes an example which shows that the strict analogue of Packer's result for normalizers does not hold in this setting, by exhibiting an ergodic action $\sigma$ of a group $G$ on the hyperfinite II$_1$ factor $R$ which is not weak mixing, but for which the inclusion $L(G) \subseteq R \rtimes_\sigma G$ is singular.  As a consequence, we build on results of Grossman and the sixth author on the relationship between singularity and the analytical properties of subfactor inclusions. Theorem \ref{compact quasinormalizer} presents a basic structural result for normalizers of the inclusion $L(G) \subseteq N \rtimes_\alpha G$ associated to a compact, ergodic action of an i.c.c. group on a tracial von Neumann algebra. We deduce from this a number of further examples in which we are able to compute the von Neumann algebras generated by normalizers and quasinormalizers, such as the situation of a profinite action of a discrete group on a probability space.

\vskip 0.07in
\noindent   The main results of Section \ref{section:quasinormalizers} extend Proposition 6.9 of \cite{Io} to ergodic $W^*$-dynamical systems, and characterize quasinormalizers of subalgebras of $L(G)$ in $M \rtimes_\alpha G$ in terms of the dynamics of the underlying system $(M,G,\alpha,\rho)$.  Although our approach to Theorem \ref{theorem:QuasinormalizersIntro} (stated below) is similar in outline to the one appearing in \cite{Io}, our methods are different, since measure-theoretic tools are not available.  Key tools from \cite{BCM1} and \cite{BCM2} isolate the finite-dimensional invariant subspaces of the induced Koopman representation on $L^2(M,\rho)$ associated to the system, and the ``compact quotient" in this setting is modeled by the {\emph Kronecker subalgebra} $M_K$, generated by elements of $M$ with compact orbit under the group action.
 
\begin{thmx}\label{theorem:QuasinormalizersIntro} Let $M$ be a von Neumann algebra and $\rho$ a normal, faithful state on $M$.  Suppose that a discrete group $G$ acts ergodically by $\rho$-preserving automorphisms on $M$.  Then
\[\QN(L(G)\subseteq M \rtimes_\alpha G)'' = M_K \rtimes_\alpha G.\]
\end{thmx}

\noindent The second part of Section \ref{section:quasinormalizers} specializes to the case of a discrete group acting on a tracial von Neumann algebra $(M,\tau)$, in such a way that a von Neumann subalgebra $N \subseteq M$ is left invariant.  The associated system $(N \subseteq M, G, \alpha, \tau)$ is called a \emph{$W^*$-dynamical extension system}, and we consider the dynamics of the action relative to the subalgebra $N$.  In this case, the subspace $\mathcal{P}_{N,M}$ of \emph{relatively almost periodic} elements of $M$ (see Definition \ref{apr}) captures the quasinormalizer of the inclusion $N \rtimes_\alpha G \subseteq M \rtimes_\alpha G$ in terms of the dynamics of the action.  

\begin{thmx}\label{wqncontainment1 intro} \label{quasinormalizer2 intro}  Let $(N\subseteq M, G, \sigma,\tau)$ be a $W^*$-dynamical extension system and assume that $\mathcal{QN}(N\subseteq M)''= M$.  Then $N\subseteq \mathcal P_{N,M}\subseteq M$ is a $G$-invariant von Neumann subalgebra and $vN(\mathcal{QN}^{(1)}(N\rtimes_\sigma G\subseteq M\rtimes_\sigma G))= \mathcal P_{N,M}\rtimes_\sigma G$. 

\end{thmx}

\noindent In the final part of Section \ref{section:quasinormalizers}, we apply these methods to present several von Neumann algebraic versions of the Furstenberg-Zimmer structure theorem  for actions of groups on probability spaces \cite{Fu77,Zi76}, including the following result.

\begin{thmx}\label{fztower1 intro} Let $(N\subseteq M, G, \sigma,\tau)$ be a $W^*$-dynamical extension system. Then one can find  an ordinal $\alpha$ and a $G$-invariant von Neumann subalgebra $N \subseteq Q_\beta \subseteq M$ for every $\beta\leqslant\alpha$ satisfying the following properties:  \begin{enumerate}
\item [1.] For all $\beta\leqslant  \beta' \leqslant  \alpha$ we have  $N=Q_o\subseteq   Q_\beta \subseteq Q_{\beta'} \subseteq M$;
\item [2.] For every successor ordinal $\beta+1 \leqslant  \alpha$ we have $Q_{\beta+1}= \mathcal P_{Q_\beta, M}''$ and  $$vN(\mathcal {QN}^{(1)}(Q_\beta \rtimes_\sigma G\subseteq M\rtimes_\sigma G))\subseteq Q_{\beta+1}\rtimes_\sigma G.$$

\item [3.] For every limit ordinal $\beta \leqslant \alpha$ we have $\overline{\cup_{\gamma<\beta} Q_\gamma}^{\rm SOT}  = Q_\beta$ and $\overline{\cup_{\gamma<\beta} Q_\gamma \rtimes_\sigma G}^{\rm SOT}  = Q_\beta\rtimes_\sigma G$.
\item [4.] There are nets $(g_\lambda)_\lambda\subseteq G$ and $(u_\lambda)_\lambda\subseteq \mathcal U(Q_\alpha)$ such that for every nonzero $x, y \in M \ominus  Q_\alpha$ we have  $$\lim_{\lambda} \|E_ {Q_\alpha} (x u_\lambda\sigma_{g_\lambda}(y))\|_2= 0.$$

\end{enumerate}
\end{thmx}
\noindent As a consequence, we obtain the following purely von Neumann algebraic description of the classical Furstenberg-Zimmer tower, using quasinormalizers and relatively almost periodic elements, recapturing a previously unpublished result \cite{CP11} of the third-named author and Peterson.  
\vskip 0.07in

\begin{thmx}[\cite{CP11}]\label{fztower3 intro} Let $G \overset{\sigma}\ca X$ be a probability measure-preserving $(p.m.p.$ in the sequel$)$ ergodic action of a countable discrete group $G$ on a standard probability space $X$ and let $(G \overset{\sigma}\ca X_\beta)_{\beta\leqslant \alpha}$ be the corresponding Furstenberg-Zimmer tower. Let $M = L^\infty(X)\rtimes_\sigma G$ and $M_\beta = L^\infty(X_\beta)\rtimes_\sigma G$ be the corresponding crossed product von Neumann algebras.  Then the following hold:  \begin{enumerate}
\item [1.] For all $\beta\leqslant  \beta' \leqslant  \alpha$ we have the following inclusions $L(G)=M_o\subseteq   M_\beta \subseteq M_{\beta'} \subseteq M_\alpha \subseteq M$;
\item [2.] For every successor ordinal $\beta+1 \leqslant  \alpha$ we have that    $vN(\mathcal {QN}^{(1)}(M_\beta\subseteq M))=\mathcal{QN}(M_\beta\subseteq M )''=M_{\beta+1}$.  Moreover, there is a sequence $(g^\beta_n)_n\subseteq G$ such that for every $x,y\in L^\infty(X)\ominus L^\infty(X_{\beta+1})$ we have 
$$\lim_{n\rightarrow \infty} \|E_{L^\infty(X_\beta)}(x\sigma_{g^\beta_n}(y))\|_2=0.$$
\item [3.] For every limit ordinal $\beta \leqslant \alpha$ we have $\overline{\cup_{\gamma<\beta} L^\infty(X_\gamma)}^{\rm SOT}  = L^\infty(X_\beta)$ and also $\overline{\cup_{\gamma<\beta} M_\gamma}^{\rm SOT}  = M_\beta$.
\item [4.] There is an infinite sequence $(g_n)_n\subseteq G$ such that for every nonzero $x, y \in L^\infty(X) \ominus  L^\infty(X_\alpha)$ we have  $$\lim_{n\rar \infty} \|E_ {L^\infty(Y_\alpha)} (x \sigma_{g_n}(y))\|_2= 0.$$
\end{enumerate}
\end{thmx}
\noindent  In Section \ref{section:approximation} we turn to analytical properties of the inclusion $L(G) \subseteq M \rtimes_\alpha G$. Here, we assume the group $G$ acts compactly and ergodically on a finite von Neumann algebra $(M,\tau)$, preserving the trace.  
The Haagerup approximation property for finite von Neumann algebras was developed in \cite{Ch} as a counterpart to the Haagerup property for groups \cite{Ha}.  The two properties weaken the notion of amenability for von Neumann algebras and discrete groups, respectively, and encompass many more useful examples, including free groups and their associated von Neumann algebras.  It was shown in \cite{Ch} that a (countable) discrete group $G$ has the Haagerup property if and only if $L(G)$ satisfies the von Neumann algebraic version.   A relative version of the Haagerup property for inclusions $B \subseteq N$ of finite von Neumann algebras first appeared in \cite{Bo} and has been employed extensively, for instance, in Popa's celebrated results on the class of $\mathcal{HT}$ algebras \cite{Po}.  Notably, the relative Haagerup property for an inclusion $B \subseteq N$ is not a weakening of relative amenability; there are examples of inclusions for which the latter condition holds, but not the former \cite{CaKlSkVoWa}.  Our main result in Section \ref{section:approximation} builds on the methods and results of Section \ref{section:quasinormalizers} to show that in the above setting, the relative Haagerup property for the inclusion $L(G) \subseteq M \rtimes_\alpha G$ encodes the dynamics of the group action.  
\begin{thmx} \label{theorem:ApproximationIntro} If $G$ is a discrete group and $(M,G,\alpha,\tau)$ an ergodic, trace-preserving $W^*$-dynamical system, then $L(G)\subseteq M \rtimes_\alpha G$ has the relative Haagerup property if and only if the action $\alpha$ is compact.  
\end{thmx} 
\noindent Theorem \ref{theorem:ApproximationIntro} combines with Theorem \ref{theorem:QuasinormalizersIntro} to establish a fully general noncommutative analogue of Proposition 6.9 of \cite{Io}.  The same main tools from \cite{BCM1} and \cite{BCM2} are employed in the proof;  versions of those results suitable for our purposes are stated in Lemmas \ref{Invariant Subspaces Lemma} and \ref{Kronecker Subalgebra Lemma}.

  \begin{Remark}   During the preparation of this paper, the authors became aware of the preprint \cite{JaSp} which has some overlap in subject matter.  Although there is some overlap in techniques between that note and this one (e.g., the authors develop some variations on the methods in \cite{BCM1} and \cite{BCM2}), there is little overlap in the results. \end{Remark}

\section{Background and Preliminaries} \label{section:background} \noindent We recall in this section basic facts about von Neumann algebras, $W^*$-dynamical systems and extension systems, and modules over finite von Neumann algebras to be used in the sequel (the reader may consult the books \cite{AnPo}, \cite{SiSm}, \cite{StZs} for further details). Let $N$ be a von Neumann algebra and $\varphi$ a normal, faithful state on $N$.  The \emph{centralizer} of $\varphi$ is the von Neumann subalgebra $N^\varphi = \sett{x \in N: \varphi(xn)=\varphi(nx),\,n \in N}$ of $N$. Note that the restriction of $\varphi$ to $N^\varphi$ is a trace, so $N^\varphi$ is a finite von Neumann algebra. We denote the GNS space associated to $\varphi$ by $L^2(N,\varphi)$ (or, simply, $L^2(N)$ when the context is clear), and the canonical cyclic and separating vector in $L^2(N,\varphi)$ by $\Omega_\varphi$ (or leave off the subscript when context allows).  When $N$ is in standard form on $L^2(N,\varphi)$, the embedding $x \mapsto x \Omega_\varphi$ induces a norm on $N$ which we denote by $\norm{\cdot}_2$. There also exists a conjugate linear isometry $J$ on $L^2(N,\varphi)$, which is the polar part of the preclosed map $x\Omega_\varphi \mapsto x^* \Omega_\varphi$, and satisfies $N'=JNJ$.

\subsection{$W^*$-dynamical systems}
A \emph{$W^*$-dynamical system} (or, simply, a \emph{system}) is a quadruple $(M,G,\alpha,\rho)$ consisting of a von Neumann algebra $M$ with a normal, faithful state $\rho:M\rightarrow \mathbb{C}$ $($i.e. $M$ is $\sigma$-finite), together with a strongly continuous action $\alpha$ of a locally compact group $G$ on $M$ by $\rho$-preserving automorphisms.  In this paper, the group $G$ will always be assumed to be discrete (possibly uncountable).  Since we will often make reference to the following well-known concepts, we remind the reader of their definitions.

\begin{definition}\label{basicdef}

\medskip
    \begin{itemize}
    \item[{\rm{(i)}}] 
    A system is said to be \emph{ergodic} if the scalar multiples of the identity are the only elements of $M$ fixed by $\alpha_g,$ $g \in G$.
\medskip
    \item[{\rm{(ii)}}]
    A system is \emph{compact} if for any $x \in M$, the orbit ${\rm Orb}(x)=\sett{\alpha_g(x):g \in G}$ has compact closure in the $\norm{\cdot}_2$-norm.
\medskip
    \item[{\rm{(iii)}}]
    A system in which the action $\alpha$ of a discrete group $G$ on a von Neumann algebra $M$ leaves a von Neumann subalgebra $N \subseteq M$ $($globally$)$ invariant will be called a \emph{$W^*$-dynamical extension system} $($or, simply, an \emph{extension system}$)$ and denoted by a quadruple $(N \subseteq M, G, \alpha, \rho)$.

    \end{itemize}
\end{definition}

\vskip 0.06in
\noindent  If $(M,G,\alpha,\rho)$ is a system with $M$ in standard form on $L^2(M,\rho)$, there is a faithful, normal representation $\pi$ of $M$ on $L^2(M,\rho)\otimes \ell^2(G)$, given by 
\[\pi(x)(\xi \otimes \delta_h)=\alpha_h^{-1}(x)\xi \otimes \delta_h, \quad \xi \in L^2(M,\rho),\, h \in G,\]
and a unitary representation $u$ of $G$ on $L^2(M,\rho)\otimes \ell^2(G)$, given by 
\[u_g (\xi \otimes \delta_h) = \xi \otimes \delta_{gh}, \quad \xi \in L^2(M,\rho),\, h \in G.\]
Note that for $g \in G$ and $x \in M$ we have the relation
\[u_g \pi(x) u_g^* = \pi(\alpha_g(x)).\]
The von Neumann algebra on $L^2(M,\rho)\otimes \ell^2(G)$  generated by $\pi$ and $u$ is known as the \emph{crossed product}, and is denoted by $M \rtimes_\alpha G$.  The crossed product is the $w^*$-closure of the $\ast$-algebra of finite linear combinations of operators of the form $\pi(x)u_h$,  $x\in M$, $h \in G$, which we will denote by $x u_h$ for brevity.  The functional $\widehat{\rho}$ defined on finitely nonzero sums by 
\begin{equation}\label{statecrossedproduct}\widehat{\rho}(\sum_{g \in G}x_g u_g) = \rho(x_e)\end{equation}
extends to a normal, faithful state on $M \rtimes_\alpha G$.  It follows from this construction that the GNS space $L^2(M\rtimes_\alpha G,\widehat{\rho})$ is isomorphic to $L^2(M,\rho)\otimes \ell^2(G)$, and that both $M$ and $L(G)$ are (unital) von Neumann subalgebras of $M\rtimes_\alpha G$. Note in particular that $G$-invariance of $\rho$ implies that the centralizer of $\widehat{\rho}$ in $M \rtimes_\alpha G$ contains $L(G)$. 

\vskip 0.07in
\noindent  An essential tool for investigating the inclusion $L(G) \subseteq M \rtimes_\alpha G$ is Jones's \emph{basic construction}.  In this setting, it is the von Neumann subalgebra $\ip{M\rtimes_\alpha G}{e_{L(G)}}$ of $\textbf{B}(L^2(M \rtimes_\alpha G,\widehat{\rho}))$ generated by $M\rtimes_\alpha G$ and the orthogonal projection $e_{L(G)}:L^2(M \rtimes_\alpha G,\widehat{\rho}) \rightarrow L^2(L(G))$.  It can be shown that $\ip{M\rtimes_\alpha G}{e_{L(G)}} = (JL(G) J)',$ i.e., $\ip{M\rtimes_\alpha G}{e_{L(G)}}$ consists of operators $T \in \textbf{B}(L^2(M \rtimes_\alpha G,\widehat{\rho}))$ that commute with the right action of $L(G)$ on $L^2(M \rtimes_\alpha G,\widehat{\rho})$.  It follows that the projections in $\ip{M\rtimes_\alpha G}{e_{L(G)}}$ are in bijective correspondence with the right $L(G)$-submodules of $L^2(M\rtimes_\alpha G,\widehat{\rho})$.  

\subsection{Modules over finite von Neumann algebras}\label{section:modules}
We recall here some basic facts about modules over finite von Neumann algebras, to be used in the sequel. Let $N$ be a finite von Neumann algebra, with a fixed faithful normal trace $\tau$.  A \emph{left $($respectively, right$)$ $N$-module} is a Hilbert space $\Hil$, paired with a normal, unital homomorphism (respectively, anti-homomorphism) $\pi$ of $N$ into $\mathbf{B}(\Hil)$.  This note will focus primarily on right modules. If $(\Hil,\pi)$ is a right $N$-module, then there exist a set $S$, a projection $p \in \mathbf{B}(\ell^2(S))\votimes N$, and a unitary $U:\Hil \rightarrow p(\ell^2(S)\otimes L^2(N))$ which intertwines $\pi$ with the representation $I \otimes JNJ$ of $JNJ$ on $p(\ell^2(S)\otimes L^2(N))$, so the two spaces are isomorphic as $N$-modules. This is proved in \cite[Proposition 8.2.2]{AnPo} for separable left $N$-modules where $S$ can be taken to be $\Nat$, and has an easy extension to the general case, as noted in the footnote on \cite[p. 124]{AnPo}.

%A closed subspace of a right $N$-module $\Hil$ which is itself a right $N$-module is called a \emph{submodule} of $\Hil$.  

\vskip 0.08in
\noindent  If $(\Hil,\pi)$ is a right $N$-module, for $x \in N$ and $\eta \in \Hil$, we simply denote $\pi(x)\eta$ by $\eta x$.  Any element $\xi \in \Hil$ gives rise to an $($possibly unbounded$)$ operator $L_\xi: L^2(N) \rightarrow \Hil,$ defined on the dense subspace $N \Omega$ by $L_\xi(x \Omega) = \xi x$.  If $L_\xi$ extends to a bounded operator, that is, there is some $C>0$ such that $\norm{\xi n} \leqslant C\norm{n}_2$ for all $n \in N$, then the vector $\xi$ is said to be \emph{left-bounded}.  The set $\Hil^0$ of left-bounded vectors of $\Hil$ is a dense subspace of $\Hil$.  Moreover, for any $\xi, \eta \in \Hil^0$, the operator $L_{\xi}^* L_{\eta} \in \mathbf{B}(L^2(N))$ commutes with the right action of $N$ on $L^2(N)$, so defines an element of $N$ itself.  In this way, $\Hil^0$ is endowed with an $N$-valued inner product, given by $\ip{\xi}{\eta}=L_{\xi}^*L_{\eta}$, with the additional property that $L_{\xi}^*(\eta) = \ip{\xi}{\eta}\Omega$ for any $\xi,$ $\eta \in \Hil^0$. 

\vskip 0.08in
\noindent A right $N$-module $\Hil$ is said to be \emph{finitely generated} if there exist $\xi_1, \ldots , \xi_n \in \Hil$ such that $\Hil$ is the closure of $\sum_i \xi_i N$.  A Gram-Schmidt argument shows that, in this case, the $\xi_i$ may be taken to be left-bounded vectors which are  \emph{orthonormal} with respect to the $N$-valued inner product.   That is, if $\Hil$ is finitely generated, then there exist $\eta_1,  \ldots, \eta_n \in \Hil^0$ such that $\Hil = \sum_i \overline{\eta_i N }$, and for each $i,j$ we have $\ip{\eta_i}{\eta_j} = \delta_{ij} p_j$ for some projection $p_j \in N$.  It follows that any $\zeta \in \Hil^0$ may be expressed uniquely as a Fourier series over $N$ by
\[ \zeta = \sum_i \eta_i \ip{\eta_i}{\zeta}.\]
The above identification of right $N$-modules with subspaces of $\ell^2(S) \otimes L^2(N)$ maps the finitely generated right $N$-modules to those of the form $p(\ell^2(S) \otimes L^2(N))$, where $p \in \textbf{B}(\ell^2(S))\votimes N$ has the form $\oplus_{k=1}^n q_k$, for projections $q_1, \ldots, q_n \in N$.  An immediate consequence of this, which we will use implicitly in what follows, is that submodules of finitely generated modules are also finitely generated.  
\vskip 0.08in

\noindent  For further use, we continue by recalling a result on finitely generated bimodules from \cite[Lemma 3.4]{FaGaSm} which, in turn, was inspired by \cite[Theorem 1.4.2]{Po}. 
\begin{thm}[\cite{FaGaSm}]\label{technicalresult} Let $N\subseteq (M, \tau)$ be an inclusion of tracial von Neumann algebras. Suppose that $\cH \subseteq L^2(M)$ is an $N$-bimodule, and that $\cH$ is a finitely generated right $N$-module with an orthonormal basis of length $k$. Let $P_\cH$ be the orthogonal projection
of $L^2(M)$ onto $\cH$. Then there exists a sequence of projections $z_n \in N' \cap M$ such that
$\lim_{n\rar \infty} z_n = 1$ in SOT and for each $n$ there exist finitely many elements   $x_{n,1}, \ldots , x_{n,k} \in M$ that are $N$-orthogonal and satisfy 
$$z_nP_\cH z_n(x\Omega) = \sum^k_{i=1} x_{n,i}E_N(x^*_{n,i}x)\Omega \text{, for all }x\in M. $$
\end{thm}
\noindent Next we highlight a consequence of Theorem \ref{technicalresult}  which is needed in the proofs of the main results in Section 6. For the reader's convenience we include a complete proof. 

\begin{thm}\label{bimodularapprox} Let $N\subseteq (M, \tau)$ be an inclusion of tracial von Neumann algebras. Suppose that $\cH \subseteq L^2(M)$ is an $N$-bimodule that is  finitely generated as a right $N$-module. Let $\xi_1,\ldots,\xi_n \in \cH$. Then for every $\varepsilon>0$ there are $\eta_1,\ldots,\eta_n\in M$ with $\|\xi_i-\eta_i \|_2<\frac{\varepsilon}{16}$ for all $1\leqslant i\leqslant n$ together with $N$-orthogonal elements $x_1,\ldots,x_k \in M$ such that for all $a_i,b_i\in (N)_1$ with $1\leqslant i\leqslant n$ we have \begin{equation}\label{bimodularapproxeq}
     \|\sum_i a_i\xi_i b_i- \sum_{i,j} x_j E_N(x_j^* a_i \eta_i b_i)\|_2\leqslant\varepsilon. 
 \end{equation} 
Moreover, for every $\zeta\in M$ we have \begin{equation}\label{contraction}
    \|\sum_i x_i E_N(x_i^* \zeta)\|_2\leqslant\|\zeta\|_2.
\end{equation}  
If $\xi_i\in M \cap \cH$ then we can take $\eta_i=\xi_i$ above. In particular, this holds for every  $\xi\in \mathcal{QN}^{(1)}(N\subseteq M)$.
\end{thm}

\begin{proof}Let $\eta_i\in M$ be such that $\|\xi_i-\eta_i \|_2<\frac{\varepsilon}{16n}$ so that $\|\eta_i-P_\cH(\eta_i)\|_2< \frac{\varepsilon}{8n}$. By Theorem \ref{technicalresult} one can find a sequence of projections $z_l\in N'\cap M$ such that $z_l \rar 1$ in SOT and for every $l$ there exist $N$-orthogonal elements $x_{l}^s\in M$ for $1\leqslant s\leqslant {j}$ ($j$ depending on the length of orthonormal basis of $\cH$) such that for all $\eta\in M $ we have
\begin{equation}\label{leftbasis4}\begin{split}
    z_l P_{\cH}z_l(\eta)=\sum_s x^s_{l}E_{N}((x^s_{l})^*\eta).
    \end{split}
\end{equation} 
Since $z_l \rightarrow 1$ in SOT we have that $P_{\cH} -z_l P_{\cH} z_l \rar 0$ in SOT as $l \rar \infty$. Thus there is $l$ large enough such that we have  $\|P_{\cH}(\eta_i)-z_l P_{\cH} z_l(\eta_i)\|_2\leqslant \frac{\varepsilon}{8n}$ for all $i$. For every $a,b\in N$ and $\xi_0\in L^2(M)$ we have  $P_{\cH}(a \xi_0 b)=a P_{\cH}( \xi_0)b$ and since $z_l \in N'\cap M$ we also have $z_lP_{\cH}z_l(a \xi_0 b)=a z_lP_{\cH}z_l( \xi_0)b$. Thus  we conclude that for all $a_i,b_i\in (N)_1$ with $1\leqslant i\leqslant n$ we have 

\begin{equation*}\label{sotapprox}\begin{split}&\|\sum_i a_i \xi_i b_i-z_l P_{\cH} z_l\left(\sum_i a_i \xi_i b_i\right)\|_2\leqslant \sum_i\| a_i \xi_i b_i-z_l P_{\cH} z_l(a_i \xi_i b_i)\|_2 \\ 
& \leqslant \sum_i \left(2\|\xi_i-\eta_i\|_2+\|P_{\cH}( a_i \eta_i b_i)-z_l P_{\cH} z_l(a_i \eta_i b_i)\|_2 +\frac{\varepsilon}{8n}\right ) \\
& \leqslant\frac{\varepsilon}{4} + \sum_i\|P_{\cH}(\eta_i)-z_l P_{\cH} z_l(\eta_i)\|_2 \leqslant \frac{\varepsilon}{4}+\frac{\varepsilon}{8}\leqslant\frac{\varepsilon}{2}.\end{split}\end{equation*}

Combining this with \eqref{leftbasis4} we get \eqref{bimodularapproxeq}. Also notice that \eqref{leftbasis4} gives \eqref{contraction}.

\noindent When $\xi_i\in M$ we can obviously take $\eta_i=\xi_i$. 
When $\xi\in \mathcal{QN}^{(1)}(N\subseteq M)$ one can find $y
_1,\ldots,y_{j} \in M $ such that $N \xi\subseteq \sum_i y_i N $. Letting $ \mathcal H$ to be the linear closure of  $N \xi N$ we see that $\mathcal H$ is an $N$-bimodule that is finitely generated as a right $N$-module. Then the conclusion follows from the previous part.\end{proof}

\noindent We conclude this section with some remarks on right modules arising in the setting of crossed products.  Let $(M,G,\rho,\alpha)$ be a $W^*$-dynamical system, and denote by $M\rtimes_\alpha G$ the associated crossed product.  Then the GNS space $L^2(M\rtimes_\alpha G,\widehat{\rho})$ (where $\widehat{\rho}$ is the state defined in Equation \eqref{statecrossedproduct}) is a right $L(G)$-module, which may be identified with the right $L(G)$-module  $\ell^2(G)\otimes L^2(M,\rho)$, via the isomorphism which extends
\[ m g \, \Omega_{\widehat{\rho}} \longmapsto \delta_g \otimes m \Omega_{\rho}, \quad m \in M, g \in G.\]
 Then any $\eta \in L^2(M \rtimes_\alpha G, \widehat{\rho})$ may be expressed as a function  $g \mapsto \eta(g) \in L^2(M, \rho)$, with $\sum_{g \in G} \norm{\eta(g)}_2^2 < \infty$.   Likewise, any $x \in L(G)$ may be expressed as a square-summable sequence $(\beta_h)_{h \in G}$. 
Under this identification, the right action of $L(G)$ on $L^2(M \rtimes_\alpha G, \widehat{\rho})$ is given by the convolution formula
\begin{align}\label{convformula} (\eta x)(h) = \sum_{k \in G} \beta_{k^{-1}h} \eta(k).\end{align}
We will use this identification and convolution formula in our investigation of submodules of $L^2(M\rtimes_\alpha G,\widehat{\rho})$ in Section \ref{section:quasinormalizers}.

\section{Normalizers and quasinormalizers of $L(G) \subseteq M \rtimes_\alpha G$} \label{section:normalizers}

\noindent In this section, we present several classes of examples of actions $\alpha$ of a discrete group $G$ on a von Neumann algebra $M$ for which $L(G)$ is singular in the crossed product $M \rtimes_\alpha G$, while nontrivial quasinormalizers of $L(G)$ exist (see Theorem \ref{example1} and Corollary \ref{profinite normalizer2}).  These results accomplish two goals.  First, they show that unitary normalizers in the associated crossed product are not sufficient to capture the dynamics of such an action, so the precise statements of the results of Nielsen, \cite{Ni}, and Packer, \cite{Pac0}, which inspired this work do not hold in this setting.  Second, they expand a collection of examples introduced in \cite{GrWi} of inclusions of von Neumann algebras which are singular, but do not satisfy the weak asymptotic homomorphism property. Recall the following definition.
\begin{definition}
An inclusion $B \subseteq M$ of finite von Neumann algebras, with conditional expectation $E_B:M \rightarrow B$, satisfies the \emph{weak asymptotic homomorphism property (WAHP)} if there is a net $(u_\lambda)_\lambda$ of unitaries in $B$ such that for any $x,y \in M,$
\[\norm{E_B(x u_\lambda y)-E_B(x)u_\lambda E_B(y)}_2 \rightarrow 0.\]
\end{definition}
\noindent The WAHP was introduced in \cite{RoSiSm}, where it was shown that any inclusion satisfying the WAHP is singular.  This property has been useful in constructing examples of singular inclusions, as singularity is generally hard to verify using the definition.  The WAHP is known to be equivalent to singularity of $B \subseteq M$ when $B$ is a masa \cite{SiSmWhWi}.  By contrast, Grossman and Wiggins \cite{GrWi} produced inclusions $N \subseteq M$ of II$_1$ factors which are singular, but do not satisfy the WAHP. These inclusions had finite Jones index; they showed, more generally, that no finite  index inclusions satisfy the WAHP.  
\vskip 0.06in
\noindent A more general version of the WAHP has been useful in the study of  one-sided quasinormalizers. Specifically, a triple inclusion  $B\subseteq N\subseteq M$ of finite von Neumann algebras satisfies \emph{the relative WAHP} if there is a net $(u_\lambda)_\lambda$ of unitaries in $B$ such that for any $x,y \in M,$
\[\norm{E_B(x u_\lambda y)-E_B(E_N(x)u_\lambda E_N(y))}_2 \rightarrow 0.\]
It is well-known that if  $B\subseteq N\subseteq M$ satisfies the relative WAHP then $N$ absorbs all one-sided quasinormalizers of $B$ in $M$. In  \cite[Theorem 3.1]{FaGaSm} a converse result was established which asserts, essentially, that this analytic property characterizes
the von Neumann algebra generated by the one-sided quasinormalizers of $B$.

\begin{thm}[\cite{FaGaSm}]\label{rwahp}
 Let $B\subseteq (M, \tau)$ be tracial von Neumann algebras and denote by $N:= vN(\mathcal {QN}^{(1)}(B\subseteq M))$. Then  $B\subseteq N\subseteq M$ satisfies the relative WAHP.    
\end{thm}

Using this result one can establish that the relative WAHP actually ``descends'' to all subgroups of $\mathcal U(B)$ that generate $B$ as a von Neumann algebra. As we will see shortly, this upgrade is very useful in applications. We include only a brief proof, largely based on prior techniques \cite{Po03,C06,FaGaSm,JoSt,CaFaMu} and we encourage the reader to consult these results beforehand.

\begin{thm}\label{srwm} A triple inclusion $B\subseteq N\subseteq M$ of finite von Neumann algebras has the relative WAHP if and only if for every subgroup $\mathscr B \subseteq \mathcal U(B)$ satisfying $\mathscr B''=B$ one can find a net $(g_\lambda)_\lambda \subseteq \mathscr B$ such that for any $x,y\in M$ we have 
\begin{equation}\label{rwahpsubg}\|E_B(x g_\lambda y) - E_B(E_N(x) g_\lambda E_N(y))\|_2 \rightarrow 0.\end{equation}
    \end{thm}
\begin{proof}We only prove the forward implication as the converse is straightforward.
Assume by contradiction there is a subgroup $\mathscr B \subseteq \mathcal U(B)$ satisfying $\mathscr B''=B$ for which \eqref{rwahpsubg} does not hold. Thus using the same argument from the proof of \cite[Corollary 2.3]{Po03} one can find a scalar $C>0$ and a finite subset $ \emptyset\neq F\subset M \ominus N$ such that for all $b\in \mathscr B$, \begin{equation}\label{awayfrom0}
    \sum_{x,y\in F} \|E_{B}(x^* b y)\|_2^2\geqslant C.
\end{equation}
Consider the basic construction $B \subseteq M \subseteq\langle M,B\rangle= \{M,e_B\}''\subseteq \textbf{B}(L^2(M))$, where $e_B: L^2(M) \rightarrow L^2(B)$ is the canonical orthogonal projection. Let $Tr$ be the canonical semifinite trace on $\langle M,B\rangle$ given by $Tr(xe_B y)=\tau(xy)$ for all $x,y\in M$. Let $\xi:= \sum_{x\in F} xe_B x^*\in \langle M,B\rangle_+$ and notice $0<Tr(\xi)<\infty$. Using $e_Bme_B =E_B(m)e_B$ for all $m\in M$ together with other basic calculations and \eqref{awayfrom0} we see for all $b\in \mathscr B$, 
\begin{equation}\label{awayfrom02}\begin{split}
Tr(\xi b \xi b^*)&= \sum_{x,y\in F}Tr( xe_B x^* b ye_B y^* b^*)=\sum_{x,y\in F}Tr( e_B x^* b ye_B y^* b^* x e_B)=\\&
= \sum_{x,y\in F} Tr(E_B (x^*b y)e_B E_B (y^* b^* x))= \sum _{x,y\in F}\|E_B(x^*b y)\|_2^2\geqslant C.\end{split}\end{equation}
Let $K=\overline{{\rm co} \{ b\xi b^* \,:\, b\in \mathscr B\}}^{w}$ and denote by $\eta\in K$ the unique element of minimal $\|\cdot \|_{2,Tr}$-norm.  Fix $b\in \mathscr B$. Since $Tr$ is a trace then $\|b \eta b^*\|_{2,Tr}=\|\eta\|_{2,Tr}$. Also, since $\mathscr B$ is a group then $b\eta b^*\in K$. Thus uniqueness implies that $b\eta b^*=\eta $ for all $b\in \mathscr B$ and since $\mathscr B''=B$ we conclude that $\eta \in B'\cap \langle M,B\rangle_+$. One can also check that $Tr(\eta)\leqslant Tr(\xi)<\infty$. Furthermore, \eqref{awayfrom02} entails $\eta\neq 0$. 
\vskip 0.03in Now consider the orthogonal projection $e_N: L^2(M)\rightarrow L^2(N)$ and notice that $e_N \in N'\subseteq B' $. Moreover, as $Je_N=e_N J$ we also have $e_N \in JB'J=\langle M,B\rangle$ and hence $e_N \in B'\cap \langle M,B\rangle$. Next we can see that for every $b\in \mathscr B$, 
\begin{equation*}
    e_N b \xi b^*= be_N \xi b^* =\sum_{x\in F} b e_N xe_B x^*b^*= \sum_{x\in F} b E_N (x) e_B x^*b^*=0. 
\end{equation*} 
Taking convex combinations and weak limits, this further implies that $e_N \eta =0$. Thus $\eta e_N =0$ and hence $\eta\in (1-e_N) (B'\cap \langle M,B\rangle) (1-e_N)$. Taking a suitable spectral projection of $\eta$ one can find a projection $0\neq p\in (1-e_N) (B'\cap \langle M,B\rangle) (1-e_N)$ such that $Tr(p)<\infty$. 
\vskip 0.07in
Denote by $Q := vN (\mathcal{QN}^{(1)}(B\subseteq M))$ and let $e_Q: L^2(M)\rightarrow L^2(Q)$ be the canonical orthogonal projection. On the one hand, using verbatim the same arguments from the proof of the implication $(ii)\Rightarrow (i)$ in \cite[Theorem 3.1]{FaGaSm} (see page 9/line -2 --- page 10/line 10) we get that $p \leqslant e_Q$. On the other hand, as $B\subseteq N\subseteq M$ satisfy the relative WAHP, implication $(i)\Rightarrow (iii)$ in \cite[Theorem 3.1]{FaGaSm} yields $Q\subseteq N$ and hence $e_Q\leqslant e_N$. Altogether, these imply $p\leqslant e_N$. Since $p\leqslant 1-e_N$ we get that $p=0$, which is a contradiction. \end{proof}

\noindent The relative WAHP is closely connected to the following notion of relative weak mixing for trace-preserving $W^*$-dynamical extension systems.

\begin{definition}\label{rwm} Let $\mathfrak M=(N\subseteq M, G, \alpha,\tau)$ be a $\tau$-preserving $W^*$-dynamical extension system and let $B \subseteq N$ be a $G$-invariant von Neumann subalgebra. Then $\mathfrak M$ is called \emph{weak mixing relative to $B$} if there exist nets $(b_\lambda)_\lambda \subseteq \mathcal U(B)$ and $(g_\lambda)_\lambda\subseteq G$ such that for all $x,y\in M\ominus N$ we have  \[\| E_B(x b_\lambda \alpha_{g_\lambda}(y)) \|_2\rightarrow 0.\]      
When $M$ is separable the nets can be replaced with sequences.\end{definition}

\noindent We note in passing that this generalizes Popa's notion of relative weak mixing for actions, described in \cite[Definition 2.9]{Po05}. Indeed, it is rather easy to check if one could pick $(b_\lambda)_\lambda$ to have only finitely many values then Definition \ref{rwm} is equivalent to Popa's notion. For instance, this is the case when $B \subseteq \mathcal Z(M)$ (one can pick $b_\lambda=1$). Thus, when $M$ is abelian, Definition \ref{rwm} recovers the notion of weak mixing for extensions introduced by Furstenberg and Zimmer in the 70's, \cite{Fu77,Zi76}. Finally, when $B=N=\mathbb C 1$, this recovers the notion of weak mixing for trace-preserving actions of $G$ on $M$.

\vskip 0.05in
\noindent For further use we record the following result connecting relative weak mixing with relative WAHP. Its proof is a straightforward application of Theorem \ref{srwm} and other existing methods in the literature (\cite{JoSt}, \cite{CaFaMu}, \cite{FaGaSm}) and we include it here only for the sake of completness.
\begin{lem} \label{lemma:WeaklyMixingSingular}
 Let $\mathfrak M=(N\subseteq M, G, \alpha,\tau)$ be a $\tau$-preserving $W^*$-dynamical extension system and let $B \subseteq N$ be a $G$-invariant von Neumann subalgebra. Then $\mathfrak M$ is weak mixing relative to $B$ if and only if the triple inclusion $B \rtimes_\alpha G \subseteq N \rtimes_\alpha G\subseteq M\rtimes_\alpha G$ has the relative WAHP.
\end{lem}

\begin{proof}
First we show the forward implication. Let $(b_\lambda)_{\lambda\in \Lambda}\subset \mathcal U(B)$ and $(g_\lambda)_{\lambda\in \Lambda}\subseteq G$ be such that for every $\xi,\zeta \in M\ominus N$ we have 
\begin{equation}\label{rwm3}
    \|E_B(\xi b_\lambda \alpha_{g_\lambda}(\zeta))\|_2\rightarrow 0.
\end{equation} Next we show the net $x_\lambda:= b_\lambda u_{g_\lambda}\in \mathcal U(B\rtimes_\alpha G)$ witnesses the relative WAHP for $B \rtimes_\alpha G \subseteq N \rtimes_\alpha G\subseteq M\rtimes_\alpha G$. Notice this is equivalent to showing that for every $0\neq y,z\in (M\rtimes_\alpha G) \ominus (N\rtimes_\alpha G)$ we have    \begin{equation}\label{rwm4}
    \| E_{B\rtimes_\alpha G} (yx_\lambda z)\|_2\rightarrow 0.
\end{equation}  Fix $\varepsilon>0$. Using the Kaplansky density theorem one can find finite subsets $E_\varepsilon, F_\varepsilon \subset G$ and $y_\varepsilon= \sum_{g\in E_\varepsilon} y_g^\varepsilon u_g  , z_\varepsilon =\sum_{h\in F_\varepsilon} z_h^\varepsilon u_h$ with $y_g^\varepsilon, z_h^\varepsilon \in M\ominus N$ such that \begin{equation}\label{est1}\begin{split}
    &\|y-y_\varepsilon\|_2<\frac{\varepsilon}{4\|z\|}\text{ and }
     \|z-z_\varepsilon\|_2< \frac{\varepsilon}{4\|y_\varepsilon\|}. \end{split}
\end{equation} Using \eqref{est1} together with the triangle inequality for all $\lambda \in \Lambda$ we have 
\begin{equation}\label{ineq3'}\begin{split}
    \| E_{B\rtimes_\alpha G} (yx_\lambda z)\|_2 & \leqslant \frac{\varepsilon}{2}+ \|E_{B\rtimes_\alpha G} (y_\varepsilon x_\lambda z_\varepsilon)\|_2\\
    &\leqslant  \frac{\varepsilon}{2}+ \sum_{g\in E_\varepsilon, h\in F_\varepsilon }\|E_{B\rtimes_\alpha G} (y^\varepsilon_g u_g  b_\lambda u_{g_\lambda} z^\varepsilon_h u_{h})\|_2 \\
    & =  \frac{\varepsilon}2+ \sum_{g\in E_\varepsilon, h\in F_\varepsilon }\|E_{B} (\alpha_{g^{-1}}(y^\varepsilon_g)  b_\lambda \alpha_{g_\lambda} (z^\varepsilon_h ))\|_2.  \end{split}
\end{equation}Since $y_g^\varepsilon \in M\ominus N$ then $\alpha_{g^{-1}}(y^\varepsilon_g )\in M\ominus N$. Using \eqref{rwm3}, for every $g\in E_\varepsilon, h\in F_\varepsilon$ one can find $\lambda^\varepsilon_{g,h} \in \Lambda$ such that  $\|E_{B} (\alpha_{g^{-1}}(y^\varepsilon_g)  b_\lambda \alpha_{g_\lambda} (z^\varepsilon_h ))\|_2 \leqslant\frac{\varepsilon}{2(|E_\varepsilon|+|F_\varepsilon|)}$, for all $\lambda\succeq \lambda_{g,h}^\varepsilon$; here ''$\succeq$'' denotes the preorder on $\Lambda$. As $(\Lambda, \succeq)$ is directed and $E_\varepsilon, F_\varepsilon$ are finite one can find $\lambda_\varepsilon\in \Lambda$ such that  $\lambda_\varepsilon\succeq \lambda^\varepsilon_{g,h}$ for all $g\in E_\varepsilon,h\in F_\varepsilon$. Altogether, these combined with \eqref{ineq3'} yield that $\| E_{B\rtimes_\alpha G} (yx_\lambda z)\|_2\leqslant \varepsilon$ for all $\lambda\succeq \lambda_\varepsilon$, thereby proving 
\eqref{rwm4}.
\vskip 0.05in
To see the converse, assume $B \rtimes_\alpha G \subseteq N \rtimes_\alpha G\subseteq M\rtimes_\alpha G$ satisfy the relative WAHP. Since $\mathscr G=\{ b u_g \,:\, b\in \mathcal U(B), g\in G  \}\subseteq \mathcal U(B\rtimes_\alpha G)$ is a subgroup with $\mathscr G''=B\rtimes_\alpha G$, using Theorem \ref{srwm}, one can find a net $x_\lambda:=b_\lambda u_{g_\lambda}\in \mathscr G$ so that for all $x,y\in M\rtimes_\alpha G$, 
\begin{equation}\label{rwahp2}\|E_{B\rtimes_\alpha G }(x x_\lambda y) - E_{B\rtimes_\alpha G}(E_{N\rtimes_\alpha G}(x) x_\lambda E_{N\rtimes_\alpha G}(y))\|_2 \rightarrow 0.\end{equation}
Fix $x,y\in M\ominus N$ and notice $E_{N\rtimes_\alpha G}(x)= E_{N\rtimes_\alpha G}(y)=0$. Basic computations combined with these relations and also \eqref{rwahp2} show that 
\begin{equation*}\begin{split}\|E_B(x b_\lambda \alpha_{g_\lambda}(y))\|_2 &= \|E_{B\rtimes_\alpha G}(x b_\lambda \alpha_{g_\lambda}(y)) \|_2=\|E_{B\rtimes_\alpha G}(x x_\lambda y u_{g_\lambda^{-1}})\|_2\\&=\|E_{B\rtimes_\alpha G}(x x_\lambda y ) u_{g_\lambda^{-1}} \|_2  =\|E_{B\rtimes_\alpha G}(x x_\lambda y )\|_2\\ &=\|E_{B\rtimes_\alpha G}(x x_\lambda y ) -  E_{B\rtimes_\alpha G}(E_{N\rtimes_\alpha G}(x) x_\lambda E_{N\rtimes_\alpha G}(y)) \|_2 \rightarrow 0,\end{split}
 \end{equation*}which yields that $\mathfrak M$ is weak mixing relative to $B$.
\end{proof}

\noindent Over the next three subsections we 
 present several constructions  of inclusions of II$_1$ factors with infinite Jones index that are singular, and fail the WAHP; see Theorem \ref{example1}, Corollary \ref{profinite normalizer2}, Theorem \ref{tensoraction profinite}. The last two depict even more extreme situations, namely, infinite Jones index inclusions $N\subseteq M$ of II$_1$ factors which are simultaneously singular and quasiregular, i.e.,\ $\mathcal {QN}(N\subseteq M)''=M$.

\subsection{An action on the hyperfinite II$_1$ factor} We now construct our first example. Denote by $\mathbb M_2$ the $2\times 2$ matrices with complex entries.  Define unitary matrices by
\begin{align}
&v_1=\left(\begin{array}{cc}1&\phantom{-}0\\0&-1\end{array}\right),\ \ \ 
v_2=\left(\begin{array}{cc}0&1\\1&0\end{array}\right),\ \ \  \notag\\
&v_3=\left(\begin{array}{cc}\phantom{-}0&1\\-1&0\end{array}\right),\ \ \ 
v_4=\left(\begin{array}{cc}\phantom{-}1/2&\sqrt{3}/2\\-i\sqrt{3}/2&i/2\end{array}\right)
.\label{ex.1}
\end{align}
Note that $\{v_1,v_2,v_3\}$ form a basis for the subspace of matrices of zero trace. These three unitaries satisfy the following easily verified relations:
\begin{align}
&v_1v_2=v_3,\ \ v_1v_3=v_2,\ \ v_2v_1=-v_3,\ \ v_2v_3=-v_1,\notag\\
&v_3v_1=-v_2,\ \ v_3v_2=v_1,\ \ v_1^2=v_2^2=1,\ \ v_3^2=-1.\label{ex.2}
\end{align}

\noindent Let $G$  be the free group $\mathbb{F}_4$ with generators $ \{g_i:1\leqslant i \leqslant 4\}$. In defining an action $\beta$ of $G$ on $\mathbb{M}_2$, we need only specify the values of $\{\beta_{g_i}:1\leqslant i\leqslant 4\}$, so we set
\begin{equation}\label{ex.3}
\beta_{g_i}={\mathrm{Ad}}\,(v_i),\ \ \ \ \ 1\leqslant i \leqslant 4.
\end{equation}

\noindent We regard the hyperfinite II$_1$ factor $R$ as the infinite tensor product of copies of $\mathbb{M}_2$ indexed by the elements of $\mathbb{F}_4$, and we let $\gamma$ denote the Bernoulli action of $\mathbb{F}_4$ on $R$. We define $M$ to be $\mathbb{M}_2\otimes R$ with an action of $\mathbb{F}_4$ given by $\alpha=\beta\otimes \gamma$. We note that $\alpha$ is an outer action of $G$ on $M$ since $\gamma$ is an outer action of $G$ on $R$ \cite[Corollary 1.12]{Ka}.

\begin{thm}\label{example1}
With the above notation, $L(G)$ is singular in $M\rtimes_\alpha G$, while the quasinormalizers of $L(G)$ generate a von Neumann algebra which is strictly larger than $L(G)$.
\end{thm}

\begin{proof} Any $x \in \mathbb{M}_2$ has Orb$(x) \subseteq \mathbb{M}_2$.  In particular, for such an $x$ and any $h \in G$ we have $h x = \alpha_h(x) h \in \sum_i v_i L(G)$, where $v_i$, $1 \leqslant i \leqslant 4$,  are the unitaries from equation \eqref{ex.1}. In order to conclude from this that $L(G) x \subseteq \sum_i v_i L(G),$ it suffices to show that the module $\sum_i v_i L(G)$ is $w^*$-closed.  This follows from a general result proved subsequently and independently in Lemma \ref{weakly closed modules}. Thus, $L(G)x \subseteq \sum_i v_i L(G)$ and a similar argument shows $x L(G) \subseteq \sum_i L(G) v_i$.  It follows that $L(G)$ admits nontrivial quasinormalizers, and moreover, that the algebra of quasinormalizers in the crossed product contains the subalgebra $\mathbb{M}_2\rtimes_\beta  G$ of $M \rtimes_\alpha G$.  Thus, it remains to show that $L(G)$ is singular in $M\rtimes_\alpha G$, which we break into several steps.
\medskip

\noindent {\bf{Step 1.}}  {\emph{The only fixed points in $M$ of $\alpha(G)$ are in $\mathbb{C}1$.}}

\medskip

\noindent The fact that the Bernoulli shift is mixing implies that the only candidates for fixed points of $\alpha(G)$ must have the form $x\otimes 1$ for $x\in \mathbb{M}_2$.
If $x\in \mathbb{M}_2$ is a fixed point for $\beta(G)$, then
\begin{equation}\label{ex.4}
v_nxv_n^*=x, \ \ \ \ \ 1\leqslant n\leqslant 4,
\end{equation}
so $x$ commutes with $\{1,v_1,v_2, v_3\}$, showing that it is central in $\mathbb{M}_2$. In particular, this shows that the action of $G$ on $M$ is ergodic.
\medskip

\noindent {\bf{Step 2.}} {$L(G)'\cap (M\rtimes_\alpha G)=\mathbb{C}1.$}

\medskip

\noindent Let $x\in L(G)'\cap (M\rtimes_\alpha G)$ have Fourier series
$x=\sum_{g\in G}x_gu_g$. Then, for $h\in G$,
\begin{equation}\label{ex.5}
\sum_{g\in G}u_hx_gu_g=\sum_{g\in G}\alpha_h(x_g)u_{hg}=\sum_{k\in G}\alpha_h(x_{h^{-1}k})u_k,
\end{equation}
while
\begin{equation}\label{ex.6}
\sum_{g\in G}x_gu_gu_h=\sum_{k\in G}x_{kh^{-1}}u_k.
\end{equation}
Thus 
\begin{equation}\label{ex.7}
\alpha_h(x_{h^{-1}k})=x_{kh^{-1}},\ \ \ \ h, k\in G,
\end{equation}
and so, after making the substitution $r=kh^{-1}$,
\begin{equation}\label{ex.8}
\alpha_h(x_{h^{-1}rh})=x_r, \ \ \ \  h,r\in G.
\end{equation}
If $x_r\ne 0$ for some $r\ne e$, then $r$ has infinitely many distinct conjugates for which $\|x_{h^{-1}rh}\|_2=\|x_r\|_2\ne 0$, an impossibility. Thus $x_r=0$ for $r\ne e$, so $x$ reduces to being  $x_e\in M$, and commutation with $u_g$ for $g\in G$ shows that $x_e$ is a fixed point for $\alpha(G)$. Step 2 now follows from Step 1.
\vskip 0.04in
\noindent As a consequence of Step 2, we note that $M\rtimes_\alpha G$ is a factor.

\medskip

\noindent {\bf{Step 3.}}  For $1\leqslant i\leqslant 3$, {{$W^*(u_{g_i},u_{g_{i+1}}^2)'\cap (R\rtimes_\gamma G)=\mathbb{C}1$.}} (Subscripts are  mod 3).

\medskip

\noindent The  proofs of these equalities are all identical, so we consider only the initial case $i=1$.
First consider an element $x\in R\rtimes_\gamma G$ that commutes with $W^*(u_{g_1})$, and write its Fourier series as $\sum_{g\in G} y_gu_g$ with $y_g\in R$. Commuting with $u_{g_1}^n$ for $n\in \mathbb{Z}$ entails
\begin{equation}\label{ex.25}
\sum_{g\in G}\gamma_{g_1^n}(y_g)u_{g_1^ng}=\sum_{g\in G}y_gu_{gg_1^n},\ \ \ \ n\in \mathbb{Z},
\end{equation}
so changing variables ($k=g_1^ng$ for the first sum, $k=gg_1^n$ for the second) leads to 
\begin{equation}\label{ex.26}
\sum_{k\in G}\gamma_{g_1^n}(y_{g_1^{-n}k})u_k=\sum_{k\in G}y_{kg_1^{-n}}u_k,\ \ \ \ n\in \mathbb{Z}.
\end{equation}
From \eqref{ex.26} we obtain
\begin{equation}\label{ex.27}
\gamma_{g_1^n}(y_{g_1^{-n}k})=y_{kg_1^{-n}},\ \ \ \ n\in \mathbb{Z},\ k\in G,
\end{equation}
and the further change of variables $s=kg_1^{-n}$ allows us to rewrite \eqref{ex.27} as
\begin{equation}\label{ex.28}
\gamma_{g_1^n}(y_{g_1^{-n}sg_1^n})=y_{s},\ \ \ \ n\in \mathbb{Z},\ s\in G.
\end{equation}
Any $s\notin\langle g_1\rangle$ has infinitely many distinct conjugates by powers of $g_1$. If $y_s\ne 0$ for such an $s$, then \eqref{ex.28} gives infinitely many coefficients in the Fourier series with equal nonzero 2-norms, an impossibility. We conclude that $y_g=0$ for $g\notin\langle g_1\rangle$. If we further assume that $x$ commutes with $W^*(u_{g_2}^2)$, then we see that $y_g=0$ for $g\ne e$ and that $y_e$ is a fixed point for $\gamma_{g_1}$. Since $\gamma$ is the Bernoulli action, this ensures that $y_e$ is a scalar, and so also is $x$.

\medskip

\noindent {\bf{Step 4.}}  For $1\leqslant i \leqslant 3$, {{$W^*(u_{g_i},u_{g_{i+1}}^2)'\cap (M\rtimes_\alpha G)=W^*(v_i)$.}} (Subscripts are mod 3).

\medskip

\noindent If $x\in M\rtimes_\alpha G$ commutes with $W^*(u_{g_1},u_{g_2}^2)$ and has Fourier series $\sum_{g\in G} y_gu_g$ with $y_g\in M$, then we can repeat the argument of Step 3 to conclude that $y_g=0$ for $g\ne e$ and $y_e$ is a fixed point for $\beta_{g_1}$ and $\beta_{g_2}^2$. These fixed points are precisely the matrices in $W^*(v_1)$. This proves the first equality, and the argument for the other two cases is identical.

\medskip

\noindent {\bf{Step 5.}} For $1\leqslant i\leqslant 3$, $W^*(v_4)'\cap W^*(v_i)=\mathbb{C}1$.

\medskip

\noindent General matrices  $x_i\in W^*(v_i)$, $1\leqslant i\leqslant 3$, respectively have the form
\begin{equation}\label{general}
x_1=\left(\begin{array}{cc}\lambda&0\\0&\mu\end{array}\right),\ \ 
x_2=\left(\begin{array}{cc}\lambda&\mu\\\mu&\lambda\end{array}\right),\ \ 
\mathrm{and} \ \ 
x_3=\left(\begin{array}{cc}\phantom{-}\lambda&\mu\\-\mu&\lambda\end{array}\right).
\end{equation}
The requirement for $x_i$ to commute with $v_4$ results in 
\begin{equation}\label{general1}
\left(\begin{array}{cc}\lambda&0\\0&\mu\end{array}\right)
\left(\begin{array}{cc}\phantom{-}1/2&\sqrt{3}/2\\-i\sqrt{3}/2&i/2\end{array}\right)
=
\left(\begin{array}{cc}\phantom{-}1/2&\sqrt{3}/2\\-i\sqrt{3}/2&i/2\end{array}\right)
 \left(\begin{array}{cc}\lambda&0\\0&\mu\end{array}\right),\ \ \ (i=1),
 \end{equation}
 
 \begin{equation}\label{general2}
\left(\begin{array}{cc}\lambda&\mu\\\mu&\lambda\end{array}\right)
\left(\begin{array}{cc}\phantom{-}1/2&\sqrt{3}/2\\-i\sqrt{3}/2&i/2\end{array}\right)
=
\left(\begin{array}{cc}\phantom{-}1/2&\sqrt{3}/2\\-i\sqrt{3}/2&i/2\end{array}\right)
 \left(\begin{array}{cc}\lambda&\mu\\\mu&\lambda\end{array}\right),\ \ \ (i=2),
 \end{equation}
 and
 \begin{equation}\label{general3}
\left(\begin{array}{cc}\phantom{-}\lambda&\mu\\-\mu&\lambda\end{array}\right)
\left(\begin{array}{cc}\phantom{-}1/2&\sqrt{3}/2\\-i\sqrt{3}/2&i/2\end{array}\right)
=
\left(\begin{array}{cc}\phantom{-}1/2&\sqrt{3}/2\\-i\sqrt{3}/2&i/2\end{array}\right)
\left(\begin{array}{cc}\phantom{-}\lambda&\mu\\-\mu&\lambda\end{array}\right), \ \ (i=3).
 \end{equation}
Comparison of the (1,2) matrix entries leads easily to the conclusion that $\lambda=\mu$ in \eqref{general1} and to $\mu =0$ in \eqref{general2} and \eqref{general3}. Thus $x_i\in \mathbb{C}1$ in all cases.

\medskip

\noindent {\bf{Step 6.}} {\emph{$L(G)$ is singular in $M\rtimes_\alpha G.$}}

\medskip

\noindent Let $E:M\rtimes_\alpha G\to L(G)$ be the trace-preserving conditional expectation, and let $u\in M\rtimes_\alpha G$ be a unitary that normalizes $L(G)$.

\medskip

\noindent {\bf{Case 1:}} $E(u)\ne 0.$

\medskip

\noindent Let $y=E(u)\ne 0$, and write $\phi$ for the automorphism Ad$\,(u)$ of $L(G)$. Then
\begin{equation}\label{ex.11}
ux=\phi(x)u,\ \ \ \ x\in L(G).
\end{equation}
Apply $E$ to \eqref{ex.11} to obtain
\begin{equation}\label{ex.12}
yx=\phi(x)y,\ \ \ x\in L(G).
\end{equation}
A standard argument then shows that $y^*y$ is central in $L(G)$ so is $\lambda 1$ for some $\lambda >0$. Thus $v:=y/\sqrt{\lambda}\in L(G)$ is a unitary that implements $\phi$. It follows that $u^*v\in L(G)'\cap(M\rtimes_\alpha G)=\mathbb{C}1$ by Step 2. Thus $u\in L(G)$.

\medskip

\noindent {\bf{Case 2:}} $E(u)=0$. (We will show that this case cannot occur.)

\medskip

\noindent Let $v_0=1\in \mathbb{M}_2$, so that $\{v_i:0\leqslant i\leqslant 3\}$ is a basis for $\mathbb{M}_2$. Then $u^*\in M\rtimes_\alpha G$ can be expressed as $\sum_{i=0}^3 v_i f_i$ where $f_0,\ldots,f_3\in R\rtimes_\gamma G$.
Since 
\begin{equation}\label{ex.13}
E(u^*)=\sum_{i=0}^3 {\mathrm{tr}}(v_i)E(f_i)=E(f_0),
\end{equation}
we see that $E(f_0)=0$.
 As above, we write $\phi={\mathrm{Ad}}\,(u)\in {\mathrm{Aut}}(L(G))$, so that
\begin{equation}\label{ex.14}
u_gu^*=u^*\phi(u_g),\ \ \ \ g\in G,
\end{equation}
which is equivalent to 
\begin{equation}\label{ex.15}
u_g(v_0f_0+v_1f_1+v_2f_2+v_3f_3)=(v_0f_0+v_1f_1+v_2f_2+v_3f_3)\phi(u_g),\ \ \ \ g\in G.
\end{equation}

\noindent There are two possibilities:

\medskip

\noindent {\bf{Case 2a:}} $f_0\ne 0$.

\medskip

\noindent The trace-preserving conditional expectation $E_{R\rtimes_\gamma G}:M\rtimes_\alpha G\to R\rtimes_\gamma G$ is given on generators by $(x\otimes r)u_g\mapsto {\mathrm{tr}}(x)ru_g$ for $x\in \mathbb{M}_2$, $r\in R$, and $g\in G$. Note that, for $g\in G$ and $i\in \{1,2,3\}$, $E_{R\rtimes_\gamma G}(u_gv_if_i)=E_{R\rtimes_\gamma G}(\beta_g(v_i)u_gf_i)=0$ since $\mathrm{tr}(\beta_g(v_i))=0$. Applying this expectation to \eqref{ex.15}, we see that 
\begin{equation}\label{ex.15a}
u_gf_0=f_0\phi(u_g),\ \ \ \ g\in G,
\end{equation}
from which it follows that $f_0f_0^*$ commutes with $L(G)$.  From Step 2, $f_0f_0^*$ is a nonzero scalar so, after scaling, $f_0\in R\rtimes_\gamma G$ is a unitary that normalizes $L(G)$. 
The Bernoulli action $\gamma$ on $R$ is mixing, and so $f_0\in L(G)$ since $L(G)$ is singular in this crossed product by Lemma \ref{lemma:WeaklyMixingSingular}. This contradicts $E(f_0)=0$, so this case cannot occur.

\medskip

\noindent {\bf{Case 2b:}} $f_0=0$.

\medskip

\noindent In this case, \eqref{ex.15} reduces to 
\begin{equation}\label{ex.15b}
u_g(v_1f_1+v_2f_2+v_3f_3)=(v_1f_1+v_2f_2+v_3f_3)\phi(u_g),\ \ \ \ g\in G.
\end{equation}
Now 
\begin{equation}\label{ex.16}
u_{g_j}v_i=\beta_{g_j}(v_i)u_{g_j}=v_jv_iv_j^*u_{g_j},\ \ \ 1\leqslant i,j\leqslant 3.
\end{equation}
Using \eqref{ex.2}, we see that
\begin{equation}\label{ex.17}
u_{g_1}v_1=v_1u_{g_1},\ \ u_{g_1}v_2=-v_2u_{g_1},\ \ u_{g_1}v_3=-v_3u_{g_1}.
\end{equation}
Thus, from \eqref{ex.15b},
\begin{equation}\label{ex.18}
v_1u_{g_1}f_1-v_2u_{g_1}f_2-v_3u_{g_1}f_3=
v_1f_1\phi(u_{g_1})+v_2f_2\phi(u_{g_1})+v_3f_3\phi(u_{g_1}).
\end{equation}
If we successively multiply this equation on the left by $v_1$, $v_2$, and $v_3$, and apply $E_{R\rtimes_\gamma G}$ each time, the results are
\begin{equation}\label{ex.19}
u_{g_1}f_1=f_1\phi(u_{g_1}),\ \ \  u_{g_1}f_2=-f_2\phi(u_{g_1}),\ \ \ u_{g_1}f_3=-f_3\phi(u_{g_1}).
\end{equation}

Repeating this argument for the group elements $g_2$ and $g_3$ leads to similar sets of equations:
\begin{equation}\label{ex.20}
u_{g_2}f_1=-f_1\phi(u_{g_2}),\ \ \  u_{g_2}f_2=f_2\phi(u_{g_2}),\ \ \ u_{g_2}f_3=-f_3\phi(u_{g_2})
\end{equation}
and
\begin{equation}\label{ex.21}
u_{g_3}f_1=-f_1\phi(u_{g_3}),\ \ \  u_{g_3}f_2=-f_2\phi(u_{g_3}),\ \ \ u_{g_3}f_3=f_3\phi(u_{g_3}).
\end{equation}
Then there exists $i\in\{1,2,3\}$ so that $f_i \ne 0$.  From the  equalities of (\ref{ex.19}-\ref{ex.21}), we see that
\begin{equation}\label{ex.22}
u_{g_i}f_i=f_i\phi(u_{g_i}),\ \ \ \ u_{g_{i+1}}^2f_i=f_i\phi(u_{g_{i+1}}^2),
\end{equation}
and these equations can be rearranged  to give 
\begin{equation}\label{ex.23}
u_{g_i}^*f_i=f_i\phi(u_{g_i}^*),\ \ \ \ {u_{g_{i+1}}^{*2}}f_i=f_i\phi({u_{g_{i+1}}^{*2}}).
\end{equation}
It follows from \eqref{ex.22} and \eqref{ex.23} that $f_if_i^*$ commutes with all elements of the self-adjoint subspaces 
span$\{u_{g_i},u_{g_i}^*\}$ and span$\{{u_{g_{i+1}}^2},{u_{g_{i+1}}^{*2}}\}$ and thus lies in the relative commutants of $W^*(u_{g_i})$
and $W^*(u_{g_{i+1}}^2)$ in $R\rtimes_\gamma G$. By Step 3,  $f_if_i^*$ is a nonzero positive scalar so, after scaling, there is a unitary $w_i\in R\rtimes_\gamma G$ with $f_i$  a multiple of $w_i$ and
\begin{equation}\label{ex.24}
u_{g_i}^{\pm 1}w_i=w_i\phi(u_{g_i}^{\pm 1})=w_iuu_{g_i}^{\pm 1}u^* .
\end{equation}
Thus $w_iu\in W^*(u_{g_i})'\cap(M\rtimes_\alpha G)$, and similarly $w_iu$ commutes with $W^*({u_{g_{i+1}}^2})$. From Step 4, there exists a unitary $x_i\in W^*(v_i)$ such that $w_iu=x_i$, so $u=w_i^*x_i$. For each $g\in G$, $w_i^*x_iu_gx_i^*w_i=uu_gu^*\in L(G)$, so $x_iu_gx_i^*\in R\rtimes_\gamma G$. Multiply on the right by $u_g^*$ to obtain $x_i\beta_g(x_i^*)\in R\rtimes_\gamma G$, for all $g\in G$, and this implies that $x_i\beta_g(x_i^*)\in \mathbb C1$. In particular, there is a scalar $\eta$ so that $x_i\beta_{g_4}(x_i^*)=\eta 1$, which becomes 
$x_iv_4=\eta v_4x_i$. Taking the determinant shows that $\eta=1$, and it now follows from Step 5 that $x_i\in \mathbb{C}1$.  Thus $u\in R\rtimes_\gamma G$ so, as above, $u\in L(G)$ by the singularity of this subalgebra of $R\rtimes_\gamma G$. This contradicts $E(u)=0$, so this case cannot occur. We have now verified the singularity of $L(G)$ in $M\rtimes_\alpha G$.
\end{proof}
\noindent We note further that the existence of ``nontrivial" quasinormalizers of $L(G)$ in $M \rtimes_\alpha G$ precludes the WAHP.  Therefore, we have established the existence of an inclusion of II$_1$ factors which is singular, fails the WAHP, and has infinite Jones index.

\subsection{Profinite actions of i.c.c.\ groups}  %Here we exhibit another class of groups whose actions give rise to crossed product inclusions for which the normalizer and quasinormalizer differ sharply.  These are known as {\emph W$^*$-superrigid groups}, and they have been explored widely over the past decade in the operator algebras literature.  We begin with a definition.  As we have done above, assume all groups to be discrete. For such a group $G$, denote by $u_g,$ $ g \in G$ the canonical group unitaries in $L(G)$.

In this subsection we exhibit a fairly large and natural class of crossed product von Neumann algebras, $L^{\infty}(X)\rtimes_\alpha G$ associated with  p.m.p. actions of \textit{countable} i.c.c. groups on standard probability spaces $G\curvearrowright^\alpha (X,\mu)$ for which we are able to describe in detail all normalizing unitaries in  $\mathcal N(L(G)\subseteq L^\infty(X)\rtimes_\alpha G)$; see Corollaries \ref{profinite normalizer1} and \ref{compact quasinormalizer}. These results can be regarded as non-commutative counterparts of Packer's prior results, \cite[Theorem 2.3]{Pac0}.
\vskip 0.06in
\noindent Using our description of normalizers we then highlight additional examples of von Neumann algebra inclusions $P \subseteq M$ of infinite Jones index for which the normalizer and the quasi-normalizer algebras of $P$ differ very sharply. For instance, Corollary \ref{profinite normalizer2} and the remarks succeeding it provide natural examples when $P$ is a subfactor that is simultaneously singular and quasiregular.

If $N$ has separable predual and $G \curvearrowright^{\alpha} N$ is an ergodic, compact trace-preserving action, using [BaCaMu2, Theorem 4.7] (see also Lemmas 5.4 and 5.5 below) we can always find a sequence $(N_{k})$ of finite-dimensional $G$--invariant subspaces of $N$ such that $\cup_{k} N_{k}$ is $\|\cdot\|_2$-dense in $N$. We next leverage such a sequence to obtain a general structural result from which subsequent examples will be obtained.  In the sequel, for a positive integer $k$ and an element $x$ of a von Neumann algebra $M$, denote by ${\rm diag}(x)$ the $k \times k$ diagonal matrix $x \otimes I \in \mathbb{M}_k(M)$.

 \begin{thm}\label{compact quasinormalizer} Let $G$ be an i.c.c. group and $\alpha$ an ergodic, compact, trace-preserving action of $G$ on a tracial von Neumann algebra $(N,\tau)$ with separable predual. Denote by $M = N\rtimes_\alpha G$ the associated crossed product von Neumann algebra.  
Fix an increasing  sequence $(N_k)$ of finite-dimensional $G$-invariant subspaces of $N$ such that $\cup_k N_k$ is $\|\cdot\|_2$-dense in $N$.

\begin{itemize}
\item[(i)] For any $w\in \N(L(G) \subseteq M)$ one can find $k \in \mathbb N$, an orthonormal basis $\{\xi_1,\ldots, \xi_{n_k}\}\subseteq N_k$ and elements $w_1, \ldots,  w_{n_k}\in L(G)$ such that $w = \sum_i \xi_i w_i$. 

\item[(ii)] Let $\alpha^k: G \rightarrow \mathcal U(N_k)$ be the unitary representation induced by $\alpha$ and for every $g\in G$ consider the matrix $M(g)=(\langle \alpha^k_g(\xi_j), \xi_i\rangle)_{1\leqslant i,j\leqslant n_k}\in \mathbb M_{n_k}(\mathbb C)$.  If we let $X\in \mathbb M_{n_k}(M)$ be the matrix whose entries satisfy $x_{i,j}= w_i$ if $j=1$ and $x_{i,j}= 0$ if $j> 1$ then the following holds:  $$\diag (u_g) M(g) X= X {\rm diag}({\rm Ad}(w^*)(u_g))\text{ for all }g\in G.$$
\end{itemize}
\end{thm}

\begin{proof} Fix an orthonormal basis $\{\xi_1,\ldots, \xi_{n_k}\}\subseteq N_k$ and notice that for all $g\in G$ we have 
    \begin{equation}\label{comprep}
        \alpha^k_g(\xi_i)= \sum^{n_k}_{j=1} \langle\alpha^k_g(\xi_i), \xi_j \rangle\xi_j. 
    \end{equation}

\noindent Using relation \eqref{comprep} and the same argument as in the beginning of the proof of Theorem \ref{example1} one can show that $L(G)\xi_i L(G)\subseteq \sum_j \xi_j L(G)$, $\sum_j L(G)\xi_j$, for all $1\leqslant i\leqslant n_k$. Thus $\xi_i\in \mathcal{QN}( L(G)\subseteq M)$ for all $1\leqslant i\leqslant n_k$. 
\vskip 0.04in
\noindent Now denote by $\mathcal P_k$ the orthogonal projection onto $\overline{\sum^{n_k}_{j=1}L(G)\xi_j L(G)}^{\|\cdot\|_2}$. Since $G$ is i.c.c. and the action $\alpha$ is ergodic we have that $L(G)'\cap M=\mathbb C 1$. Thus, since the range of $\mathcal P_k$ is an $L(G)$-$L(G)$ bimodule, for every $x,y\in L(G)$ and $\eta\in M$ we have \begin{equation}\label{bimodularity}
\mathcal P_k (x\eta y)= x\mathcal P_k(\eta)y.
\end{equation}
Moreover, as $\cup_k N_k$ is $\|\cdot\|_2$-dense in $N$, for every $x\in M$ we have\begin{equation}
    \lim_k\|\mathcal P_k(x)-x\|_2=0.
\end{equation}

\noindent Let $w\in \N(L(G) \subseteq M) $.  Then,  the $\ast$-automorphism $\theta_w= {\rm Ad}(w):L(G)\rightarrow L(G)$ satisfies  \begin{equation}\label{intertwining7}
   \theta_w(x)w=wx, \text{ for all }x\in L(G). 
\end{equation}
Thus there exists a smallest $k\in \mathbb N$ such that $\mathcal P_{k}(w)\neq 0$. Applying the orthogonal projection $\mathcal P_{k}$ to relation \eqref{intertwining7} and using the bimodularity condition \eqref{bimodularity} we get $ \theta_w(x)\mathcal P_{k}(w)=\mathcal P_{k}(w)x$ for all $x\in L(G)$. Using this in combination with \eqref{intertwining7} for all $x\in L(G)$ we have  \[
w^\ast \mathcal P_{k}(w) x= w^* \theta_w(x)\mathcal P_{k}(w)= xw^\ast \mathcal P_{k}(w),
\]so $w^\ast \mathcal P_{k}(w)\in L(G)'\cap M$. However as observed before,  $L(G)'\cap M=\mathbb C1$. Thus there is $\lambda\in \mathbb C \setminus \{0\}$ such that  $w^\ast \mathcal P_{k}(w)=\lambda 1$ and hence $\mathcal P_{k}(w)=\lambda w$. In particular, $w\in \sum^{n_k}_{i=1} \xi_i L(G) $ and one can find $w_1, \ldots,  w_{n_k}\in L(G)$ such that $w = \sum_i \xi_i w_i$. Using this, the fact that $u_{g}$ implements $\alpha_{g}^{k}$ on $N_{k}$, and the relations \eqref{intertwining7} and \eqref{comprep} we get that for every $g\in G$ we have 
\begin{align*}
\sum_i \xi_i u_g (\sum_j \langle \alpha_g^k(\xi_j), \xi_i\rangle w_j)&=\sum_{i,j}\langle \alpha_g^k(\xi_j), \xi_i\rangle\xi_i u_gw_{j}\\
&=\sum_{j}\alpha_{g}^{k}(\xi_{j})u_{g}w_{j}=\sum_{j}u_{g}\xi_{j}u_{g}^{*}u_{g}w_{j} \\
&=u_{g}w=w\theta_{w^{*}}(u_{g})=\sum_i \xi_i w_i \theta_{w^*}(u_g).
\end{align*}

%\[ \sum_i \xi_i w_i \theta_{w^*}(u_g) = \sum_i u_g \xi_i w_i= \sum_i \xi_i u_g (\sum_j \langle \alpha_g^k(\xi_j), \xi_i\rangle w_j). \]

\noindent Hence, using the $L(G)$-orthonormal basis property  of the $\xi_i$'s,  for every $i$ we have that \begin{equation}\label{matrixrel}w_i \theta_{w^*}(u_g)= u_g  \sum_j \langle \alpha_g^k(\xi_j), \xi_i\rangle w_j.\end{equation}
%\textcolor{red}{OMIT: These imply, furthermore, that for all $i,k$ we have  \begin{equation}\label{matrixrel2}
%\sum_l w_i w_l^* u_g \langle \alpha_g^k(\xi_l), \xi_k\rangle  = \sum_l u_g w_l w_k^* \langle \alpha_g^k(\xi_l), \xi_i\rangle.\end{equation}}
%\vskip 0.03in
\noindent Now consider the unitary matrix $M(g)=(\langle \alpha^k_g(\xi_j), \xi_i\rangle)_{1\leqslant i,j\leqslant n_k}\in \mathbb M_{n_k}(\mathbb C)$, and let $X\in \mathbb M_{n_k}(M)$ be the matrix whose entries satisfy $x_{i,j}= w_i$ if $j=1$ and $x_{i,j}= 0$ if $j> 1$. Then, relations \eqref{matrixrel} are equivalent to  \begin{equation}\label{matrixintertwining}\diag (u_g) M(g) X= X {\rm diag}({\rm Ad}(w^*)(u_g))\text{ for all }g\in G.\end{equation}    

\vskip 0.06in
\noindent We note in passing that the prior relation is in fact equivalent to \eqref{intertwining7}. 
\end{proof}

\noindent Specializing Theorem \ref{compact quasinormalizer} to the cases of group measure space constructions associated with profinite actions yields an even more concrete description of these normalizers. Before introducing the result, we briefly recall the construction of profinite actions.
\vskip 0.05in
\noindent  Recall that a discrete group $G$ is said to be \emph{residually finite} if there is a sequence $G_1 \supseteq G_2 \supseteq G_3 \supseteq \cdots$ of finite-index subgroups of $G$ with intersection $\sett{e}$.  In this situation, for each $k$, $G$ acts by left translation on the (finite) set $G/G_k$ of left cosets. When $G/G_k$ is equipped with counting measure $\mu_k$, we obtain an ergodic, p.m.p. action $\alpha_k$ of $G$ on $(G/G_k, \mu_k)$.  Moreover, for each $k$ there is a quotient map $q_k: G/G_{k+1} \rightarrow G/G_k$, given by 
\[q_k(s G_{k+1}) = t G_k \quad \text{  iff  } \quad sG_{k+1} \subseteq t G_k.\]
\noindent Then the inverse limit $X=\varprojlim (G/G_k,\mu_k)$ is a probability space, and the inverse limit action $\alpha$ of $G$ on $X$ can be shown to be ergodic and measure-preserving.  An action of this form is profinite, i.e., it has the form $\alpha = \varprojlim \alpha_k$ for a sequence of measure-preserving actions of $G$ on finite probability spaces $(X_k,\mu_k)$.  Notably, any ergodic, profinite action of a discrete group arises in this manner (see \cite[Example 1.2 and Theorem 1.6]{Io}).
\vskip 0.06in
\noindent For further use we also state the following result, which is an immediate consequence of \cite[Lemma 1.4]{Io}. 

\begin{lem} \label{finite index}
    Let $G$ be an i.c.c. residually finite group, and let $G \curvearrowright^\alpha (X,\mu) = \varprojlim (X_k,\mu_k)$ be an ergodic, profinite, p.m.p. action.  Write $M = L^\infty(X) \rtimes_\alpha G$.  Then for any finite-index subgroup $H$ of $G$ there is some $n \in \Nat$ such that 
    \[ L(H)' \cap M \subseteq L^\infty(X_{n}).\]
    
\end{lem}

\begin{cor}\label{profinite normalizer1} Let $G$ be an i.c.c., residually finite  group.  Let $\alpha$ be a ergodic, profinite action of $G$ on $X = \varprojlim (X_k,\mu_k)$ and denote by $M = L^\infty(X,\mu) \rtimes_\alpha G$ the associated crossed product von Neumann algebra.  

\noindent Then for every  
 $w\in \N(L(G) \subseteq M)$ there exist $k\in \mathbb N$ and unitaries $a\in L^\infty(X_{k})$ and $v\in L(G)$  such that $w=a v$. Hence $a\in \N(L(G) \subseteq M)$, and moreover one can find  $\eta\in {\rm Char}(G)$ such that the following hold:\begin{enumerate}
 \item ${\rm Ad}(a) (u_g) = \eta(g)  u_g $,  for all $g\in G$, and 
 \item there is an atom $e\in L^\infty (X_{k})$ such that $a= \sum_g \eta(g) \alpha_g(e)$ where the sum is over a set of representatives of $G/G_k$, with $G_k = \{g\in G \,:\,(\alpha_k)_g(x)=x \text{ for all } x\in X_k\}$.  In addition, we have  $G_k \subseteq  \ker(\eta)$. \end{enumerate}
    
\end{cor}

\begin{proof} To simplify the notations let $A_k:= L^\infty(X_k)$ for all $k$ and let $A:= L^\infty(X)$. From assumptions we also have that  $M_k =A_k \rtimes_\alpha G$ where $A_k={\rm span}\{\xi_1,\ldots ,\xi_{n_k}\}=:I$ with $\xi_i \in A_k$ orthogonal projections of equal traces such that the action $\alpha_g(\xi_i)= \xi_{g \cdot i }$ for some transitive action $G\curvearrowright I$. 
\vskip 0.06in
\noindent Furthermore, $(A_k)_k$ forms an increasing tower of finite-dimensional $G$-invariant von Neumann subalgebras such that $A= \overline{\cup_k A_k}^{\rm SOT}$. Let $w\in \N(L(G) \subseteq M)$ be arbitrary. Applying Theorem \ref{compact quasinormalizer} there is $k\in \mathbb N$ such that $w\in M_k$ and hence one can find elements $w_i\in L(G)$ satisfying \begin{equation}\label{unitdecomp}
w= \sum^{n_k}_{i=1} \xi_i w_i.    \end{equation}    
For $g\in G$ consider $M(g)=(\langle \alpha_g(\xi_j), \xi_i\rangle)_{1\leqslant i,j\leqslant n_k}$.   If we let $X\in \mathbb M_{n_k}(M)$ be the matrix whose entries satisfy $x_{i,j}= w_i$ if $j=1$ and $x_{i,j}= 0$ if $j> 1$ then the following holds:  
\begin{equation}\label{matrixintertwining}\diag (u_g) M(g) X= X {\rm diag}({\rm Ad}(w^*)(u_g))\text{ for all }g\in G.\end{equation} 
\vskip 0.06in
\noindent Since $\xi_i$'s are orthogonal projections we have that $\langle\alpha_g (\xi_i),\xi_j\rangle= \langle \xi_{g\cdot i}, \xi_j\rangle = \delta_{g\cdot i,\, j}$. Thus there is a finite-index normal subgroup $G_k \subseteq G$ such that $M(g)=I_{n_k}$ for all $g\in G_k$. This  shows that $XX^*$ commutes with $diag(u_{g})$ for all $g\in G_k$. A basic calculation further implies that $XX^*\in \mathbb M_{n_k}(L(G_k)'\cap L(G))$. As $G$ is i.c.c.\ and $G_k\subseteq G$ has finite index we have that $L(G_k)'\cap L(G)=\mathbb C1$ and hence  $XX^*\in \mathbb M_{n_k}(\mathbb C 1)$. Thus one can find scalars $\lambda_{i,j}$ such that \begin{equation}\label{scalars'}w_iw_j^*= \lambda_{i,j}1\text{ for all }1\leqslant i,j\leqslant n_k.\end{equation}
Observe that $\lambda_{i,i}\geqslant 0$ for all $i\in I$. Using relation \eqref{scalars'} for $i=j$ we get $w_i = \sqrt{\lambda_{i,i}} v_i$ for some unitary $v_i \in L(G)$. Since $w\neq 0$ at least one of the scalars $\lambda_{i,i}\neq 0$ and using \eqref{scalars'} again we see that one can find scalars $\mu_i\in \mathbb C$ and a unitary $v\in L(G)$ such that $w_i = \mu_i v$ for all $i\in I$. In conclusion, we have shown that $w= \sum_i \xi_iw_i = (\sum_i \mu_i \xi_i )v= av$ for some unitary $a\in A_k$. To this end we observe that since $a$ is a unitary in the previous relations we have that $\abs{\mu_i}=1$ for all $i\in I$. This yields the first part of the statement modulo a phase factor.

Next observe that the first part shows that $a\in \N(L(G) \subseteq M)$ and $\text{Ad}(a)=\theta_a :L (G) \rightarrow L(G)$ is a $\ast$-automorphism satisfying $\theta_a (x) a=ax$ for all $x\in L(G)$. Specializing to $x=u_{g}$, multiplying $\theta_a (u_{g}) a=au_{g}$ on the right by $\xi_{j}$ and using the prior relations and basic computations we can see that for all $g\in G$ we have  \begin{equation*}\theta_a (u_g) \mu_j \xi_j=a u_g \xi_j= (\sum_i \mu_i \xi_i) \alpha_g(\xi_j) u_g= \mu_{g\cdot j} \xi_{g\cdot j} u_g. \end{equation*}
Applying the expectation $E_{L(G)}$ on the prior relation we get  $\theta_a(u_g)= \frac{\mu_{g\cdot j}}{\mu_j} u_g$ for all $g\in G$. Notice that this shows $\eta(g)=\frac{\mu_{g\cdot j}}{\mu_j}$ is independent of $j$ and it is also a multiplicative character. Thus $\theta_a(u_g)= \eta(g) u_g$ for all $g\in G$. Moreover, since $\mu_{g\cdot j}= \eta(g) \mu_j$ for all $g$ we get that $a= \sum_g \mu_j \eta(g) \xi_{g\cdot j}= \mu_j \sum_g \eta(g)\alpha_g(\xi_j)$ for some $j$. Now replace $a$ by $\overline{\mu_j}a$. This finishes the proof.\end{proof}

\vskip 0.05in
\noindent Exploiting the prior result, we obtain effective computations of the normalizing algebras in the case of profinite actions.

\begin{cor}\label{profinite normalizer2} Let $G$ be an i.c.c. residually finite group with finite abelianization.  Let $\alpha$ be a ergodic, profinite action of $G$ on $X = \varprojlim (X_k,\mu_k)$ and denote by $M = L^\infty(X,\mu) \rtimes_\alpha G$ the associated crossed product von Neumann algebra.  Then the following hold:
\begin{enumerate}
\item There exist a positive integer $k_0$ and a finite-dimensional, $G$-invariant subalgebra $A \subseteq L^\infty(X_{k_0})$ such that 
\[\N(L(G) \subseteq M)'' = A \rtimes_\alpha G.\]
\item Moreover, if the abelianization  of $G$ is trivial, then 
\[ \N(L(G)\subseteq M)''= L(G).\]
\end{enumerate}
\end{cor}
\begin{proof} Fix $w \in \N (L(G)\subseteq M)$. Using Corollary \ref{profinite normalizer1} one can find $k\in \mathbb N$ and unitaries $a\in L^\infty(X_k)$ and $v\in L(G)$ such that $w=av$. Moreover, from relation (1) in the conclusion of Corollary \ref{profinite normalizer1} there is a character $\eta \in {\rm Char}(G)$ such that ${\rm Ad}(a) u_g = \eta(g) u_g$  for all $g\in G$. This is equivalent to $a$ being an eigenvector, i.e., $\alpha_g(a)=\overline{\eta(g)}a$, $g \in G$.  Since $G/[G,G]$ is finite, $G$ has a finite-index normal subgroup $G_0$ such that $G_0 \subseteq {\rm ker}\, \omega$ for all $\omega \in {\rm Char}(G)$.  In particular, this holds for $\eta$, and we have $\alpha_h(a)=a$ for all $h \in G_0$.  Using the finite index condition and Lemma \ref{finite index}, since $a$ is fixed by $\alpha$, we have that $a \in L^\infty(X_{k_0})$ for some $k_0 \in \Nat$ \textit{which depends only on} $G_0$. Altogether, these show that $\N (L(G)\subseteq M)''\subseteq L^\infty(X_{k_0})\rtimes_\alpha G$. Then part (1) of the conclusion follows from \cite[Theorem 3.10]{ChDa} or \cite[Corollary 3.11]{JS19}. 
\vskip 0.05in
\noindent To see part (2) just notice trivial abelianization implies that $\eta =1$ and hence $\sigma_g(a)=a$ for all $g\in G$. By ergodicity, this further implies that $a=\omega 1$ with $\abs{\omega}=1$ and hence $\N (L(G)\subseteq M)\subseteq L(G)$.\end{proof}

\begin{rem}\label{Rem:Contrast}
\noindent This corollary provides new situations in which the normalizing and the quasinormalizing algebras of $L(G)$ in $N\rtimes_\alpha G$ differ sharply. For example, if in Corollary \ref{profinite normalizer2} we let $G$ be any i.c.c., residually finite, property (T) group (e.g.\ $G=PSL_n(\mathbb Z)$ with $n \geqslant 3$ or $G$ any uniform lattice in $Sp(n,1)$, with $n\geqslant 2$) and we take the action $\alpha$, then using part (1) we can find $k\in\mathbb N$ such that $\N(L(G) \subseteq M)'' \subseteq L^\infty(X_k) \rtimes_\alpha G\subsetneq M= \mathcal{QN}(L(G) \subseteq M)''$ $($see Theorem \ref{RelativeHAPTheorem}$)$,  which in the case when $|X_k| \nearrow \infty$,  implies that $\N(L(G) \subseteq M)''\subseteq \mathcal{QN}(L(G) \subseteq M)''$ has infinite Jones index. 
\end{rem}

%\textcolor{blue}{Further, part $(2)$ of the conclusion of Corollary \ref{profinite normalizer2} provides situations when $L(G)$ is singular in the crossed product despite the action being compact, which is a sharp contrast when the acting group is abelian \cite{Pac0, Ni, NeSt}.}

\vskip 0.06in
\noindent We continue by briefly presenting  an example of a residually finite i.c.c. group with trivial abelianization, that is based on several deep results of Wise \cite{Wis}, Haglund-Wise \cite{HW} and Agol \cite{A} concerning groups acting on cubical complexes; see also \cite[Theorem 5.2]{CIOS3} and \cite{CIOS1}. We are grateful to Denis Osin for suggesting this example to us.
As mentioned to us by one of the referees, examples of groups with trivial abelianization are easy to find; $SL_3(\mathbb{Z})$ is one such example. In the following theorem we want the group to have the additional property of being hyperbolic.
\begin{thm}\label{Trivialab} There exists an i.c.c. hyperbolic group that is residually finite and has trivial abelianization.
    \end{thm}
\begin{proof} Let $F=F(a,b)$ be the free group with two generators $a$ and $b$. Also let $[F,F]$ be its derived group. One can find two words $u(a,b),v(a,b)\in [F,F]$ such that the group with the following presentation 
\begin{equation*}
    G=\langle a,b \,|\, a= u(a,b), b=v(a,b)\rangle 
\end{equation*}
 is  a $C'(1/6)$-group. 
 \vskip 0.06in
\noindent  Recall in \cite{Wis} it was shown that every finitely presented $C^\prime(1/6)$ group acts geometrically (i.e., properly cocompactly) on a CAT(0) cube complex. Using the work \cite{HW}, it was proved in \cite{A} that every hyperbolic group acting geometrically on a CAT(0) cubical complex satisfies certain additional conditions, is residually finite. In conclusion, since finitely presented $C^\prime(1/6)$ groups are hyperbolic, it follows that our group $G$ is residually finite. 

\vskip 0.06in
\noindent From its presentation we can see that $G$ is torsion free. Therefore, since it is hyperbolic, using \cite[Theorem 2.35]{DGO}, it follows that $G$ is i.c.c. 
 \vskip 0.06in
 \noindent Finally, using the two relations of $G$ one can see immediately that $G= [G,G]$; thus $G$ has trivial abelianization.     \end{proof}

\noindent Notice that if $M= L^\infty(X)\rtimes_\alpha G$ ($G$ as in Theorem \ref{Trivialab})  is the crossed product of any profinite action $G \curvearrowright^\alpha X =\varprojlim(X_k, \mu_k)$ then part (2) of the prior result implies that $L(G)$ is simultaneously singular and quasiregular in $M$.   

 \vskip 0.08in

\noindent Finally, we observe that combining the prior results in this subsection with the main Theorem \ref{QuasinormalizerCompactSystemTheorem} in the next section we obtain a general upgrade of Theorem \ref{compact quasinormalizer} to the case of von Neumann algebras of all ergodic actions on probability spaces.
\begin{cor}
Let $G$ be an i.c.c. group, let $G \curvearrowright^\alpha (X, \mu)$ be a ergodic action, and denote by $M = L^\infty(X,\mu) \rtimes_{\alpha} G$ the associated crossed product von Neumann algebra.  

\noindent Let $w\in \N(L(G) \subseteq M)$. Then one can find  a finite-dimensional $G$-invariant subspace $D\subseteq L^\infty(X)$  with an orthonormal basis $\{\xi_1,\ldots, \xi_{n}\}\subseteq D$ and $w_1, \ldots,  w_{n}\in L(G)$ such that $w = \sum_i \xi_i w_i$.  
We still denote by $\alpha: G \rightarrow \mathcal U(D)$ the unitary representation induced by the action $\alpha$ and for every $g\in G$ consider the matrix $M(g)=(\langle \alpha_g(\xi_j), \xi_i\rangle)_{1\leqslant i,j\leqslant n}$. If we let $X\in \mathbb M_{n}(M)$ be the matrix whose entries satisfy $x_{i,j}= w_i$ if $j=1$ and $x_{i,j}= 0$ if $j> 1$ then the following holds:  $$\diag (u_g) M(g) X= X {\rm diag}({\rm Ad}(w^*)(u_g))\text{ for all }g\in G. $$
\end{cor}

\begin{proof} Notice that from Theorem \ref{QuasinormalizerCompactSystemTheorem} there exists a maximal compact ergodic quotient $G \curvearrowright^{\alpha_c} X_c$
of $G \curvearrowright^\alpha X$ such that $\mathcal {QN}^{(1)}(L(G)\subseteq M)''= L^\infty(X_c)\rtimes_{\alpha_c} G$. Also note that compactness of $G \curvearrowright^{\alpha_c} X_c$ implies the existence of a sequence $D_k\subseteq L^\infty(X_c)$ of finite-dimensional $G$-invariant subspaces satisfying $\cup_k D_k$ is $\|\cdot\|_2$-dense in $L^{\infty}(X_c)$ (use Theorem 6.10 \cite{BCM1}). Now since $w\in L^\infty(X_c)\rtimes_{\alpha_c} G$ the conclusion follows from Theorem \ref{compact quasinormalizer} applied to the action $G\curvearrowright ^{\alpha_c} X_c$.   
\end{proof}

\subsection{Automorphism-rigid actions of discrete groups}  In this section we concentrate on actions of countable discrete groups on von Neumann algebras with separable predual and focus on controlling normalizers. The examples in the remainder of this section arise from a rigidity property of $W^*$-dynamical systems satisfied by a variety of groups acting on tracial von Neumann algebras.  The definition is as follows.

\begin{definition} Let $G$ be a discrete group.  A trace-preserving action $\alpha$ of $G$ on a finite von Neumann algebra $(N,\tau)$ is said to be {\em automorphism-rigid} if for any automorphism $\Theta$ of $N \rtimes_\alpha G$, there exist an automorphism $\theta$ of $N$,  an automorphism $\delta$ of $G$, a character $\eta: G\rightarrow \mathbb T$  and a unitary $w \in N \rtimes_\alpha G$ such that
\[\Theta(n u_g) = \eta(g) w \theta(n) u_{\delta(g)} w^*, \quad \text{for all }n \in N, g \in G.\]
\end{definition}
\noindent We also recall the following standard definition, for the reader's convenience.
\begin{definition} A trace-preserving action $\alpha$ of a discrete group $G$ on a finite von Neumann algebra $(N,\tau)$ is called \emph{weak mixing} if there is a sequence $(g_n)_n$ of group elements such that 
$\tau(\alpha_{g_n}(x)y) \rightarrow \tau(x)\tau(y),$ for any $x,y \in N$. $($Sequences are replaced by nets when $G$ is uncountable.$)$
\end{definition}

\noindent There are many  examples known of weak mixing automorphism-rigid actions. Below we include some natural classes which emerged from Popa's deformation/rigidity theory \cite{Po06a}.
\begin{enumerate}
    \item  Bernoulli action $G \curvearrowright (\overline\otimes_G A, \tau)$ where $A$ is abelian and $G$ is an i.c.c., property (T) group or $G= G_1\times G_2$ where $G_i$ are i.c.c.\ non-amenable \cite{Po03,Po04,Po05,Po06b};
    \item The fibered versions of Rips construction $Q \curvearrowright L(N_1\times N_2)$ from \cite{CDK19,CDHK20}.
\end{enumerate}

\noindent We now have the following results.  
\begin{thm}\label{normalizer2}
    Let $G$ be an i.c.c. group and let $G \curvearrowright^\sigma A$ and $G\curvearrowright^\alpha B$ be trace-preserving actions on finite von Neumann algebras. Assume that the action $\sigma$ is automorphism-rigid. Denote by $M=(A \overline \otimes B)\rtimes_{\sigma\otimes \alpha} G$ the corresponding crossed product von Neumann algebra.  
    
    Then for every $w\in \N(A\rtimes_\sigma G \subseteq M)$ one can find unitaries $b\in A\overline\otimes B$ and $v\in A\rtimes_\sigma G$ such that $w=vb$. Thus $b\in \N(A\rtimes_\sigma G \subseteq M)$ and moreover, one can find $\zeta\in {\rm Char} (G)$ such that the $\ast$-automorphism $\theta_b={\rm Ad}(b): A\rtimes_\sigma G \rightarrow A\rtimes_\sigma G$ satisfies \begin{enumerate}
        \item $(\theta_b)_{|A} \in {\rm Aut} (A) $ and
        \item $\theta_b (u_g)=\zeta(g)u_g$, for all $g\in G$.
    \end{enumerate} 
\end{thm}

\begin{proof} To simplify the notation let $Q:=A\rtimes_\sigma G$. As  $w\in \N(Q\subseteq M)$ then the $\ast$-automorphism ${\rm Ad} (w)=\theta_w: Q \rightarrow Q$ satisfies  \begin{equation}\label{intertwining4}
   \theta_w(x)w=wx \text{ for all }x\in Q. 
\end{equation}
Since $\sigma$ is automorphism-rigid, one can find $w_0 \in \mathcal U(Q)$, $\theta \in {\rm Aut}(A)$, $\delta\in {\rm Aut}(G)$ and $\eta \in {\rm Char} (G)$ such that for all $a\in A$ and $g\in G$ we have \begin{equation}\label{autrig}
 \theta_w (au_g) = \eta(g) w_0 \theta(a) u_{\delta(g)}w_0^*.  
\end{equation}  

Letting $y=w_0^*w $ and using relations \eqref{intertwining4} and \eqref{autrig} we obtain that \begin{equation}\label{intertwining5}\theta (x) u_{\delta(g)} y = \overline{\eta(g)}y xu_g,\text{ for all }x\in A\text{ and }g\in G. \end{equation} 
Now consider the Fourier expansion $y= \sum_{h\in G}  y_h u_h$, with $y_h\in A \overline\otimes B$ for all $h\in G$. Using this in relation \eqref{intertwining5} we get that $\sum_h \theta(x) (\sigma\otimes \alpha)_{\delta(g)}(y_h ) u_{\delta(g)h}= \sum_h \overline{\eta(g)} y_h \sigma_h(x) u_{hg}$ for all $g\in G$. Identifying the Fourier coefficients we further get \begin{equation}
    \theta(x)(\sigma\otimes \alpha)_{\delta(g)} (y_{\delta(g^{-1}) s})= \overline{\eta(g)}y_{sg^{-1}} \sigma_{sg^{-1}}(x) \text{ for all } s,g\in G \text{ and }x\in A.
\end{equation}   

In particular, $(\sigma\otimes \alpha)_{\delta(g)} (y_{s})= \overline{\eta(g)}y_{ \delta(g)sg^{-1}}$ for all $s,g\in G$ and hence $\|y_s\|_2= \| y_{\delta(g)sg^{-1}}\|_2$, for all $s,g\in G$. Thus $y$ is supported on $s\in G$ such that the orbit $\{\delta(g)sg^{-1} \,:\, g\in G\}$ is finite. Next we claim that there is only one such orbit and that it consists of a singleton. 

 To this end, suppose there exist $s_1 \neq s_2$ such that both $y_{s_1}$ and $y_{s_2}$ are nonzero.  Then there exist finite-index subgroups $G_{s_1},G_{s_2}$ of $G$ such that $\delta(g) = s_1 g s_1^{-1}$ for all $g \in G_{s_1}$ and $\delta(g) = s_2 g s_2^{-1}$ for all $g \in G_{s_2}$.  In particular, $s_1 g s_1^{-1} = s_2 g s_2^{-1},$ for all $g \in G_{s_1} \cap G_{s_2}$.  Then $s = s_2^{-1}s_1$ is central in $G_{s_1} \cap G_{s_2}$.  But $G_{s_1} \cap G_{s_2}$ has finite index in $G$, which is i.c.c., a contradiction. 
 Therefore, $y_g = 0$ for all but one $g \in G$. 

In conclusion, we have that $y= au_s$ for some unitary $a\in A\overline \otimes B$. Moreover, as a consequence we also have that $\delta= {\rm Ad}(s)$. Therefore if we let $b= (\sigma\otimes \alpha)_s^{-1}(a)$ then we get $w= v b$, where $v = w_0 u_s\in A\rtimes_\sigma G$. 

Combining these with relation \eqref{intertwining5} we get that $\theta(x)u_{sg} b=\theta(x)u_{sgs^{-1}} au_s= \overline{\eta(g)} u_s  b xu_g$ and hence \begin{equation}\label{intertwining6a}\sigma_{s^{-1}}\circ \theta(x)u_g b=\overline{\eta(g)} b x u_g \text{ for all }x\in A,
 \ g\in G.\end{equation}
In particular, this implies  
\begin{equation*}
   \sigma_{s^{-1}}\circ \theta(x) = b x b^*,\text{ for all } x\in A.
\end{equation*}
Moreover, this combined with \eqref{intertwining6a} implies that \begin{equation*}
   b u_g b^* =\eta(g) u_g,\text{ for all }g \in G. 
\end{equation*}
The last two relations give the desired conclusion for $\zeta=\eta$.\end{proof}

\begin{thm}\label{tensoraction profinite} Let $G$ be an i.c.c. group with finite abelianization. Let $G \curvearrowright^\sigma A$ be a weak mixing  automorphism-rigid action and let $G\curvearrowright^\alpha B= \overline{\cup_n B_n}^{\rm SOT}$ be an ergodic, profinite action. Denote by $M=(A \overline \otimes B)\rtimes_{\sigma\otimes \alpha} G$ the crossed product von Neumann algebra corresponding to the canonical diagonal action $G \curvearrowright^{\sigma\otimes\alpha}  A\overline\otimes B$.  
    
\noindent     Then one can find $k \in \mathbb N$ and a $G$-invariant von Neumann subalgebra $B_0 \subseteq B_k$  such that $\N(A\rtimes_\sigma G \subseteq M)'' = (A\overline\otimes B_0) \rtimes_{\sigma\otimes\alpha} G$. 
    \end{thm}

\begin{proof} Since $G\curvearrowright^{\sigma} A $ is weak mixing, the inclusion $B\rtimes_\alpha G \subseteq (A \overline\otimes B)\rtimes_{\sigma\otimes\alpha}  G$ satisfies the following form of the  WAHP: there exists an infinite sequence $(g_i)_i\subseteq G$ such that for every $z,t \in M\ominus (B\rtimes_\alpha G)$ we have \begin{equation}\label{wahp}
    \lim_i \| E_{B\rtimes_\alpha G}(zu_{g_i}t)\|_2
= 0.\end{equation} 
Now fix $w\in \N(A\rtimes_\sigma G \subseteq M)$. By Theorem \ref{normalizer2} one can find unitaries $b\in A\overline\otimes B$ and $v\in A\rtimes_\sigma G$ such that $w=vb$. Moreover, we have $\theta_b(u_g)=\zeta(g) u_g$ for all $g\in G$ which is equivalent to \begin{equation}\label{intertwining6}
    b u_g =\zeta(g)u_g b, \text{ for all } g\in G.
\end{equation}
Now consider $z:= b-E_{B\rtimes_\alpha G}(b)= b-E_B(b)\in M\ominus B\rtimes_\alpha G$. Thus relation \eqref{intertwining6} implies that $E_{B\rtimes_\alpha G}(b) u_g= \zeta(g)u_g E_{B\rtimes_\alpha G}(b)$ and hence, subtracting we get $zu_g= \zeta(g)u_gz$, for all $g\in G$. Using this relation we see that for all $g\in G$  we have 
\begin{equation}\begin{split}
\|E_{B\rtimes_\alpha G}(zz^*)\|_2&= \| u_g E_{B\rtimes_\alpha G}(zz^*)\|_2=\|E_{B\rtimes_\alpha G}( u_g zz^*)\|_2\\
&= \|E_{B\rtimes_\alpha G}(\overline{\zeta(g)} z  u_g z^*)\|_2=\|E_{B\rtimes_\alpha G}(z  u_g z^*)\|_2. \end{split}
\end{equation}
\vskip 0.04in
\noindent  Applying this for the sequence $(g_i)_i\subseteq G$ as in relation \eqref{wahp} we get that \[\|E_{B\rtimes_\alpha G}(zz^*)\|_2= \lim_i \|E_{B\rtimes_\alpha G}(z  u_{g_i} z^*)\|_2=0. \] Therefore $z=0$ and hence $b\in B$.
\vskip 0.07in
\noindent      Relation \eqref{intertwining6} implies that $\alpha_g(b)= \overline{\zeta(g)} b$ for all $g\in G$. Moreover,
$G/[G,G]$ is finite, and hence $G$ has a finite-index normal subgroup $G_0$ such that $G_0 \subseteq {\rm ker}\, \omega$ for all $\omega \in {\rm{Char}}(G)$.  In particular, this holds for $\bar \zeta$, and we have $\alpha_h(b)=b$ for all $h \in G_0$.  Using the finite index condition and Lemma \ref{finite index}, since $b$ is fixed by $\alpha$, we have that $b \in B_k$ for some $k \in \Nat$ (depending only on $G_0$).  As $b $ was an arbitrary normalizer, we conclude that $\N(L(G) \subseteq M)''$ is contained in $B_k\rtimes_\alpha G$.  It now follows from  \cite[Theorem 3.10]{ChDa} that $\N(L(G) \subseteq M)'' = B_0 \rtimes_\alpha G$ for some finite-dimensional von Neumann subalgebra $B_0$ of $B_k$. \end{proof}

\section{Quasinormalizers in Crossed Products and Compactness} \label{section:quasinormalizers}

\noindent This section contains our main results on quasinormalizers for crossed product inclusions.  The first subsection considers the situation of a general $W^*$-dynamical system $(M,G,\alpha,\rho)$, where $\rho$ is a faithful normal state.  In this setting, we characterize  quasinormalizers of the associated inclusion $L(G) \subseteq M \rtimes_\alpha G$ in terms of the Kronecker subalgebra,  which arises from the maximal compact subsystem of the dynamical system.  Stronger results are obtained in \S\ref{compact extensions}, in which we study inclusions of the form $N \rtimes_\alpha G \subseteq M \rtimes_\alpha G$ associated to certain trace-preserving $W^*$-dynamical extension systems $(N \subseteq M, G, \alpha, \tau)$.   In this case, the quasinormalizer is described in terms of the relatively almost periodic elements of $M$ which, under suitable regularity conditions on the inclusion $N \subseteq M$, form a von Neumann algebra which generalizes the Kronecker subalgebra.  Subsection \ref{FZ} presents an application of these structural results to noncommutative dynamics. In particular, a version of the Furstenberg-Zimmer distal tower for general $W^*$-extension systems, \cite{Fu77,Zi76}, is presented in terms of iterated quasinormalizers.
  
\subsection{Quasinormalizers and maximal compactness} The main result of this subsection relates the compact subsystems of a $W^*$-dynamical system $\mathfrak{M}=(M,G,\alpha,\rho)$ to the quasinormalizers of $L(G)$ in the associated crossed product where $G$ is a discrete group $($could be uncountable$)$.  The recent paper \cite{BCM2} analyzed a compact, ergodic system in terms of finite-dimensional invariant subspaces for the associated Koopman representation.  Crucially, it was shown in \cite{BCM1} that any such subspace comes from the von Neumann algebra itself, in fact, from the centralizer $M^\rho \subseteq M$ of the state. 

In \cite{BCM2}, the \emph{Kronecker subalgebra} $M_K \subseteq M$ was defined to be the von Neumann algebra generated by the finite-dimensional subspaces of $M$ that are invariant under $\alpha$.  It was observed that $M_K$ is injective and tracial (under ergodicity and seperability of $M$), and has the following properties:
\medskip
\begin{itemize}
    \item[(i)]$M_K$ is globally invariant under $\alpha$.
    \medskip
    \item[(ii)] The restriction of $\alpha$ to $M_K$ defines a compact subsystem of $\mathfrak{M}$.
    \medskip
    \item[(iii)]  $M_K$ is maximal with respect to properties (i) and (ii), in the sense that $M_K$ contains every $x \in M$ whose orbit under $\alpha$ is  $\norm{\cdot}_2$-precompact. 
\end{itemize}

For a proof of the above when $M$ is $\sigma$-finite with a prescribed faithful normal state $\rho$, use Lemmas \ref{Invariant Subspaces Lemma}, \ref{Kronecker Subalgebra Lemma} and Remark \ref{notfinite}.

The crossed product $M_K \rtimes_\alpha G$ associated to this system is a finite von Neumann algebra, and we consider the inclusion $L(G) \subseteq M_K \rtimes_\alpha G \subseteq M \rtimes_\alpha G$, which provides the key connection between quasinormalizers and dynamics in our main result, in which we compute the von Neumann algebra generated by the one-sided quasinormalizers of $L(G)$.  

\begin{thm} \label{QuasinormalizerCompactSystemTheorem} Let $(M,G,\alpha,\rho)$ be an ergodic $W^*$-dynamical system.  Then 
\begin{equation}\label{main equality}
    vN(\mathcal{QN}^{(1)}(L(G)\subseteq M\rtimes_{\alpha}G))=M_K \rtimes_{\alpha}G.
\end{equation}
\end{thm}

\noindent The remainder of this section will comprise the proof of Theorem \ref{QuasinormalizerCompactSystemTheorem} and its immediate corollaries.  Accordingly, we let $\mathfrak{M}=(M,G,\alpha,\rho)$ be a fixed ergodic $W^*$-dynamical system, and consider the inclusion $L(G) \subseteq M_K \rtimes_\alpha G \subseteq M \rtimes_\alpha G$.  We first prove that $M_K \rtimes_\alpha G \subseteq vN(\mathcal{QN}^{(1)}(L(G) \subseteq M\rtimes_{\alpha}G))$.  To do this, we will need the following lemma, which shows that certain finite-dimensional $L(G)$ modules arising in $M \rtimes_\alpha G$ are automatically $w^*$-closed.  

\begin{lem} \label{weakly closed modules}  For any finite subset $\sett{x_1,\ldots , x_n}$ of $M$, the module $\sum_j x_j L(G)$ is $w^*$-closed in $M \rtimes_\alpha G$.  
\end{lem}
\begin{proof} Applying the Gram-Schmidt procedure from Section \ref{section:modules} in $L^2(M,\rho)$ to the $x_j$, $1 \leqslant j \leqslant n$, we may assume they are mutually orthogonal and have $\norm{x_j}_2 = 1$.  By the Krein-Smulian theorem, we need only show that the intersection of $\sum_j x_j L(G)$ with any closed ball of finite radius is $w^*$-closed.  
\vskip 0.06in 
\noindent Let $(y_\lambda)$ be a uniformly bounded net in $\sum_j x_j L(G)$ such that $y_\lambda$ converges in the $w^*$-topology to $y \in M \rtimes_\alpha G$.  We may then write each $y_\lambda$ as a sum
\[ y_\lambda = \sum_j x_j E_{L(G)}(x_j^* y_\lambda).\]
If $K>0$ is such that $\norm{y_\lambda} \leqslant K$, then for any $\lambda$ and $1 \leqslant j \leqslant n$, we have 
\[ \norm{E_{L(G)}(x_j^* y_\lambda)} \leqslant K \max \norm{x_j^*}.\]
By $w^*$-compactness, we may then drop to a subnet of $(y_\lambda)$ so that for each $1 \leqslant j \leqslant n$ there is some $z_j \in L(G)$ such that $E_{L(G)}(x_j^* y_\lambda)$ converges to $z_j$ in the $w^*$-topology.  Then  $y_\lambda$ converges to $\sum_j x_j z_j,$ an element of $\sum x_j L(G)$.  Thus, the module $\sum x_j L(G)$ is $w^*$-closed. \end{proof}
\noindent Now fix $s \in G$ and suppose $x \in M$ has finite-dimensional orbit under $\alpha$, i.e., there exist $x_1, \ldots , x_n \in M$ such that  ${\rm Orb}(x) = \sett{\alpha_g(x): g \in G} \subseteq {\rm span}\sett{x_1, \ldots x_n}.$
Then, for any $h \in G$, 
\[ u_h x u_s = \alpha_h(x)u_h u_s \in \sum_i x_i L(G).\]
It follows by Lemma \ref{weakly closed modules} that $L(G) xu_s \subseteq \sum_j x_j L(G).$  A symmetric argument shows that $x u_s L(G) \subseteq \sum_j L(G) x_j$, that is, $x u_s$ is a quasinormalizer of $L(G)$.  The von Neumann algebra generated by elements of this form is precisely $M_K \rtimes_\alpha G$, so we conclude that \[M_K \rtimes_\alpha G \subseteq \mathcal{QN}(L(G) \subseteq M \rtimes_\alpha G)''.\]  
This completes the proof of one inclusion in the statement of Theorem \ref{QuasinormalizerCompactSystemTheorem}. The main observation for the opposite inclusion is the following lemma, which implies that the Fourier coefficients of a one-sided quasinormalizer $x = \sum_g x_g u_g  \in M \rtimes_\alpha G$ must have compact orbits under the group action.  

\begin{lem} \label{almost finite dimensional orbit}
Let $0\neq x=\sum_{s\in G}x_{s}u_s\in M\rtimes_{\alpha}G$ be a one-sided quasinormalizer of $L(G)$. Then for each $s\in G$ with $x_{s}\neq 0$ and $\eps>0$ there is a finite-dimensional subspace $K_{\eps,s}$ of $L^2(M,\rho)$ such that for all $z\in \text{Orb}(x_{s})$ we have
\begin{align}
    {\rm dist}(z,K_{\eps,s})=\inf_{m \in K_{\eps,s}}\|z-m\|_{2}<\eps.
\end{align}
\end{lem}
\begin{proof} Let $0\neq x=\sum_{s\in G}x_{s}u_s\in M\rtimes_{\alpha}G$ be a one-sided quasinormalizer of $L(G)$.  Then the $L(G)$-module $\mathcal{H} = \overline{L(G) x L(G)}^{\norm{\cdot}_2} \subseteq L^2(M \rtimes_\alpha G)$ is finitely generated. Denote by $\ip{\cdot}{\cdot}$ the $L(G)$-valued inner product. Then $\Hil$ admits a finite orthonormal basis of left-bounded vectors $\sett{\eta_1, \ldots , \eta_n}$, so that any left-bounded vector $\eta \in \Hil$ may be expressed as a combination
\[ \eta = \sum_{j=1}^n \eta_j \ip{\eta_j}{\eta}, \quad \eta \in \Hil.\]
Let $h \in G.$  Then the vector $\xi_h = u_h x \Omega \in \Hil$ is left-bounded, so by the above formula,
\[ \xi_h = \sum_{j=1}^n \eta_j \ip{\eta_j}{\xi_h}.\]
   For $j = 1,  \ldots, n$ we may express $\eta_j$ and $\ip{\eta_j}{\xi_h}$ as functions on $G$, valued (respectively) in $L^2(M)$ and $\mathbb{C}$, and write
\[ \eta_j = (\eta_j(g))_{g \in G} \quad \text{and} \quad \ip{\eta_j}{\xi_h} = (\lambda_{j,t}^{(h)})_{t \in G},\]
where $\sum_g \norm{\eta_j(g)}_2^2 < \infty$ and $\sum_t \abs{\lambda_{j,t}^{(h)}}^2 < \infty.$  Note that, for $1\leqslant j\leqslant n$ and any $h \in G$, we have 
\[(\sum_t \abs{\lambda_{j,t}^{(h)}}^2)^{1/2} = \norm{\ip{\eta_j}{\xi_h}}_2 \leqslant \norm{L_{\eta_j}} \norm{x}_2 \leqslant C \norm{x},\]
where $C = \max_j \norm{L_{\eta_j}}$.
\vskip 0.06in 
\noindent 
Each product $\eta_j \ip{\eta_j}{\xi_h}$, $1\leqslant j\leqslant n$, then defines an element of $\Hil$ by the convolution formula (see \eqref{convformula}) 
\[ (\eta_j \ip{\eta_j}{\xi_h})(k) = \sum_{g \in G} \lambda_{j,\,{g^{-1}k}}^{(h)} \eta_j(g), \quad k \in G.\]
\vskip 0.06in 
\noindent It follows that for each $s \in G$, the sum 
$\sum_{j=1}^n (\eta_j \ip{\eta_j}{\xi_h})(hs)$
picks out the $hs$-coefficient in the Fourier series of $u_h x = \sum_{g \in G} u_h x_g u_g = \sum_{g \in G} \alpha_h(x_g) u_{hg}$. That is, for each $s \in G$ we have 
\[ \alpha_h(x_s)\Omega = \sum_{j=1}^n \sum_{g \in G} \lambda_{j,\,{g^{-1}hs}}^{(h)} \eta_j(g).\]
Given $\eps>0$, choose a finite subset $F \subseteq G$ so large that 
\begin{align} \sum_{j=1}^n \sum_{g \in G\setminus F} \norm{\eta_j(g)}_2^2 < \frac{\eps}{n\,C\,\norm{x}}.\end{align}
Then, using the Cauchy-Schwarz inequality,
\begin{align*} 
\norm{\alpha_h(x_s)\Omega-\sum_{j=1}^n \sum_{g \in F} \lambda_{j,\,{g^{-1}hs}}^{(h)} \eta_j(g)}_2
&=\norm{\sum_{j=1}^n \sum_{g \in G\setminus F} \lambda_{j,\,g^{-1}hs}^{(h)} \eta_j(g)}_2 \\
&\leqslant \sum_{j=1}^n \left(\sum_{g \in G \setminus F} \norm{\eta_j(g)}_2^2 \right)^{1/2} \left(\sum_{g \in G\setminus F} \abs{ \lambda_{j,\,g^{-1}hs}^{(h)}}^2 \right)^{1/2} \\
&<\sum_{j=1}^n \left(\frac{\eps}{n\,C\,\norm{x}}\right) \left(\sum_{g \in G\setminus F} \abs{ \lambda_{j,\,g^{-1}hs}^{(h)}}^2 \right)^{1/2} \\
&\leqslant \sum_{j=1}^n \left(\frac{\eps}{n\,C\,\norm{x}}\right) \left(\sum_{t \in G} \abs{ \lambda_{j,t}^{(h)}}^2 \right)^{1/2} \\
&\leqslant \sum_{j=1}^n \left(\frac{\eps}{n\,C\,\norm{x}}\right)\,C\,\norm{x}=\eps.\end{align*}
Each vector $\sum_{j=1}^n \sum_{g \in F} \lambda_{j,{g^{-1}hs}}^{(h)} \eta_j(g)$ lies in the finite-dimensional subspace 
\[K_{\eps, s}={\rm span}\sett{\eta_{j}(g):1 \leqslant j \leqslant n, g \in F} \subseteq L^2(M,\rho),\]
and the above estimate is independent of $h \in G$.  Thus, the space $K_{\eps,s}$ is such that ${\rm dist}(z,K_{\eps,s})<\eps$ for any $z \in {\rm Orb}(x_s)$.  \end{proof}
\vskip 0.03in 
\noindent As noted above, Lemma \ref{almost finite dimensional orbit} implies that the Fourier coefficients of a one-sided quasinormalizer have precompact orbit under the group action.  
\begin{cor} \label{compact orbit}
If $x = \sum_{s \in G}x_s u_s \in \mathcal{QN}^{(1)}(L(G) \subseteq M \rtimes_\alpha G)$, then for each $s \in G$, ${\rm Orb}(x_s)$ is $\norm{\cdot}_2$-precompact.
\end{cor}
\begin{proof}
Let $x = \sum_{s \in G}x_s u_s \in \mathcal{QN}^{(1)}(L(G) \subseteq M \rtimes_\alpha G)$ and fix $s \in G$. There is nothing to prove if $x_s=0$, thus assume that $x_s\neq 0$. Given $\eps >0$, let $K_{\eps,s}$ be the subspace of $L^2(M,\rho)$ satisfying the requirements of Lemma \ref{almost finite dimensional orbit}.  Write $F$ for the closed ball of radius $\norm{x_s}_2 + \eps$ in $K_{\eps,s}$.  Then by compactness of $F$ and density of $M$ in $L^2(M)$ there exist $m_1, \ldots, m_k \in M$ such that ${\rm Orb}(x_s) \subseteq \bigcup_{i=1}^k B_{\eps}(m_i)$.  It follows that ${\rm Orb}(x_s)$ is totally bounded in $\norm{\cdot}_2$, and this completes the proof.
\end{proof}\vskip 0.03in 
\noindent It now follows from Corollary \ref{compact orbit} that if $x = \sum_{s \in G} x_s u_s \in \mathcal{QN}^{(1)}(L(G) \subseteq M \rtimes_\alpha G)$ then each Fourier coefficient $x_s$ lies in the subalgebra $M_K \subseteq M$, and therefore, $x \in M_K \rtimes_\alpha G$.  This shows $vN(\mathcal{QN}^{(1)}(L(G) \subseteq M \rtimes_\alpha G)) \subseteq M_K \rtimes_\alpha G$, thus completing the proof of Theorem \ref{QuasinormalizerCompactSystemTheorem}.

\noindent Theorem \ref{QuasinormalizerCompactSystemTheorem} also addresses the natural question of whether the von Neumann algebras generated by one-sided and two-sided quasinormalizers may coincide for inclusions of the form $L(G) \subseteq M \rtimes_\alpha G$.  It was shown in \cite{FaGaSm} that these may differ for inclusions of group von Neumann algebras $L(H) \subseteq L(G)$.  An immediate corollary of the proof of Theorem \ref{QuasinormalizerCompactSystemTheorem} is that this cannot happen in the crossed product setting.  

\begin{cor}
If $(M,G,\alpha,\rho)$ is a $W^*$-dynamical system, then 
\[\mathcal{QN}(L(G) \subseteq M \rtimes_\alpha G)''=\mathcal{QN}^{(1)}(L(G) \subseteq M \rtimes_\alpha G)''.\]
\end{cor}

\noindent We now observe that Theorem \ref{QuasinormalizerCompactSystemTheorem} significantly generalizes $($allowing uncountable discrete groups and $\sigma$-finite $M)$ the  result of Packer \cite{Pac0}, mentioned in the introduction.  That result computed the normalizer for an inclusion $L(G) \subseteq M \rtimes_\alpha G$ in terms of the von Neumann subalgebra $M_0 \subseteq M$ generated by the eigenvectors of $\alpha$ for actions of countable groups.  Note that the subalgebra $M_0$ is clearly invariant under $\alpha$, and can be seen to be equal to $M_K$, as follows.  If $x \in M$ is such that ${\rm Orb}(x)$ has finite-dimensional span, then the vector $x \Omega$ lies in a finite-dimensional subspace $\mathcal{F}$ of $M \Omega$ which is invariant under the group $\sett{V_h: h \in G}$ of Koopman unitaries associated to $\alpha$.  Since $G$ is abelian, the ${V_h}_{\upharpoonleft \mathcal{F}}$ are simultaneously unitarily diagonalizable, and hence, $\mathcal{F}$ has an orthonormal basis of eigenvectors of $\alpha$.  Therefore, $x \in M_0$.  This proves that $M_K \subseteq M_0,$ and the reverse containment is obvious, so they are equal, and moreover $M_K \rtimes_\alpha G = M_0 \rtimes_\alpha G$. On the other hand, if $x\in M_0$ is an eigenvector of the action $\alpha$ there exists a character $\chi$ of $G$ such that $\alpha_g(x)=\chi(g) x$ for all $g\in G$. Scaling $x$ so that $\| x\|=1$ and using ergodicity of the action we deduce that $x^*x$ and $xx^*$ are both 
 $1$, thus $x$ is a unitary in $M_0$. Then $uu_g u^* = u\alpha_g(u^*) u_g =\overline{\chi(g)} u_g$ for all $g\in G$. Consequently, $u$ normalizes $L(G)$. The following is then immediate:
\begin{align*}
  \mathcal{QN}^{(1)}(L(G) \subseteq M \rtimes_\alpha G)''&=\mathcal{QN}(L(G) \subseteq M \rtimes_\alpha G)''\\
  &= M_K \rtimes_\alpha G = M_0 \rtimes_\alpha G \\
  &\subseteq\mathcal{N}(L(G) \subseteq M \rtimes_\alpha G)''\subseteq \mathcal{QN}^{(1)}(L(G) \subseteq M \rtimes_\alpha G)''.
\end{align*}
%\textcolor{red}{On the other hand, since $G$ is abelian and $\alpha$ is ergodic, $L(G)$ is a masa in $M \rtimes_\alpha G$, so by %Corollary 4.5 of \cite{FaGaSm}, we have  $\mathcal{N}(L(G) \subseteq M \rtimes_\alpha G)''=\mathcal{QN}(L(G) \subseteq M %\rtimes_\alpha G)''$.} 

These remarks prove the following corollary.
\begin{cor}
If $G$ is a discrete abelian group, and $\alpha$ is an ergodic, trace-preserving action of $G$ on a finite von Neumann algebra $(M,\tau)$, then 
\[\mathcal{N}(L(G) \subseteq M \rtimes_\alpha G)'' = M_K \rtimes_\alpha G.\]
\end{cor}

\begin{rem}
  Note that many results about normalizers of masas or subalgebras in the literature require the assumption of a separable predual of $M$ or the hypothesis that the acting group is countable and discrete or separable (see for instance \cite[Proposition  2.1]{NeSt}). This is because, for example in the context of masas the ambient masa is identified with a suitable $L^\infty$ space and separability assumption of the GNS space allows the use of spectral theorem in its strongest form through direct integrals; the latter being invoked either implicitly or explicitly. The proofs in this section do not need any such requirements and work even when the predual of $M$ is not separable, $\rho$ is a faithful normal state and the acting discrete group is uncountable.
\end{rem}

The following corollary generalizes \cite[Corollary 7.6]{BCM1}.

\begin{cor}\label{one-side-qregular}
Let $(M,G,\alpha,\rho)$ be an ergodic $W^*$-dynamical system. Then $$\mathcal{QN}^{(1)}(L(G) \subseteq M \rtimes_\alpha G)''\subseteq (M \rtimes_\alpha G)^{\widehat\rho}.$$ 
In particular, if $M \rtimes_\alpha G$ is properly infinite then $L(G)$ cannot be one-sided quasiregular in $M \rtimes_\alpha G$ and $\alpha$ must have a non-trivial weakly mixing component.
%{statecrossedproduct}
\end{cor}

\begin{proof}
    The proof follows directly from Theorem \ref{QuasinormalizerCompactSystemTheorem} and the fact that $\widehat\rho_{\upharpoonleft M_K\rtimes_\alpha G}$ is a trace.
\end{proof}

%{We conclude this section with an observation on the connection between mixing properties of inclusions and actions which is well known to experts, though we could not find a proof in the literature.  Recall that an inclusion $B \subseteq N$ of finite von Neumann algebras, with conditional expectation $E:N \rightarrow B$ is said to be \emph{weak mixing} [CITE CFM] if there is a sequence $(u_k)$ of unitary operators in $B$ such that 
%\[ \lim_{k \rightarrow \infty} \norm{E(x u_k y)}_2 = 0\]
%for any $x,y \in N$ with $E(x)=E(y)=0$ (this condition on $B \subseteq N$ is also known as the \emph{weak asymptotic homomorphism property} [CITE SSWW]).  It was shown in [CITE CFM] that an inclusion $B \subseteq N$ is weak mixing if and only if $\mathcal{QN}(B \subseteq N)=B$.  Then, combining Theorem \ref{QuasinormalizerCompactSystemTheorem} with the main results of [BCM2] we have the following corollary.  
%\begin{Corollary}If $G$ is a discrete group, and $\alpha$ is an ergodic, properly outer, trace-preserving action of $G$ on a finite von Neumann algebra $(M,\tau)$, then $\alpha$ is weak mixing if and only if $L(G) \subseteq M \rtimes G$ is a weak mixing inclusion.\end{Corollary}
%}

\subsection{Relatively almost periodic elements and compact extensions} \label{compact extensions} Throughout this section we let $N\subseteq (M,\tau)$ be an inclusion of tracial von Neumann algebras and we let $(M, G, \sigma,\tau)$ be a $\tau$-preserving $W^*$-dynamical system such that $N$ is a $G$-invariant von Neumann subalgebra. Such a  system is called a \emph{$\tau$-preserving  $W^*$-dynamical extension system} and will be denoted by  $(N\subseteq M, G, \sigma,\tau)$. 
\begin{definition}\label{apr} Let $(N\subseteq M, G, \sigma,\tau)$ be a $\tau$-preserving $W^*$-dynamical extension system. An element $f\in L^2(M)$ is called \emph{almost periodic relative to $N$} if and only if for every  $\varepsilon>0$ one can find elements $\eta_1,\ldots ,\eta_n\in  L^2(M)$ such that for every $g\in G$ there exist $\kappa(g,j) \in N$
with $1\leqslant j\leqslant n$ satisfying \begin{enumerate}
    \item $\sup_{h\in G}\|\kappa (h,j)\|_\infty<\infty$, and
    \medskip
    \item $\|\sigma_g(f)-\sum^n_{j=1} \eta_j \kappa(g,j)\|_2<\varepsilon$.
\end{enumerate}
\end{definition}
\noindent Basic approximations show that in part (2) of the previous definition one can actually pick $\eta_j \in M$ as opposed to its 
$L^2$-space. Throughout the remaining sections we denote by $\cK_{N,M}\subseteq L^2(M)$ the set of all elements that are almost periodic  relative to $N$. We will also denote by $\mathcal P_{N,M}:= M\cap \cK_{N,M}$.

\begin{prop}\label{subspace} The following properties hold: \begin{enumerate}
    \item $\cK_{N,M}\subseteq L^2(M)$ is a $G$-invariant Hilbert subspace.
    \medskip
    \item $\mathcal P_{N,M}\subseteq M$ is a SOT-closed, $G$-invariant, linear subspace. 
\end{enumerate} 
\end{prop}

\begin{proof} It is immediate from the definitions that $\cK_{N,M}$ is $\|\cdot\|_2$-closed and $G$-invariant, and that $\mathcal{P}_{N,M}$ is a $G$-invariant linear subspace. It remains to show  that $\mathcal P:=\mathcal P_{N,M}$ is SOT-closed.  Let $f\in\overline {\mathcal P}^{\rm SOT}$ and let $(f_i)_i\subset\mathcal P$ be a net such that $f_i\rar f$ in SOT.  Fix $\varepsilon>0$ and let $i$ be such that $\|f_i-f\|_2<\frac{\varepsilon}{2}$. From Definition \ref{apr}, there exist $\eta_1,\ldots, \eta_n \in M$ and  $\kappa(g,j)\in N$ with $1\leqslant j\leqslant n $ such that $\sup_{h\in G}\|\kappa(h,j)\|_\infty <\infty$ and $\|\sigma_g(f_i)- \sum_j \eta_j \kappa(g,j)\|_2< \frac{\varepsilon}{2}$. These, combined with the triangle inequality, show that
\begin{equation*}\begin{split}
 \|\sigma_g(f)- \sum_j \eta_j \kappa(g,j)\|_2&\leqslant \|\sigma_g(f)-\sigma_g(f_i )\|_2+\| \sigma_g(f_i)-\sum_j \eta_j \kappa(g,j)\|_2\\&\leqslant \|f-f_i\|_2 +\frac{\varepsilon}{2} <\varepsilon,  \end{split}\end{equation*}
proving that $f\in \mathcal{P}$ and  giving the desired claim.\end{proof}

\begin{definition}\label{compactquasinormalincl}A $\tau$-preserving $W^*$-dynamical extension system $(N\subseteq M, G, \sigma,\tau)$ is called \emph{compact} if and only if $\cK_{N,M}= L^2(M)$. 
\end{definition}

%The compactness condition on an extension action imposes drastic conditions at the level of the underlying inclusion $\cN\subseteq \cM$. For example we have the following
%\begin{proposition} If $\Gamma \ca (\cN\subseteq \cM)$ is is a compact action then $\mathscr{QN}_\cM(\cN)''=\cM$.
%\end{proposition}

%\begin{proof}

%\end{proof}

\noindent Note that when the subalgebra $N = \mathbb{C}1$, this condition coincides with compactness of the system $(M,G,\alpha,\tau)$.  Next, we show that Definition \ref{compactquasinormalincl} extends the classical notion of compactness for extensions of actions on abelian von Neumann algebras. For convenience we recall one of the equivalent definitions from the classical situation.

\begin{definition}\label{compactabelian} Let $M$ be an abelian von Neumann algebra.
Then a $W^*$-dynamical extension system $(N\subseteq M, G, \sigma,\tau)$ is compact if and only if $L^2(M)$ decomposes as a direct sum of finitely generated $G$-invariant $N$-modules.
\end{definition}

\begin{prop}\label{same} If $M$ is an abelian von Neumann algebra and $(N\subseteq M, G, \sigma,\tau)$ is a $\tau$-preserving, ergodic  $W^*$-dynamical extension system, then the notions of compactness from Definitions \ref{compactabelian} and \ref{compactquasinormalincl} coincide.

\end{prop}

\begin{proof} Assume $(N\subseteq M, G, \sigma,\tau)$ is compact as in Definition \ref{compactabelian}. Then, $L^2(M)=\oplus_i \cH_i$ where $\cH_i$ are finitely generated $G$-invariant $N$-modules. Moreover, since the action of $G$ on $M$ is ergodic, using a standard argument (e.g.\ \cite[Proposition 3.4]{IT20}) we can assume each $\cH_i$ admits a finite $N$-basis, $(\eta_j)_j\subseteq M$ -- i.e.,\ $\eta_j$'s are $N$-orthogonal and for every $\xi\in \cH_i$ we have $\xi= \sum_j \eta_j E_N(\eta_j^*\xi)$. Thus for every $g\in G$ and $\xi\in \cH_i$ we have $\sigma_g(\xi) = \sum_j \eta_j E_N(\eta_j^*\sigma_g (\xi))$. Let $\xi\in \cH_i$ and $\varepsilon>0$. Let $m\in \sum_j\eta_jN\subseteq M$  be such that $\|\xi-m\|_2 <\varepsilon$. Since $(\eta_j)_j$ is a $N$-basis we have  $\|\sigma_g(\xi)-\sum_j \eta_j E_N(\eta_j^*\sigma_g(m))\|^2_2= \|\sum_j \eta_j E_N(\eta_j^*(\sigma_g (\xi-m)) \|^2_2\leqslant \|\sigma_g(\xi-m )\|_2^2=\|\xi-m\|_2^2<\varepsilon^2$. Letting $\kappa(g,j)= E_N(\eta_j^*\sigma_g(m))$ we see that $\|\kappa(g,j)\|_\infty \leqslant (\max_j\|\eta_j\|_\infty) \|m\|_\infty$. Altogether, the prior relations show that $\xi\in \mathcal K_{N,M}$. Thus, $\cH_i\subseteq \mathcal K_{N,M}$. Now use (1) of Proposition \ref{subspace} to establish $L^2(M)\subseteq \mathcal K_{N,M}$. Hence Definition \ref{compactquasinormalincl} is satisfied.  
\vskip 0.08in
\noindent Now we show the reverse implication. Let $N \subseteq P\subseteq M$ be the maximal von Neumann subalgebra such that $(N\subseteq P, G, \sigma,\tau)$ is compact in the sense of Definition \ref{compactabelian}; see for instance \cite[Corollary 3.6]{IT20}. Thus $(P\subseteq M, G, \sigma,\tau)$ is weak mixing relative to $N$, i.e., \ there is a net $(g_\lambda)_\lambda\subseteq G$ (in this case $b_\lambda$'s in Definition \ref{rwm} can be chosen to be $1$) such that for all $\eta,\zeta\in M\ominus P$ we have  
\begin{equation}\label{weakmixing2}
    \lim_{\lambda}\|E_N(\eta\sigma_{g_\lambda}(\zeta))\|_2=0.
\end{equation} 
To get our conclusion it suffices to show that $P=M$. Pick $\xi\in M\ominus P$ and fix $\varepsilon>0$. By Definition \ref{apr} there exist $\eta_i\in M$ for $1\leqslant i\leqslant n$ and $\kappa(g,i)\in N$ with $\sup_g \|\kappa (g,i)\|_\infty=:C<\infty$ such that 
\begin{equation}\label{compactorbit6}\|\sigma_g(\xi)-\sum^n_{i=1} \eta_i \kappa(g,i)\|_2<\frac{\varepsilon}{3}\quad \text{for all }g\in G.\end{equation} Since $M$ is abelian then $\cH=\overline{\eta_1 N+\cdots 
+\eta_n N}$ is a finitely generated $N$-bimodule and by Theorem  \ref{bimodularapprox} one can find $N$-orthogonal elements $y_j\in M$ with 
$1\leqslant j\leqslant m$ so that \begin{equation}\label{ineq2}\| x- \sum^{m}_{j=1} y_j E_N(y_j^* x) \|_2<\frac{\varepsilon}{3Cn},\end{equation} for all $x = \eta_1 x_1 +\cdots +\eta_n x_n$ with $x_i\in (N)_C$. 
\newline Using triangle inequality and basic estimates together with \eqref{compactorbit6}, inequality \eqref{ineq2} for $x= \eta_i$ with $1\leqslant i\leqslant n$ and also the estimate \eqref{contraction} we see for every $g\in G$ we have  
\begin{align}\label{leftbasisapprox}
 &\|\sigma_g(\xi)-\sum_{j=1}^m y_j E_N(y_j^* \sigma_g(\xi))\|_2\\ 
\nonumber\leqslant &\|\sigma_g(\xi)-\sum_{i=1}^n \eta_i \kappa(g,i)\|_2+ \|\sum_{i=1}^n \eta_i \kappa(g,i)-\sum_{j=1}^m y_j E_N(y_j^* \sum_{i=1}^n \eta_i \kappa(g,i))\|_2\\
\nonumber&\quad\quad\quad\quad\quad\quad\quad\quad\quad+\|\sum_{j=1}^m y_j E_N\big(y_j^* (\sigma_g(\xi)-\sum_{i=1}^n \eta_i \kappa(g,i))\big)\|_2\\
\nonumber\leqslant &\|\sigma_g(\xi)-\sum_{i=1}^n \eta_i \kappa(g,i)\|_2+ \sum_{i=1}^n\| \eta_i -\sum_{j=1}^m y_j E_N(y_j^*  \eta_i )\|_2 \|\kappa(g,i)\|_{\infty}\\
\nonumber&\quad\quad\quad\quad\quad\quad\quad\quad+\|\sum_{j=1}^m y_j E_N\big(y_j^* (\sigma_g(\xi)-\sum_{i=1}^n \eta_i \kappa(g,i))\big)\|_2\\
\nonumber< &\frac{\varepsilon}{3}+(Cn)\frac{\varepsilon}{3Cn}+ \|\sigma_g(\xi)-\sum_{i=1}^n \eta_i \kappa(g,i)\|_2<\varepsilon.   
\end{align} 
%%%%%%%%%%%%%%%%%%%%%%%%%%%%%%%%%%%%%%%%%%%%%%%%%%%%%%%%%%
%\begin{equation}\label{leftbasisapprox}\begin{split}
%  &\|\sigma_g(\xi)-\sum_{j=1}^n y_j E_N(y_j^* \sigma_g(\xi))\|_2\\ \leqslant &\|\sigma_g(\xi)-\sum_{i=1}^n \eta_i \kappa(g,i)\|_2+ \|\sum_{i=1}^n \eta_i \kappa(g,i)-\sum_{j=1}^n y_j E_N(y_j^* \sum_{i=1}^n \eta_i \kappa(g,i))\|_2+\\&\quad\quad\quad\quad\quad\quad+\|\sum_{j=1}^n y_j E_N(y_j^* (\sigma_g(\xi)-\sum_{i=1}^n \eta_i \kappa(g,i))\|_2\end{split}\end{equation}
%  \begin{equation*}\begin{split}
 % &\leqslant \|\sigma_g(\xi)-\sum_i \eta_i \kappa(g,i)\|_2+ \sum_i\| \eta_i -\sum_j y_j E_N(y_j^*  \eta_i )\|_2 \|\kappa(g,i)\|_{\infty}+\\&+\|\sum_j y_i E_N(y_j^* (\sigma_g(\xi)-\sum_i \eta_i \kappa(g,i))\|_2\\
 % &< \frac{\varepsilon}{3}+(Cn)\frac{\varepsilon}{3Cn}+ \|\sigma_g(\xi)-\sum_i \eta_i \kappa(g,i)\|_2<\varepsilon.   
%\end{split}\end{equation*} 
%%%%%%%%%%%%%%%%%%%%%%%%%%%%%%%%%%%%%%
\noindent 
Finally, using \eqref{leftbasisapprox} and Cauchy-Schwarz inequality we see for all $\lambda$ we have \begin{equation*}\begin{split}
    \|\xi\|^2_2&=\|\sigma_{g_\lambda}(\xi)\|^2_2=\langle\sigma_{g_\lambda}(\xi), \sigma_{g_\lambda}(\xi)\rangle \\ &<  \varepsilon \|\xi\|_2 + |\langle \sum_{j=1}^m y_j E_N(y_j^* \sigma_{g_\lambda}(\xi)), \sigma_{g_\lambda}(\xi)\rangle|\\ 
    &=  \varepsilon \|\xi\|_2 + \sum_{j=1}^m\|  E_N(y_j^* \sigma_{g_\lambda}(\xi))\|^2_2\\
    &= \varepsilon \|\xi\|_2+ \sum^m_{j=1} \|E_N(z_j \sigma_{g_\lambda}(\xi)) +E_N(y^*_j)E_{N} (\sigma_{g_\lambda}(\xi))\|_2^2\\
    &= \varepsilon \|\xi\|_2+ \sum^m_{j=1} \|E_N(z_j \sigma_{g_\lambda}(\xi) )\|_2^2.
\end{split}
    \end{equation*} 
Here we denoted $z_j= y^*_j-E_N (y^*_j)$ for all $j$, and in the last equality we used that $E_N (\sigma_{g_\lambda}(\xi))= \sigma_{g_\lambda}(E_N(\xi))=0$. As $E_N(z_j)=0$ for all $j$, then using \eqref{weakmixing2} and taking limit over $\lambda$ above we get $\|\xi\|^2_2< \varepsilon \|\xi\|_2$. As $\varepsilon>0$ was arbitrary, we conclude $\|\xi\|_2=0$ and hence $\xi=0$. Since $\xi\in M\ominus P$ was arbitrary we get $M=P$, as desired.  
\end{proof}

\noindent  It has been  known for some time that the subspace of relative almost periodic elements $\mathcal P_{N,M}$ is not generally a von Neumann subalgebra of $M$. An example in this direction  was exhibited by  Austin-Eisner-Tao in \cite[Example 4.4]{AuEiTa11}. Thus, it is natural to investigate what conditions on the inclusion $N\subseteq M$ would ensure that $\mathcal P_{N,M}$ is a von Neumann subalgebra. In this direction, J. Peterson and the third author observed that a sufficient condition is quasi-regularity of $N\subseteq M$. A proof based on arguments in \cite{CP11} was included, with permission, in the recent preprint \cite{JaSp}, and this proof works for $\sigma$-finite $M$ and uncountable discrete $G$.

\begin{thm}[\cite{CP11}]\label{vNalgebra} Let $(N\subseteq M, G, \sigma,\tau)$ be a $W^*$-dynamical extension system. If $\mathcal{QN}(N\subseteq M)''= M$  then $ \mathcal P_{N,M}\subseteq M$ is a $G$-invariant von Neumann subalgebra.

\end{thm}

\vskip 0.08in 
\noindent 
In the remaining part of the subsection we explore the connections between (one-sided) quasinormalizers in crossed product von Neumann algebras and the subspace of relative almost periodic elements of $W^*$-dynamical extension systems.  

\begin{thm}\label{qnormalizer1} 
 Let $(N\subseteq M, G, \sigma,\tau)$ be a $W^*$-dynamical extension system.  Then, $\mathcal{QN}^{(1)}(N\rtimes_\sigma G\subseteq M\rtimes_\sigma G)\subseteq \overline{{\rm span}\, \mathcal P_{N,M}G}^{\|\cdot \|_2}$. In particular, we have  \[vN(\mathcal{QN}^{(1)}(N\rtimes_\sigma G\subseteq M\rtimes_\sigma G))\subseteq \mathcal P_{N,M}''\rtimes_\sigma G.\]
\end{thm}

\begin{proof} To simplify the writing, let $A= M\rtimes_\sigma G$, $B=N\rtimes_\sigma G$ and $C=\mathcal P_{N,M}^{\prime\prime}\rtimes_\sigma G$ and notice $B\subseteq C\subseteq A$. Fix $y\in \mathcal{QN}^{(1)}(B\subseteq A)$ and let $y= \sum_h y_h u_h$ be its Fourier expansion in $M\rtimes_\sigma G$. 

\noindent We will show that $y_h\in \mathcal P_{N,M}$ for all $h\in G $. Towards this, fix $\varepsilon>0$. As  $u_gy\in \mathcal{QN}^{(1)}(B\subseteq A)$, using Theorem \ref{bimodularapprox},  one can find $x_{i}\in A$ with $1\leqslant i\leqslant n$ so that  for every $g\in G$,  \begin{equation*}
    \begin{split}
        \|u_g y- \sum_i x_{i}E_{B}(x^*_{i}u
        _g y)\|_2\leqslant \varepsilon.
    \end{split}
\end{equation*} 
\noindent Approximating the $x_{i}$'s in $\|\cdot\|_2$ via the Kaplansky density theorem, in the prior inequality, we can assume that $x_i \in M K$ for a finite subset $K\subseteq G$.

Letting $x_i = \sum_{h\in K} x^i_h u_h$, the previous inequality implies that for all $g\in G$ we have 
\begin{equation*} \begin{split}
 \varepsilon^2 &\geqslant  \|u_g y- \sum_i x_{i}E_{B}(x^*_{i}u
        _g y)\|^2_2 \\
        &= \|\sum_{h\in G} u_g y_h u_{h}- \sum_{s,t\in K,l\in G}\sum_i x^i_s u_sE_{B}(u_{t^{-1}}(x^i_t)^* u_g y_l u_l)\|^2_2\\
        &=\|\sum_{h\in G} \sigma_g(y_h) u_{gh}- \sum_{s,t\in K,l\in G}\sum_i x^i_s E_{N}(\sigma_{st^{-1}}(x^i_t)^* \sigma_{st^{-1}g} (y_l)) u_{st^{-1}gl}\|^2_2\\
        &=\sum_{h\in G} \|\sigma_g(y_h) - \sum_{s,t\in K, i} x^i_s E_{N}(\sigma_{st^{-1}}(x^i_t)^* \sigma_{st^{-1}g} (y_{g^{-1} ts^{-1}gh})) \|^2_2\\
        &=\sum_{h\in G} \|\sigma_g(y_h) - \sum_{i,s} x^i_s \left(\sum_{t} E_{N}(\sigma_{st^{-1}}(x^i_t)^* \sigma_{st^{-1}g} (y_{g^{-1} ts^{-1}gh}))\right) \|^2_2.
        \end{split}
\end{equation*}

\noindent Since $K$ is a finite set this inequality clearly implies that each $y_h$ satisfies the compactness definition  with $\eta_j = x_s^i$ and $\kappa(g,j)= \sum_{t} E_{N}(\sigma_{st^{-1}}(x^i_t)^* \sigma_{st^{-1}g} (y_{g^{-1} ts^{-1}gh}))$. Thus $y_h\in \mathcal P_{N,M}$ for all $h\in G$ as desired. The rest of the conclusion follows.  \end{proof}

\noindent With these preparations at hand we are now ready prove the following generalization of Theorem \ref{QuasinormalizerCompactSystemTheorem} in the context of trace-preserving $W^*$-dynamical extension systems.

\begin{thm}\label{quasinormalizer2}  Let $(N\subseteq M, G, \sigma,\tau)$ be a $W^*$-dynamical extension system and assume that $\mathcal{QN}(N\subseteq M)''= M$.  Then $N\subseteq \mathcal P_{N,M}\subseteq M$ is $G$-invariant von Neumann subalgebra and $vN(\mathcal{QN}^{(1)}(N\rtimes_\sigma G\subseteq M\rtimes_\sigma G))= \mathcal P_{N,M}\rtimes_\sigma G$. 

\end{thm}

\begin{proof} As the first part of the conclusion is immediate from Theorem \ref{vNalgebra}, we will only prove the second part. Denote by $B:=N \rtimes_\sigma G$, $D:=vN(\mathcal{QN}^{(1)}(N\rtimes_\sigma G \subseteq M\rtimes_\sigma G))$ and $ A:=M \rtimes_\sigma G$. By Theorem \ref{rwahp}, the triple  $B \subseteq D\subseteq  A$ satisfies the relative WAHP. Moreover, by Theorem \ref{srwm} one can pick the net of unitaries witnessing the relative WAHP in any subgroup of unitaries generating $B$; in particular, we can pick them in $\mathcal U (N)G$. Thus one can find $(v_\lambda)_\lambda\subseteq \mathcal U(N)$ and $(u_{g_\lambda})_\lambda\subseteq G$ such that for every $\xi,\eta \in A \ominus D$ we have  \begin{equation}\label{weakmixingrelative}
    \lim_{\lambda}\|E_B(\xi v_\lambda u_{g_\lambda} \eta )\|_2=0.
\end{equation}    
Letting $P:=\mathcal P_{N,M}$, Theorem \ref{qnormalizer1} implies our conclusion, once we show that $P \subseteq D$. Toward this, fix  $\zeta\in P$ be such that $\norm{\zeta}\leqslant 1$ and let $\xi := \zeta - E_D(\zeta)\in M\ominus D$. Next we will argue that $\xi=0$, which will give the conclusion. 
\vskip 0.06in 
\noindent First, we need to establish the following. 

\begin{claim}\label{modifiedapr} For every $\varepsilon>0$  one can find $\eta_1,\ldots,\eta_n\in L^2 (M)$ and $h_1,\ldots, h_n\in G$ (non-necessarily distinct), so that for every $g\in G$ there are $\kappa(g,j) \in N$
with $1\leqslant j\leqslant n$ satisfying \begin{equation}\label{compactorbit4}
     \sup_{g\in G}\|\kappa(g,j)\|_\infty=C<\infty\text{ and }\| \sigma_g(\xi)-\sum_j \eta_j \kappa(g,j)u_{gh_jg^{-1}}\|_2<\varepsilon. 
\end{equation}
    
\end{claim}
\noindent \emph{Proof of Claim \ref{modifiedapr}}. Fix $\varepsilon>0$. As $\zeta\in P$   one can find $\eta'_1,\ldots,\eta'_m\in L^2 (M)$ such that for every $g\in G$ there exist $\kappa'(g,i) \in N$
with $1\leqslant i\leqslant m$ satisfying \begin{equation}\label{apr6}
     \sup_{h\in G}\|\kappa'(h,i)\|_\infty=:C'<\infty\text{ and }\| \sigma_g(\zeta)-\sum_i \eta'_i \kappa'(g,i)\|_2<\frac{\varepsilon}{2}. 
\end{equation}\vskip 0.03in 
\noindent  Using Theorem \ref{qnormalizer1}, we have $E_{D}(\zeta)\in D\subseteq P\rtimes_\sigma G$. Thus one can find a finite subset $F\subseteq G$ and $a_s\in (P)_1$ for all $s\in F$  such that $\|E_D(\zeta)- \sum_{s\in F} a_s u_s\|_2\leqslant \frac{\varepsilon}{4}$. Hence for all $g \in G$ we have \begin{equation}\label{norm2aprox}\|\sigma_g(E_D(\zeta))- \sum_{s\in F} \sigma_g(a_s) u_{gsg^{-1}}\|_2\leqslant \frac{\varepsilon}{4}.\end{equation}  As $a_s\in P$ there are  $\eta^s_1,\ldots,\eta^s_{n_s}\in L^2(M)$ so that for every $g\in G$ there are $\kappa^s(g,j) \in N$
with $1\leqslant j\leqslant n_s$ satisfying 
     \begin{equation}\label{apr5}\begin{split}&  \sup_{h\in G,j,s}\|\kappa^s(h,j)\|_\infty=: C''<\infty\text{ and }  \| \sigma_g(a_s)-\sum_j \eta^s_j \kappa^s(g,j)\|_2<\frac{\varepsilon}{4|F|}.
    \end{split}\end{equation} 
 Combining inequalities \eqref{norm2aprox} and \eqref{apr5}, for all $g\in G$  we get  \begin{equation*}\begin{split}
     &\|\sigma_g(E_D(\zeta))-\sum_{s\in F}\sum_j \eta^s_j \kappa^s(g,j)u_{gsg^{-1}} \|_2\\
\leqslant&\|\sigma_g(E_D(\zeta))- \sum_{s\in F} \sigma_g(a_s) u_{gsg^{-1}}\|_2 + \sum_{s\in F}\| \sigma_g(a_s)-\sum_j \eta^s_j \kappa^s(g,j)\|_2\\<& \frac{\varepsilon}{4}+ |F|\frac{\varepsilon}{4|F|}= \frac{\varepsilon}{2}. \end{split} \end{equation*}
 
\noindent  Combining this with relation \eqref{apr6}, for all $g\in G$, we have  
\begin{equation*}\begin{split}
 \| \sigma_g(\xi)-\left ( \sum_i \eta'_i \kappa'(g,i)- \sum_{s\in F} \sum_{j} \eta^s_j \kappa^s(g,j)u_{gsg^{-1}}\right )\|_2\leqslant \varepsilon.  
\end{split}
\end{equation*}
 This together with the uniform boundedness conditions from \eqref{apr6}-\eqref{apr5} yield our claim. $\hfill\blacksquare$
   
\vskip 0.08in

\noindent Now fix an arbitrary $\varepsilon>0$. Then by Claim \ref{modifiedapr} one can find $\eta_1,\ldots,\eta_n\in L^2 (M)$ and $h_1,\ldots, h_n\in G$, so that for every $g\in G$ there are $\kappa(g,j) \in N$
with $1\leqslant j\leqslant n$ satisfying \begin{equation}\label{compactorbit4}
     \sup_{g\in G}\|\kappa(g,j)\|_\infty=C<\infty\text{ and }\| \sigma_g(\xi)-\sum_j \eta_j \kappa(g,j)u_{gh_jg^{-1}}\|_2<\frac{\varepsilon}{4}. 
\end{equation}

\noindent  As $\mathcal{QN}(N\subseteq M)''=M$, using basic approximations we can assume  $\eta_i\in\mathcal{QN}(N\subseteq M)$. Using Theorem \ref{bimodularapprox} one can find  $x_{j,i}\in M$ such that $\|x \eta_j y- \sum_ix_{j,i}E_N(x_{j,i}^* x\eta_j y)\|_2\leqslant \frac{\varepsilon\|x\|_\infty\|y\|_\infty}{4Cn}$ for all $x,y\in N$. This inequality together with the second part of \eqref{compactorbit4}  implies that  for every $x\in (N)_1$ and $g\in G$ we have \begin{equation}\label{compactorbit5}\|x \sigma_g(\xi)-\sum_{j,i} x_{j,i} E_N (x^*_{j,i} x\eta_j  k(g,j))u_{gh_jg^{-1}}\|_2<\frac{\varepsilon }{2}.\end{equation}
Since $E_D(\xi)=0$ we have that  $E_D(x\sigma_g(\xi))=0$, for all $x\in N$. This further implies that  $\|\sum_{j,i} E_D(x_{j,i}) E_N (x^*_{j,i} x\eta_j  k(g,j))u_{gh_jg^{-1}}\|_2<\frac{\varepsilon }{2}$. Combining it with \eqref{compactorbit5} and letting  $y_{j,i}:=  x_{j,i}-E_D(x_{j,i})$, for every $x\in N$ and $g\in G$  we get  \begin{equation}\|x \sigma_g(\xi)-\sum_{j,i} y_{j,i} E_N (x^*_{j,i} x\eta_j  k(g,j))u_{gh_jg^{-1}}\|_2<\varepsilon .\end{equation}
Using this inequality we see that for $x=v_\lambda$ and $g=g_\lambda$ we have 
\begin{equation*} \begin{split}\|\xi\|^2_2&=  \|v_\lambda \sigma_{g_\lambda}(\xi)\|^2= \langle v_\lambda \sigma_{g_\lambda}(\xi),v_\lambda \sigma_{g_\lambda}(\xi) \rangle\\& \leqslant  \varepsilon +|\langle\sum_{j,i} y_{j,i} E_N (x^*_{j,i} v_\lambda\eta_j  \kappa(g_\lambda,j))u_{g_\lambda h_jg^{-1}_\lambda} , v_\lambda \sigma_{g_\lambda}(\xi)\rangle |\\
&\leqslant \varepsilon +\sum_{j,i}|\langle E_N (x^*_{j,i} v_\lambda\eta_j  \kappa(g_\lambda,j))u_{g_\lambda h_jg^{-1}_\lambda} ,  y_{j,i}^* v_\lambda \sigma_{g_\lambda}(\xi)\rangle |\\
&= \varepsilon +\sum_{j,i}|\langle E_N (x^*_{j,i} v_\lambda\eta_j  \kappa(g_\lambda,j))u_{g_\lambda h_j},  E_{B}(y_{j,i}^* v_\lambda u_{g_\lambda} \xi)\rangle |\\
&\leqslant \varepsilon +\sum_{j,i} C \|x_{j,i}\|_\infty \|\eta_j\|_2 \|E_{B}(y_{j,i}^* v_\lambda u_{g_\lambda} \xi)\|_2.
\end{split}    
\end{equation*}
Since by \eqref{weakmixingrelative} we have $\lim_{\lambda} \|E_{B}(y_{j,i}^* v_\lambda u_{g_\lambda} \xi)\|_2=0$ and the set $x_{j,i}$'s are finite (arguing with nets as in the proof of Lemma \ref{lemma:WeaklyMixingSingular}) the previous inequality gives that $\|\xi\|^2_2\leqslant \varepsilon$. Since $\varepsilon>0$ was arbitrary we get $\xi=0$, as desired. \end{proof}

\noindent We remark that the previous theorem and its proof yield the following corollary.

 \begin{cor} Let $(N\subseteq M, G, \sigma,\tau)$ be a $W^*$-dynamical extension system.  Then the $W^*$-dynamical extension system $(\mathcal P_{N,M}''\subseteq M, G, \sigma,\tau)$ is weak mixing relative to $N$ in the sense of Definition \ref{rwm}. %i.e. there is a sequence $(u_ng_n)\subseteq \mathcal U(N)G$ where $(g_n)$ is infinite such that for every $\xi,\eta\in M\ominus \mathcal P_{N,M}''$ we have
 %\begin{equation}
 %    \lim_n \|E_{N}(\xi u_n\sigma_{g_n}(\eta)\|_2=0.
 %\end{equation}
 %When $N\subseteq \mathcal Z(M)$ one can pick $u_n=1$ for all $n$.
 \end{cor}

\noindent A further consequence of Theorem \ref{quasinormalizer2} is that, for crossed product inclusions associated to abelian $W^*$-dynamical extension systems, the von Neumann algebras generated by quasinormalizers and one-sided quasinormalizers coincide.

\begin{cor}\label{quasinorm-compact} If $ M$ is abelian and  $(N\subseteq M, G, \sigma,\tau)$ is an ergodic  $W^*$-dynamical extension system then $$vN(\mathcal {QN}^{(1)}(N\rtimes_\sigma G\subseteq M\rtimes_\sigma G))=\mathcal {QN}(N\rtimes_\sigma G\subseteq M\rtimes_\sigma G)''=\mathcal P_{N,M}\rtimes_\sigma G.$$
\end{cor}

\begin{proof} First we note that since $M$ is abelian then $\mathcal {QN}(N\subseteq M)''=M$ and hence by Theorem \ref{vNalgebra}, $\mathcal P_{N,M}$ is a von Neumann algebra. Thus, by Theorem \ref{quasinormalizer2}, to get our conclusion we only need to show that $\mathcal P_{N,M}\rtimes_\sigma G \subseteq \mathcal {QN}(N\rtimes_\sigma G\subseteq M\rtimes_\sigma G)''$. Since $M$ is abelian then by Proposition \ref{same} we have  $\cK_{N,M}=\oplus_i \cH_i$, where $\cH_i$ are finitely generated, $G$-invariant $N$-modules.  Moreover, since the action of $G$ on $M$ is ergodic, using a standard argument (e.g.\ \cite[Proposition 3.4]{IT20}) we can assume each $\cH_i$ admits a finite $N$-basis, $(\eta_j)_j\subseteq \mathcal P_{N,M}$. In other words, for every $\xi\in \cH_i$ we have $\xi= \sum_j \eta_j E_N(\eta_j^*\xi)=\sum_j  E_N(\eta_j^*\xi)\eta_j$. Thus for every $g\in G$ we have $\sigma_g(\eta_k) = \sum_j \eta_j E_N(\eta_j^*\sigma_g (\eta_k))$ and hence $u_g \eta_k = \sum_j \eta_j E_N(\eta_j^*\sigma_g (\eta_k))u_g$. Since $M$ is abelian this further gives  $ a u_g \eta_k = \sum_j\eta_j E_N(\eta_j^*\sigma_g(\eta_k)) au_g$  for all $g\in G$   and $a\in N$. Thus for every finite combination $x= \sum_{g} a_g u_g\in N\rtimes_{\sigma,{\rm alg}} G$ we have that $x \eta_k=\sum_j \eta_j \phi_{j,k}(x)$ where we denoted $\phi_{j,k}: N\rtimes_{\sigma,{\rm alg}} G \rar N\rtimes_{\sigma,{\rm alg}} G$ the linear map given by $\phi_{j,k}(\sum_g a_gu_g) = \sum_g E_{N}(\eta_j^*\sigma_g(\eta_k)) a_g u_g$. As $N$ is abelian basic calculations show that for every $x=\sum_g a_g u_g\in N\rtimes_{\sigma,{\rm alg}}G$ we have  \begin{equation*}\begin{split}
    &e_{N\rtimes_\sigma G} \eta_j^* x \eta_k e_{N\rtimes_\sigma G}= \sum_g e_{N\rtimes_\sigma G} \eta_j^* a_g u_g \eta_k e_{N\rtimes_\sigma G}\\
    &= \sum_g e_{N\rtimes_\sigma G} \eta_j^* a_g \sigma_g (\eta_k) u_g e_{N\rtimes_\sigma G}=\sum_g e_{N\rtimes_\sigma G} \eta_j^*  \sigma_g (\eta_k) a_gu_g e_{N\rtimes_\sigma G}\\
    &=\sum_g e_{N\rtimes_\sigma G} \eta_j^*  \sigma_g (\eta_k)e_{N\rtimes_\sigma G} a_gu_g  =\sum_g E_{N\rtimes_\sigma G} (\eta_j^*  \sigma_g (\eta_k))e_{N\rtimes_\sigma G} a_gu_g\\
    &=\left(\sum_g E_{N} (\eta_j^*  \sigma_g (\eta_k)) a_gu_g\right) e_{N\rtimes_\sigma G}= \phi_{j,k}(x) e_{N\rtimes_\sigma G}. \end{split}
\end{equation*}
In particular, this relation implies that $\phi_{j,k}$ extends to a WOT-continuous linear map $\phi_{j,k}: N\rtimes_\sigma G \rar  N\rtimes_\sigma G$ which still satisfies $x \eta_k=\sum_j \eta_j \phi_{j,k}(x)$ for all $x\in  N\rtimes_\sigma G$. Therefore we have  $ ( N\rtimes_\sigma G)  \eta_k \subseteq \sum_j\eta_j ( N\rtimes_\sigma G)$. Similarly one can show that $\eta_k( N\rtimes_\sigma G)   \subseteq \sum_j( N\rtimes_\sigma G)\eta_j $ and hence  $\eta_k\in \mathcal {QN}( N\rtimes_\sigma G\subseteq  M\rtimes_\sigma G)$ for all $k$. Since the $\eta_j$'s are $N$-basis for each $\cH_i$ we conclude that $\mathcal P_{N,M}\subseteq \mathcal  {QN}(N\subseteq M)''$, as desired.  \end{proof}

%\textcolor{red}{In the next two results we assume $M$ to have separable predual and $G$ to be countable though these results can be generalized to broader context.}

\begin{thm}\label{intsubalg4} Let $(N\subseteq M, G, \sigma,\tau)$ be a $W^*$-dynamical extension system, where $G$ is an i.c.c. group. Then, there exist $G$-invariant von Neumann subalgebras $Q \subseteq P\subseteq \mathcal P_{N,M}$ satisfying \begin{enumerate}\item $\mathcal {QN}(N\rtimes_\sigma G\subseteq M\rtimes_\sigma G)''=P\rtimes_\sigma G$, and 
\item $\mathcal {N}(N\rtimes_\sigma G\subseteq M\rtimes_\sigma G)''=Q\rtimes_\sigma G$.\end{enumerate}
\end{thm}

\begin{proof} By Theorem \ref{qnormalizer1} we have that $\mathcal {QN}(N\rtimes_\sigma G\subseteq M\rtimes_\sigma G)\subseteq \mathcal {QN}^{(1)}(N\rtimes_\sigma G\subseteq M\rtimes_\sigma G)\subseteq \overline{{\rm span}\,\mathcal P_{N,M}G}^{\|\cdot\|_2}$ and hence $\mathcal {QN}(N\rtimes_\sigma G\subseteq M\rtimes_\sigma G)''\subseteq  \overline{{\rm span}\,\mathcal P_{N,M}G}^{\|\cdot\|_2}$. For simplicity denote by $S:=\mathcal {QN}(N\rtimes_\sigma G\subseteq M\rtimes_\sigma G)''$ and note that $N\rtimes_\sigma G \subseteq S\subseteq M\rtimes_\sigma G$. Using the same argument from the proof of \cite[Theorem 3.10]{ChDa} we see for every $\xi\in \mathcal P_{N,M}$ we have that $E_S(\xi)= E_M \circ E_S(\xi)$. By induction, this  further implies that $E_S(\xi)= (E_S\circ E_M \circ E_S)^k(\xi) $ for every positive integer $k$. Notice that the Jones projections satisfy $(e_Se_Me_S)^k \rightarrow e_S \wedge e_M$ in the SOT topology as $k\rightarrow \infty$. Since $e_S \wedge e_M = e_{M\cap S}$ (\cite[Theorem 4.3]{Sk}), altogether, the prior relations show that  \begin{equation}\label{goodproj}E_S(\xi)=E_{S\cap M}(\xi)\text{ for all }\xi\in \mathcal P_{N,M}.\end{equation} 

\noindent Fix $y\in S$. As $y\in\overline{{\rm span}\,\mathcal P_{N,M}G}^{\|\cdot\|_2}$ one can find $\eta_g\in \mathcal P_{N,M}$ for all $g\in G$ satisfying $y=\sum_{g\in G} \eta_g u_g$. Applying the expectation $E_S$ and using \eqref{goodproj} we further have that $y= E_S(y)= \sum_{g\in G} E_S(\eta _g)u_g = \sum_{g\in G} E_{M\cap S }(\eta_g)u_g$. In particular, this shows  $S\subseteq \overline{{\rm span} (M\cap S) G}^{\|\cdot\|_2}$. \vskip 0.05in \noindent Since $N\subseteq M\cap S \subseteq M$ is a $G$-invariant intermediate von Neumann algebra we have that ${\rm span} (M\cap S) G\subseteq (M\cap S)\rtimes_\sigma G$; hence $S\subseteq \overline{(M\cap S)\rtimes_\sigma G}^{\|\cdot\|_2}$. Since we canonically have  $(M\cap S)\rtimes_\sigma G \subseteq S$, then we conclude that $S= (M\cap S)\rtimes_\sigma G$. Letting $P=M\cap S$, the previous relations also show that $ P\subseteq \mathcal P_{N,M}$, finishing part (1) of the conclusion. The second part follows similarly and the details are left to the reader.\end{proof}

\noindent The prior results yield the following generalization of \cite[Corollary 3.14]{ChDa}.
\begin{cor}  Let $(N\subseteq M, G, \sigma,\tau)$ be a $W^*$-dynamical extension system, where $G$ is an i.c.c. group.  Then, for every intermediate von Neumann subalgebra $N\rtimes_\sigma G \subseteq Q\subseteq M\rtimes_\sigma G$ which admit a finite left and right Pimsner-Popa bases over $N \rtimes_\sigma G$, there is a $G$-invariant intermediate von Neumann algebra $N\subseteq P\subseteq M$ which admits finite left and right Pimsner-Popa bases over $N$ such that $Q=P\rtimes_\sigma G$. Moreover, if $N=\mathbb C1$ then $P$ is finite-dimensional.

\end{cor}
\begin{proof}  Since $Q$ admits a finite left Pimsner-Popa basis and also a finite right Pimsner-Popa basis over $N\rtimes_\sigma G$, then $Q \subseteq \mathcal {QN}(N\rtimes_\sigma G, M\rtimes_\sigma G)''$. Thus, using Theorem \ref{intsubalg4}, one can find a $G$-invariant von Neumann subalgebra  $N\subseteq R \subseteq \mathcal P_{N,M}$ so that $N\rtimes_\sigma  G\subseteq Q\subseteq R \rtimes_{\sigma}G$. Further, by \cite[Theorem 3.10]{ChDa}, one can find a $G$-invariant von Neumann subalgebra $N\subseteq P \subseteq R\subseteq M$ such that $Q = P\rtimes_\sigma G$. Next notice the following inclusions diagram

\begin{equation*}
\begin{array}{ccc} N\rtimes_\sigma G & \subseteq & P\rtimes_\sigma G\\
\cup  & & \cup  \\
N&\subseteq & P  
\end{array}
\end{equation*}
is a non-degenerate commuting square in the sense of Popa, \cite{Po94}. Since $N\rtimes_\sigma G \subseteq P\rtimes_\sigma G$ admits a finite left (right) Pimsner-Popa basis then  using \cite[Proposition 1.1.5 (iii)]{Po94} we conclude that  $N\subseteq P$ also has a finite left (right) Pimsner-Popa basis. Thus, if $N = \mathbb C 1$, then $P$ is finite-dimensional.
\end{proof}

\begin{remark} In connection with Theorem \ref{intsubalg4} it would be interesting to know if there are examples of $W^*$-dynamical extension  systems $(N\subseteq M, G, \sigma,\tau)$ with $G$ i.c.c.\ for which  $vN(\mathcal{QN}^{(1)}(N\rtimes _\sigma G, M\rtimes_\sigma G))$ is not of the form $Q \rtimes_\sigma G $ for an intermediate von Neumann subalgebra $N\subseteq Q \subseteq M$. 
\end{remark}

%It is also known that the i.c.c.\ condition on $G$ is necessary, \cite{Ji}. 
%\begin{corollary} Let $ \G\ca(\cN\subseteq \cM)$ be a tracial action on an inclusion of von Neumann algebras and let $\cN \subseteq \cP \subseteq \cM$ be an intermediate $ \G$-invariant  von Neumann algebra such that the action $\G\ca (
%\cN\subseteq \cP) $ is maximal compact. Then $\mathscr{QN}_{\cM\rtimes \G}(\cN\rtimes \G)''=\mathscr{QN}^{(1)}_{\cM\rtimes \G}(\cN\rtimes \G)''=\cP\rtimes \G$.

%\end{corollary}

%\begin{proof} Follows directly from Theorems \ref{qnormalizer1} and \ref{quasinormalizer2}.\end{proof}

\subsection{Von Neumann algebraic descriptions of the Furstenberg-Zimmer tower}\label{FZ} Using the prior results (e.g.\ Theorems \ref{qnormalizer1}-\ref{quasinormalizer2}) we show that the  Furstenberg and Zimmer structural theorems for action of groups on probability spaces, \cite{Fu77,Zi76} can be described solely in von Neumann algebraic terms,  using the language of one-sided quasinormlizing algebras and von Neumann subalgebras generated by the relatively almost periodic elements (see Corollary \ref{fztower3}). In particular, we recover an unpublished result of J. Peterson and the third author \cite{CP11}; see also \cite[Theorem 2.5]{ChDa}.  
\vskip 0.06in 
\noindent More generally, using the von Neumann algebraic framework we are able to introduce various types of Furstenberg-Zimmer structural towers even in the non-commutative case. We start with the following result for general $W^*$-dynamical extension systems.

\begin{thm}\label{fztower1} Let $(N\subseteq M, G, \sigma,\tau)$ be a $W^*$-dynamical extension system. Then one can find  an ordinal $\alpha$ and a $G$-invariant von Neumann subalgebra $N \subseteq Q_\beta \subseteq M$ for every $\beta\leqslant\alpha$ satisfying the following properties:  \begin{enumerate}
\item [1.] For all $\beta\leqslant  \beta' \leqslant  \alpha$ we have  $N=Q_o\subseteq   Q_\beta \subseteq Q_{\beta'} \subseteq M$;
\item [2.] For every successor ordinal $\beta+1 \leqslant  \alpha$ we have $Q_{\beta+1}= \mathcal P_{Q_\beta, M}''$ and  
\begin{align*}
 vN(\mathcal {QN}^{(1)}(Q_\beta \rtimes_\sigma G\subseteq M\rtimes_\sigma G))\subseteq Q_{\beta+1}\rtimes_\sigma G.   
\end{align*}
\item [3.] For every limit ordinal $\beta \leqslant \alpha$ we have $\overline{\cup_{\gamma<\beta} Q_\gamma}^{\rm SOT}  = Q_\beta$ and $\overline{\cup_{\gamma<\beta} Q_\gamma \rtimes_\sigma G}^{\rm SOT}  = Q_\beta\rtimes_\sigma G$.
\item [4.] There are nets $(g_\lambda)_\lambda\subseteq G$ and $(u_\lambda)_\lambda\subseteq \mathcal U(Q_\alpha)$ such that for every $x, y \in M \ominus  Q_\alpha$ we have  $$\lim_{\lambda} \|E_ {Q_\alpha} (x u_\lambda\sigma_{g_\lambda}(y))\|_2= 0.$$
\end{enumerate}
\end{thm}

\begin{proof} We will define inductively the tower of von Neumann algebras $Q_\beta$, for all ordinals $\beta$, as follows. First let $Q_0=N$. Now, if $\beta$ is a successor ordinal then let $Q_\beta:= \mathcal P_{Q_{\beta-1},M}''\subseteq M$, the von Neumann algebra generated by $\mathcal P_{Q_{\beta-1},M}$. Notice that 
by Proposition \ref{subspace} this is $G$-invariant and satisfies $Q_{\beta-1}\subseteq Q_\beta$. Moreover by Theorem \ref{qnormalizer1}, in this case we also have that $vN(\mathcal {QN}^{(1)}(Q_\beta \rtimes_\sigma G\subseteq M\rtimes_\sigma G))\subseteq Q_{\beta+1}\rtimes_\sigma G$. Altogether, these give part 2 above.
If $\beta$ is a limit ordinal then let $Q_\beta:=\overline{\cup_{\gamma<\beta} Q_{\gamma}}^{\rm SOT}$. In this case, since all $Q_\gamma$ are $G$-invariant one can easily see that so is $Q_\beta$.  In particular, we also have that $\overline{\cup_{\gamma<\beta} Q_\gamma \rtimes_\sigma G}^{\rm SOT}  = Q_\beta\rtimes_\sigma G$. Now let $\alpha$ be the first ordinal where the chain $(Q_\beta)_\beta$ stabilizes, i.e, $Q_\alpha=Q_{\alpha+1}$. Altogether, the previous relations show the tower $(Q_\beta)_{\beta\leqslant \alpha}$ satisfies conditions 1.-3. in the statement. Moreover, since $\alpha$ stabilizes the tower we have that $Q_\alpha= \mathcal P_{Q_\alpha,M}''$ and by Theorem \ref{qnormalizer1} we get that $vN(\mathcal {QN}^{(1)}(Q_\alpha \rtimes_\sigma G\subseteq M\rtimes_\sigma G))= Q_{\alpha}\rtimes_\sigma G$. However, using the relative WAHP in the same way as in the beginning of the proof of Theorem \ref{quasinormalizer2}, this further gives 4.   \end{proof}

\begin{thm}\label{fztower2} Let $(N\subseteq M, G, \sigma,\tau)$ be an ergodic $W^*$-dynamical extension system where $M$ is abelian. Then one can find  an ordinal $\alpha$ and a $G$-invariant von Neumann subalgebra $N \subseteq Q_\beta \subseteq M$ for every $\beta\leqslant\alpha$ satisfying the following properties:  \begin{enumerate}
\item [1'.] For all $\beta\leqslant  \beta' \leqslant  \alpha$ we have the following inclusions of von Neumann algebras $N=Q_o\subseteq   Q_\beta \subseteq Q_{\beta'} \subseteq M$;
\item [2'.] For every successor ordinal $\beta+1 \leqslant  \alpha$ we have     $vN(\mathcal {QN}^{(1)}(Q_\beta \rtimes_\sigma G\subseteq M\rtimes_\sigma G))=\mathcal{QN}(Q_\beta \rtimes_\sigma G\subseteq M\rtimes_\sigma G)''=Q_{\beta+1}\rtimes_\sigma G$. Moreover, there is a net $(g^\beta_\lambda)_\lambda\subseteq G$ such that for every $x,y\in M\ominus Q_{\beta+1}$ we have 
$$\lim_{\lambda} \|E_{Q_\beta}(x\sigma_{g^\beta_\lambda}(y))\|_2=0.$$
\item [3'.] For every limit ordinal $\beta \leqslant \alpha$ we have $\overline{\cup_{\gamma<\beta} Q_\gamma}^{\rm SOT}  = Q_\beta$ and also $$\overline{\cup_{\gamma<\beta} Q_\gamma \rtimes_\sigma G}^{\rm SOT}  = Q_\beta\rtimes_\sigma G.$$
\item [4'.] There is a net $(g_\lambda)_\lambda\subseteq G$ such that for every $x, y \in M \ominus  Q_\alpha$ we have $$\lim_{\lambda} \|E_ {Q_\alpha} (x \sigma_{g_\lambda}(y))\|_2= 0.$$
\end{enumerate}

\end{thm}
\begin{proof} From Theorem \ref{fztower1} one can find a tower $(Q_\beta)_{\beta\leqslant \alpha}$ of von Neumann subalgebras $N\subseteq Q_\beta\subseteq M$ satisfying the properties 1.-4. listed in the conclusion. We will show these imply our statement. 
\vskip 0.07in 
\noindent Fix $\beta+1\leqslant \alpha$ any successor ordinal. Since $M$ is abelian we have $\mathcal {QN}(Q_\beta\subseteq M)''= M$  and using Theorem \ref{vNalgebra} we get that $\mathcal P_{Q_\beta,M}\subseteq M$ is a von Neumann subalgebra. Thus  $Q_{\beta+1}= \mathcal P_{Q_\beta,M}$  and by Theorem \ref{quasinormalizer2} and Corollary \ref{quasinorm-compact} we conclude that $vN(\mathcal {QN}^{(1)}(Q_\beta \rtimes_\sigma G\subseteq M\rtimes_\sigma G))=\mathcal {QN}(Q_\beta \rtimes_\sigma G\subseteq M\rtimes_\sigma G)''=Q_{\beta+1}\rtimes_\sigma G$. This gives the first part of 2'. The moreover part of 2' follows from Lemma \ref{lemma:WeaklyMixingSingular} and Theorem \ref{rwahp}. Finally, since $M$ is abelian and $u_\lambda$ is a unitary, in 4.\ we have that $\|E_{Q_\alpha}(xu_\lambda\sigma_{g_\lambda}(y))\|_2=\|u_\lambda E_{Q_\alpha}(x\sigma_{g_\lambda}(y))\|_2= \|E_{Q_\alpha}(x\sigma_{g_\lambda}(y))\|_2 $ which gives 4'.\end{proof}\vskip 0.06in 
\noindent In particular, the previous result yields the following picture of the classical Furstenberg-Zimmer tower as a sequence of iterated quasinormalizing algebras in the context of p.m.p. actions of countable groups and separable $\sigma$-algebras. This result was originally obtained by J. Peterson and the third author in the unpublished work \cite{CP11}.

\begin{cor}[\cite{CP11}]\label{fztower3} Let $G \ca X$ be an pmp ergodic action on a standard probability space $X$ and let $(G \ca X_\beta)_{\beta\leqslant \alpha}$ be the corresponding Furstenberg-Zimmer tower. Let $M = L^\infty(X)\rtimes_\sigma G$ and $M_\beta = L^\infty(X_\beta)\rtimes_\sigma G$ be the corresponding crossed product von Neumann algebras.  Then the following hold:  \begin{enumerate}
\item [1.] For all $\beta\leqslant  \beta' \leqslant  \alpha$ we have the following inclusions $L(G)=M_o\subseteq   M_\beta \subseteq M_{\beta'} \subseteq M_\alpha \subseteq M$;
\item [2.] For every successor ordinal $\beta+1 \leqslant  \alpha$ we have that    $vN(\mathcal {QN}^{(1)}(M_\beta\subseteq M))=\mathcal{QN}(M_\beta\subseteq M )''=M_{\beta+1}$. Moreover, there is a sequence $(g^\beta_n)_n\subseteq G$ such that for every $x,y\in L^\infty(X)\ominus L^\infty(X_{\beta+1})$ we have 
$$\lim_{n\rightarrow \infty} \|E_{L^\infty(X_\beta)}(x\sigma_{g^\beta_n}(y))\|_2=0.$$
\item [3.] For every limit ordinal $\beta \leqslant \alpha$ we have $\overline{\cup_{\gamma<\beta} L^\infty(Y_\gamma)}^{\rm SOT}  = L^\infty(Y_\beta)$ and also $\overline{\cup_{\gamma<\beta} M_\gamma}^{\rm SOT}  = M_\beta$.
\item [4.] There is an infinite sequence $(g_n)_n\subseteq G$ such that for every $x, y \in L^\infty(X) \ominus  L^\infty(Y_\alpha)$ we have  $$\lim_{n\rar \infty} \|E_ {L^\infty(Y_\alpha)} (x \sigma_{g_n}(y))\|_2= 0.$$
\end{enumerate}
\end{cor}

\begin{rem}
An uncountable version of Furstenberg-Zimmer structure theorem has been recently obtained in \cite[Theorem 6.5]{Ja}. Theorem \ref{fztower2} recovers this tower through the description of iterated quasinormalizing algebras. Thus, Corollary \ref{fztower3} extends to the general case.
\end{rem}

\noindent We end with the following result, whose proof is very similar with the first result in this section.

\begin{thm}\label{fztower4} Let $(N\subseteq M, G, \sigma,\tau)$ be a $W^*$-dynamical extension system. Then one can find  an ordinal $\alpha$ and a $G$-invariant von Neumann subalgebra $N \subseteq Q_\beta \subseteq M$ for every $\beta\leqslant\alpha$ satisfying the following properties:  \begin{enumerate}
\item [1.] For all $\beta\leqslant  \beta' \leqslant  \alpha$ we have the following inclusions of von Neumann algebras $N=Q_o\subseteq   Q_\beta \subseteq Q_{\beta'} \subseteq M$;
\item [2.] For every successor ordinal $\beta+1 \leqslant  \alpha$ we have $Q_\beta\subseteq Q_{\beta+1}\subseteq \mathcal {QN}(Q_\beta\subseteq M)''$ and  $vN(\mathcal {QN}^{(1)}(Q_\beta \rtimes_\sigma G\subseteq M\rtimes_\sigma G))=Q_{\beta+1}\rtimes_\sigma G$. Moreover, there are nets $(g^\beta_\lambda)_\lambda\subseteq G$ and $(u_\lambda)_\lambda\subseteq \mathcal U (Q_\beta)$ so that for every $x,y\in M\ominus Q_{\beta+1}$,
$$\lim_{\lambda} \|E_{Q_\beta}(x u_\lambda\sigma_{g^\beta_\lambda}(y))\|_2=0.$$
\item [3.] For every limit ordinal $\beta \leqslant \alpha$ we have $\overline{\cup_{\gamma<\beta} Q_\gamma}^{\rm SOT}  = Q_\beta$ and  $\overline{\cup_{\gamma<\beta} Q_\gamma \rtimes_\sigma G}^{\rm SOT}  = Q_\beta\rtimes_\sigma G$.
\item [4.] Either $Q_\alpha= \mathcal {QN}(Q_\alpha\subseteq M)''$ or, there are nets $(g_\lambda)_\lambda\subseteq G$ and $(u_\lambda)_\lambda\subseteq \mathcal U(Q_\alpha)$ such that for every $x, y \in M \ominus  Q_\alpha$ we have $$\lim_{\lambda} \|E_ {Q_\alpha} (x u_\lambda\sigma_{g_\lambda}(y))\|_2= 0.$$
\end{enumerate}
\end{thm}

\begin{proof}We will define inductively the tower of von Neumann algebras $Q_\beta$, for all ordinals $\beta$, as follows. First let $Q_0=N$. Now, if $\beta$ is a successor ordinal then using Theorem \ref{quasinormalizer2} we define  $Q_{\beta-1} \subseteq Q_\beta\subseteq  \mathcal {QN}(Q_{\beta-1}\subseteq M)''$ as the unique $G$-invariant von Neumann algebra such that $vN(\mathcal {QN}^{(1)}(Q_{\beta-1} \rtimes_\sigma G\subseteq M\rtimes_\sigma G))=Q_{\beta}\rtimes_\sigma G$. If $\beta$ is a limit ordinal then let $Q_\beta:=\overline{\cup_{\gamma<\beta} Q_{\gamma}}^{\rm SOT}$. In this case, since all $Q_\gamma$ are $G$-invariant one can easily see that so is $Q_\beta$.  In particular, we also have that $\overline{\cup_{\gamma<\beta} Q_\gamma \rtimes_\sigma G}^{\rm SOT}  = Q_\beta\rtimes_\sigma G$. Now let $\alpha$ be the first ordinal where the chain $(Q_\beta)_\beta$ stabilizes, i.e., $Q_\alpha=Q_{\alpha+1}$. Altogether, the previous relations show that the tower $(Q_\beta)_{\beta\leqslant \alpha}$ satisfies conditions 1.-3. in the statement.  As before, the moreover part of 2 follows from Lemma \ref{lemma:WeaklyMixingSingular} and Theorem \ref{rwahp}. Finally, since $\alpha$ stabilizes the tower we have either $Q_\alpha= \mathcal {QN}(Q_\alpha\subseteq M)''$ or $vN(\mathcal {QN}^{(1)}(Q_\alpha \rtimes_\sigma G\subseteq M\rtimes_\sigma G))= Q_{\alpha}\rtimes_\sigma G$. However, using the relative WAHP from Theorem \ref{rwahp} in the same way as in the beginning of the proof of Theorem \ref{quasinormalizer2}, we have 4. in the statement.
\end{proof}

Finally using Theorem \ref{intsubalg4} we have the following.

\begin{thm}\label{fztower5}  Let $(N\subseteq M, G, \sigma,\tau)$ be a $W^*$-dynamical extension system, where $G$ is i.c.c. Then one can find  an ordinal $\alpha$ and a $G$-invariant von Neumann subalgebra $N \subseteq Q_\beta \subseteq M$ for every $\beta\leqslant\alpha$ satisfying the following properties:  \begin{enumerate}
\item [1.] For all $\beta\leqslant  \beta' \leqslant  \alpha$ we have the following inclusions of von Neumann algebras $N=Q_o\subseteq   Q_\beta \subseteq Q_{\beta'} \subseteq M$;
\item [2.] For every successor ordinal $\beta+1 \leqslant  \alpha$ we have   $\mathcal {QN}(Q_\beta \rtimes_\sigma G\subseteq M\rtimes _\sigma G)''=Q_{\beta+1}\rtimes_\sigma G$.

\item [3.] For every limit ordinal $\beta \leqslant \alpha$ we have $\overline{\cup_{\gamma<\beta} Q_\gamma}^{\rm SOT}  = Q_\beta$ and  $\overline{\cup_{\gamma<\beta} Q_\gamma \rtimes_\sigma G}^{\rm SOT}  = Q_\beta\rtimes_\sigma G$.
\item [4.] $Q_\alpha\rtimes_\sigma G= \mathcal {QN}_{M\rtimes_\sigma G}(Q_\alpha\rtimes_\sigma G)''$.% or, there are sequences $(g_n)_n\subseteq \Gamma$ and $(u_n)\subseteq \mathscr U(\cQ_\alpha)$ such that for every $x, y \in \cM \ominus  \cQ_\alpha$ we have $$\lim_{n\rar \infty} \|E_ {\cQ_\alpha} (x u_n\sigma_{g_n}(y))\|_2= 0.$$
\end{enumerate}

\end{thm}

\vskip 0.04in
\noindent In connection to the results presented in this section we also want to mention in closing the following bold open problem on intermediate von Neumann subalgebras inside group measure space von Neumann algebras.     

\begin{opr} Let $G$ be a countable i.c.c.\ group and let $G \curvearrowright^{\alpha} (X,\mu)$ be a free, ergodic, p.m.p.\ action. Is it true that for every intermediate subalgebra $L(G)\subseteq N \subseteq L^{\infty}(X, \mu)\rtimes_\alpha G$ one can find a factor $G \curvearrowright^\beta (Y,\nu)$ of $G \curvearrowright^{\alpha} (X,\mu)$ such that $N= L^{\infty}(Y, \nu)\rtimes_\beta G$.
\end{opr}

\noindent As already mentioned in the prior sections, this problem has been answered positively when $\alpha$ is a compact ergodic action (see \cite{ChDa,JS19}). Unfortunately, very little is known beyond this case. For example, is this still true when the action $\alpha$ is a distal tower of length at least two?

\section{Approximation properties of the inclusion $L(G) \subseteq M \rtimes_\alpha G$} \label{section:approximation}

\noindent In this section we consider $W^*$-dynamical systems in which the underlying von Neumann algebra is tracial. The setting will consist of a discrete group $G$, acting by trace-preserving automorphisms on a finite von Neumann algebra with a fixed normal, faithful trace $\tau$.  Using Theorem \ref{QuasinormalizerCompactSystemTheorem}, we will relate the analytical structure of the inclusion $L(G) \subseteq M \rtimes_\alpha G$ to the dynamical properties of the action $\alpha$.  Recall the following definition, due to Popa \cite{Po}.

\begin{definition}\label{relative HAP} The finite von Neumann algebra $(N, \tau)$ is said to have \emph{property} $H$ $($or the \emph{Haagerup approximation property}$)$ \emph{relative to the von Neumann subalgebra} $B \subseteq N$ if there is a net of normal, $B$-bimodular, completely positive maps $\{\Phi_\lambda:N \rightarrow N\}_{\lambda\in\Lambda}$ with the following properties:
\begin{itemize}
\item[(i)] $\tau\circ \Phi_\lambda\leqslant \tau$,  $\lambda\in\Lambda$.
    \item[(ii)] For each $z \in N$, $\lim_\lambda\norm{\Phi_\lambda (z)-z}_2 = 0$.
    \item[(iii)] Each induced operator $T_{\Phi_\lambda}$ on $L^2(N,\tau)$ has the property that for any $\eps>0,$ there is a projection $p \in \ip{N}{e_B}$ with finite trace such that $\norm{T_{\Phi_\lambda}(1-p)}<\eps.$
\end{itemize}
\end{definition}
\noindent The third condition above may be interpreted as ``compactness relative to $B$" and, in fact, implies that the $T_{\Phi_\lambda}$ are compact operators when the subalgebra $B$ is $\Cplx 1$.  Standard examples of inclusions with relative property $H$ include those of the form $B \subseteq B \overline{\otimes}P$ , where $P$ is a finite von Neumann algebra with the Haagerup approximation property, and crossed product inclusions $B \subseteq B \rtimes_\alpha \Gamma$, where $\Gamma$ is a discrete group with the Haagerup approximation property \cite{Po}. Ioana \cite{Io} proved that if $\Gamma$ is a discrete group acting by measure-preserving transformations $\sigma_\gamma$, $\gamma \in \Gamma$, on a probability space $(X,\mu)$, then the crossed product $L^\infty(X,\mu) \rtimes_\sigma \Gamma$ has property $H$ relative to the subalgebra $L(\Gamma)$ if and only if the action $\sigma$ is compact.  We extend Ioana's result to the case of a trace-preserving action of a discrete group on a general finite von Neumann algebra.  Our main result in this section -- which we combine with the results of the previous section for the reader's convenience -- is stated below as Theorem \ref{RelativeHAPTheorem}. In order to prove this result we will need several lemmas.

\nexer{\noindent The remainder of this section will concern the proof of Theorem \ref{RelativeHAPTheorem}.  Note that the first three conditions are equivalent by Theorem \ref{QuasinormalizerCompactSystemTheorem} and the discussion following it, and that condition (iv) implies condition (ii),  by Proposition 3.4 in \cite{Po}.  Thus, the proof will be complete when we show that conditions (i)-(iii) imply (iv).  Our strategy will be to construct a net of normal, completely positive maps $\phi_\lambda$ on $M$ which 
\begin{itemize}
    \item approximate the identity map on $M$ pointwise in $\norm{\cdot}_2$,
    \item  extend to compact operators on $L^2(M)$, and
    \item  commute with the automorphisms $\alpha_g$, $g \in G$.  
\end{itemize} The last condition will ensure that the maps lift to normal, completely positive $L(G)$-bimodule maps $\Phi_\lambda$ on $M \rtimes_\alpha G$, and the first two conditions will persist in such a way that the lifted maps satisfy conditions (ii) and (iii) of Definition \ref{relative HAP}.  The key to this step in the argument is the following general structural result for invariant completely positive maps, which we prove for completeness. }

\begin{lem} \label{cp extension} Let $\alpha$ be a trace-preserving action of a discrete group $G$ on a finite von Neumann algebra $M$.  Suppose that $\phi$ is a normal, completely positive $($resp., completely bounded$)$ map on $M$ such that $\phi \circ \alpha_g = \alpha_g \circ \phi$ for all $g \in G$.  Then there is a unique normal, completely positive $($resp., completely bounded$)$ extension $\Phi: M\rtimes_\alpha G \rightarrow M\rtimes_\alpha G$ of $\phi$ satisfying
\[ \Phi(u_gx u_h) = u_g\phi(x) u_h\]
for $x \in M$, $g,h \in G$.  In particular, $\Phi$ is an $L(G)$-bimodule map.  
\end{lem}
\begin{proof} 
The operators $\pi(x)$, $x\in M$, and $u_g$, $g\in G$, that generate the crossed product are easily checked to commute with $M'\otimes I$ and so $M\rtimes_\alpha G$ is a von Neumann subalgebra of $M\overline{\otimes}\mathbf{B}(\ell^2(G))$. In both cases $\phi:M\to M$ is completely bounded and so extends to a normal map $\Phi=\phi\otimes I: M\overline\otimes\mathbf{B}(\ell^2(G))\to M\overline\otimes\mathbf{B}(\ell^2(G))$. This is completely positive (resp. completely bounded) when $\phi$ is completely positive (resp. completely bounded). 

In $\mathbf{B}(\ell^2(G))$, let $e_{h,k}$ denote the rank-one matrix unit that takes $\delta_k$ to $\delta_h$ for $h,\,k\in G$. A simple calculation then gives $\lambda_ge_{h,k}=e_{gh,k}$ and $e_{h,k}\lambda_g=e_{h,g^{-1}k}$ for $g,\,h,\,k\in G$, where $\lambda$ is the left regular representation of $G$ on $\ell^2(G)$. Now
\[\pi(x)=\sum_{h\in G}\alpha_{h^{-1}}(x)\otimes e_{h,h},\ \ \ x\in M,\]
so
\[\pi(x)u_g=\sum_{h\in G} \alpha_{h^{-1}}(x)\otimes e_{h,g^{-1}h},\ \ \ x\in M,\ \ \ g\in G.\]
Thus
\begin{align*}\Phi(\pi(x)u_g)&=\sum_{h\in G}\phi(\alpha_{h^{-1}}(x))\otimes e_{h,g^{-1}h}\\
&=\sum_{h\in G}\alpha_{h^{-1}}(\phi(x))\otimes e_{h,g^{-1}h}=\pi(\phi(x))u_g,\ \ \ x\in M,\ \ \ g\in G.
\end{align*}
A similar calculation shows that 
\[\Phi(u_g\pi(x))=u_g\pi(\phi(x)),\ \ \ x\in M,\ \ \ g\in G,\]
so the normality of $\Phi$ shows that this map is an $L(G)$-bimodular extension of $\phi$ to $M\rtimes_\alpha G$. Normality and bimodularity clearly imply uniqueness of this extension.
\end{proof}

\vskip 0.03in 
\noindent The following lemma will be necessary to establish $w^*$-continuity of certain maps $\phi_\lambda$ that will appear below.  We will then prove a lemma that will allow us to pass from the case of a separable predual to the general situation in Theorem \ref{RelativeHAPTheorem}. Below, the strong$^*$ topology will be denoted by SOT$^*$.

\begin{lem}\label{uniformbound}
Let $(M,\tau)$ be a finite von Neumann algebra and let $N$ be a finite-dimensional von Neumann algebra. Let $\Gamma \subseteq {\mathbf{B}}(L^2(M))$ be an SOT$^*$-compact group of unitaries such that $\Ad (\gamma)(M)\subseteq M$ for $\gamma\in \Gamma$. Let $\{x_\delta\}_{\delta\in\Delta}$ be a uniformly bounded net converging to 0 in the $w^*$-topology, and let $T:M\to N$ be a $w^*$-continuous bounded map. Given $\eps >0$, there exists $\delta_0$ so that 
\begin{equation*}
\|T(\gamma^* x_\delta \gamma)\| < \eps,\ \ \ \ \text{ for }\delta\geqslant \delta_0,\ \ \ \  \gamma\in \Gamma.
   \end{equation*}
   \end{lem}
   
   \begin{proof}
   Since $N$ is finite-dimensional, it is a subalgebra of a matrix factor so we may assume that $N= \mathbb{M}_k$ for some integer $k$. By scaling, we may assume that $\|x_\delta\|\leqslant 1$ for all $\delta$, and that $\|T\|\leqslant 1$. Now fix $\beta >0$ so small that $3k^2\beta <\eps$.
\vskip 0.07in 
\noindent For $1\leqslant i,j\leqslant k$, let $\theta_{i,j}(x)$ be the $(i,j)$ entry of $T(x)$ for $x\in M$. Each $\theta_{i,j} $ is a $w^*$-continuous contractive linear functional on $M$, so there exist vectors $\xi_{i,j},\eta_{i,j}\in L^2(M)$ so that $\|\xi_{i,j}\|_2, \|\eta_{i,j}\|_2\leqslant 1$ and
\begin{equation*}
\theta_{i,j}(x)=\langle x\xi_{i,j},\eta_{i,j}\rangle,\ \ \ \ x\in M,\ \ \ \ 1\leqslant i,j\leqslant k.
\end{equation*}  
For $\gamma\in \Gamma$ and $x\in M$,
\begin{equation}\label{theta}
\theta_{i,j}(\gamma^* x\gamma)=\langle x\gamma\xi_{i,j},\gamma \eta_{i,j}\rangle.
\end{equation}
By SOT$^*$-compactness of $\Gamma$, the closure of 
\[\{\gamma\xi_{i,j}, \gamma\eta_{i,j}: \gamma\in \Gamma,\ 1\leqslant i,j\leqslant k\}\]
is  norm-compact  in the unit ball of $L^2(M)$, so we may choose a finite $\beta$-net $\{\omega_1,\ldots,\omega_r\}$ of vectors for this set. Now choose $\delta_0$ so that 
\begin{equation}\label{first}
|\langle x_\delta\omega_s, \omega_t\rangle |<\beta, \ \ \ \ 1\leqslant s,t\leqslant r,\ \ \ \ \delta\geqslant\delta_0,
\end{equation}
which is possible since $w^*-\lim_\delta x_\delta =0$.
\vskip 0.07in 
\noindent Now fix $\gamma\in \Gamma$ and a pair of integers $(i,j)$. Then choose $\omega_s,\omega_t$ so that
\begin{equation}\label{second}
\|\gamma\xi_{i,j}-\omega_s\|<\beta,\ \ \ \ \|\gamma\eta_{i,j}-\omega_t\|<\beta.
\end{equation}
Applying \eqref{first} and \eqref{second} to the equation \eqref{theta},
\begin{align*}
\theta_{i,j}(\gamma^* x_\delta\gamma)&=\langle x_\delta \gamma \xi_{i,j}, \gamma \eta_{i,j}\rangle \notag \\
&=\langle x_\delta (\gamma \xi_{i,j} -\omega_s),\gamma\eta_{i,j}\rangle +
\langle x_\delta \omega_s, \gamma \eta_{i,j}-\omega_t\rangle +
\langle x_\delta \omega_s, \omega_t\rangle
\end{align*}
leads to the estimate $|\theta_{i,j}(\gamma^* x_\delta \gamma)| <3\beta$ for $\delta \geqslant\delta_0$ and independent of the choices of $\gamma\in \Gamma$ and the pair $(i,j)$.
By summing the $k^2$ matrix entries, we obtain
\begin{equation*}
\|T(\gamma^*x_\delta\gamma)\|<3k^2\beta<\eps,\ \ \ \ \ \ \ \delta\geqslant \delta_0,\ \ \gamma\in \Gamma,
\end{equation*}
as required.\end{proof}

\begin{lem}\label{sepsubalg}
Let $M$ be a finite von Neumann algebra with a faithful normal trace $\tau$ and let a discrete group $G$ have a compact action $\alpha$ on $M$. For each finite subset $F\subseteq M$, there exists a unital norm-separable C$^*$-subalgebra $B_F\subseteq M$ with weak closure $M_F$ satisfying the following properties:
\begin{enumerate}
\item $F\subseteq B_F$.
\medskip
\item If $F_1\subseteq F_2$, then $B_{F_1}\subseteq B_{F_2}$ and $M_{F_1}\subseteq M_{F_2}$.
\medskip
\item  Each  $M_F$ has a norm-separable predual.
\medskip
\item Each $M_F$ is $\alpha$-invariant.
\end{enumerate}
\end{lem}

\begin{proof}
These algebras will be constructed by induction on the cardinality of the finite subsets, so we begin by constructing $B_F$ and $M_F$ for a fixed but arbitrary one-point set $F$. 
We will define inductively an increasing sequence $A_1\subseteq A_2\subseteq\ldots$ of separable unital C$^*$-algebras
and an increasing sequence $\mathcal{H}_1\subseteq \mathcal{H}_2\subseteq\ldots$ of separable closed subspaces of $L^2(M)$ with the following properties:
\begin{itemize}
 \item[{\rm{(i)}}]
 $F\subseteq A_1$ and $\Omega\in \mathcal{H}_1$.
 \medskip
 \item[{\rm{(ii)}}]
 $\mathcal{H}_n$ is an invariant subspace for $A_n$, $n\geqslant 1$.
 \medskip
 
 \item[{\rm{(iii)}}]
 For $a\in A_n$ and $g\in G$, $\alpha_g(a)\in \overline{A_{n+1}}^{\rm SOT}$.
 
 \end{itemize}
 
 To begin the induction, let $A_1$ be any separable unital C$^*$-algebra containing $F$ and let $\mathcal{H}_1=\overline{\{a\Omega:a\in A_1\}}^{\|\cdot\|_2}$.
 Now suppose that $A_n$ and $\mathcal{H}_n$ have been constructed. Fix a countable norm dense set 
 $\{a_1,a_2,\ldots\}\subseteq A_n$.

 Since the action is compact, we may choose a sequence $\{a_{i1},a_{i2},\ldots\}\subseteq\,$ Orb$(a_i)$ which is $\|\cdot\|_2$-dense in Orb$(a_i)$ for $i\geqslant 1$. We now define \[A_{n+1} = \mathrm{C}^*(A_n,\{a_{ij}\}_{i,j=1}^\infty)\]
 and 
 \[ \mathcal{H}_{n+1}=\overline{{\mathrm{span}}}^{\|\cdot\|_2}\{a\xi:a\in A_{n+1},\ \xi\in \mathcal{H}_n\}.\]
 Then (i) and (ii) are clearly satisfied and it remains to verify (iii).
 
 Fix $a\in A_n$ with $\|a\|<1$. Choose $\delta$ to satisfy $0<\delta<1-\|a\|$ and choose $a_i$ so that $\|a-a_i\|<\delta$. Then $\|a_i\|<1$. Fix $g\in G$. Then $\|\alpha_g(a)-\alpha_g(a_i)\|<\delta$ and there exists $a_{ij}\in  \text{Orb}(a_i)$
 so that $\|\alpha_g(a_i)-a_{ij}\|_2<\delta$. Thus $\|\alpha_g(a)-a_{ij}\|_2<2\delta$. Replacing $\delta$ by $\delta 2^{-m}$ for $m=1,2,3,\ldots$ successively, we obtain a sequence $\{b_m\}_{m=1}^\infty$ from $\{\mathrm{Orb}(a_i):i\geqslant 1\}$ so that $\lim_{m\to\infty}\|\alpha_g(a)-b_m\|_2=0$, and each $b_m$ is an $a_{ij}\in A_{n+1}$. Since $\{b_m\}_{m=1}^\infty$ is uniformly bounded, showing that this sequence converges strongly to $\alpha_g(a)$ only requires us to consider vectors in $M\Omega$. Accordingly
 let $x\in M$ be arbitrary. Then
 \begin{align*}\|(\alpha_g(a)-b_m)x\Omega\|_2&=
 \|(\alpha_g(a)-b_m)Jx^*J\Omega\|_2
 =\|Jx^*J(\alpha_g(a)-b_m)\Omega\|_2\\
 &\leqslant \|x^*\|\,\|\alpha_g(a)-b_m\|_2 \rightarrow 0 \text{ as }m\to\infty.
 \end{align*}
 Thus $\alpha_g(a)\in \overline{A_{n+1}}^{\rm SOT}$, establishing (iii).
 
 Now let $B_F$ be the C$^*$-algebra generated by $\bigcup_{n=1}^\infty A_n$, and let $M_F$ denote its weak closure. From (iii),  $M_F$ is $\alpha$-invariant. The Hilbert space $\mathcal{H}_F$ spanned by the $\mathcal{H}_n$'s is a separable $M_F$-invariant subspace of $L^2(M)$, and the restriction of $M_F$ to $\mathcal{H}_F$ is faithful since $\mathcal{H}_F$ contains the separating vector $\Omega$. Thus $M_F$ has a separable predual as required.
 
 Now suppose that $B_F$ and $M_F$ have been constructed to satisfy conditions (1)-(4) for all subsets $F$ of cardinality at most $n$. Consider a fixed but arbitrary subset $F$ of cardinality $n+1$, and list the subsets of $F$ of cardinality $n$ as $S_1,\ldots,S_{n+1}$. The construction of $B_F$ and $M_F$ is accomplished exactly as above, starting the induction by choosing $A_1$ to be the separable C$^*$-algebra generated by $\bigcup_{i=1}^{n+1}B_{S_i}$.  This guarantees that the  nesting properties of condition (2) are satisfied.
 \end{proof}

\vskip 0.03in 
\noindent 
We will also need the following two lemmas. These were first established under the assumption of a separable predual, so the proofs that we give will reduce the general situations to the separable predual cases.

\begin{lem}[\cite{BCM1}, Theorem 6.10]\label{Invariant Subspaces Lemma} Let $G$ be a discrete group, and $\alpha$ an ergodic, trace-preserving action of $G$ on a finite von Neumann algebra $(M,\tau)$.  Then any finite-dimensional subspace $\mathcal{K} \subseteq L^2(M, \tau)$ which is invariant under the associated representation $V$ of $G$ on $L^2(M,\tau)$ is contained in $M \Omega$.  
\end{lem}

\begin{proof}
Replacing $\mathcal{K}$ by $\mathcal{K}+\mathcal{K}^*$, we may assume without loss of generality that $\mathcal{K}$ is self-adjoint. Let $\{\xi_i,\ i\geqslant 1\}$ be a norm dense set of vectors in $\mathcal{K}_{\mathrm{s.a.}}$. Associated to each $\xi_i$ is a possibly unbounded self-adjoint operator $L_{\xi_i}$ affiliated to $M$, (see \cite[Theorem B.4.1]{SiSm}), and so the spectral projections of $L_{\xi_i}$ lie in $M$. Let $A$ be the unital separable C$^*$-algebra generated by the spectral projections of each $\xi_i$ corresponding to the intervals $(-\infty,r)$ for all rationals $r$. Let $N$ be the strong closure of $A$, represented faithfully on $L^2(N,\tau)$. Then $\xi_i\in L^2(N,\tau)$ for $i\geqslant 1$, so $L^2(N,\tau)$ contains $\mathcal{K}_{\mathrm{s.a.}}$ and thus also $\mathcal{K}$. Since every $L_\xi$ for $\xi\in \mathcal{K}_{\mathrm{s.a.}}$ is affiliated to $N$, this algebra contains all spectral projections of the $L_\xi$'s. For a fixed $g\in G$, uniqueness of the spectral resolution shows that the spectral projections of $L_{V_g(\xi_i)}$ are all of the form $\alpha_g(p)$ where $p$ ranges over the set of spectral projections for $L_{\xi_i}$. Thus $\alpha_g$ maps $A$ into $N$ and so also maps $N$ into  $N$. This applies equally to $\alpha_{g^{-1}}$, and it follows that each $\alpha_g$ restricts to an automorphism of $N$. By Kaplansky density, $L^2(A,\tau)$ is equal to $L^2(N,\tau)$, and the latter space is thus separable. We conclude that $N$ is an $\alpha$-invariant von Neumann algebra containing $\mathcal{K}$ and having separable predual, so the result now follows from the separable predual case.
\end{proof}

\begin{rem}\label{notfinite}
    Note that arguing initially as in the proof of \cite[Theorem 6.10]{BCM1} it is easy to see that Lemma \ref{Invariant Subspaces Lemma} holds for the pair $(M,\rho)$, where $\rho$ is a faithful normal state on $M$ which is not assumed to be finite. In this case, $\mathcal{K}$ will be contained inside $M^\rho\Omega_\rho$. 
\end{rem}

\begin{lem}[\cite{BCM2}, Theorem 4.7]\label{Kronecker Subalgebra Lemma} Let $G$ be a discrete group, and $\alpha$ an ergodic, trace-preserving action of $G$ on a finite von Neumann algebra $(M,\tau)$. If $\alpha$ is also compact, then $M$ is injective.  Furthermore, there is an upwardly directed family 
of finite-dimensional $\alpha$-invariant subspaces of $M$ whose union is dense in $L^2(M,\tau)$.  
\end{lem}

\begin{proof} By Lemma \ref{sepsubalg}, $M$ is the union of an upwardly directed net $\{M_F\}$ indexed by the finite subsets $F$ of $M$, and each $M_F$ is $\alpha$-invariant and has a separable predual. Applying the known separable predual case, each $M_F$ is injective and $L^2(M_F,\tau)$ has a dense subspace that is  the upwardly directed union of finite-dimensional $\alpha$-invariant subspaces. Then $M$ is injective, and finite sums of these finite-dimensional subspaces of the $L^2(M_F,\tau)$'s give the required $\alpha$-invariant finite-dimensional subspaces whose union is dense in $L^2(M,\tau)$.
\end{proof}

Let $\alpha$ be a trace-preserving action of a discrete group $G$ on a finite von Neumann algebra $M$, and define a group of unitaries $\{V_g:g\in G\}\subseteq \mathbf{B}(L^2(M))$ by 
\begin{equation}\label{defofvg}
V_g(x\Omega)=\alpha_g(x)\Omega,\ \ \ x\in M.
\end{equation}
Let $\Gamma$ be the SOT$^*$-closure of this group. Then $\Gamma$ is a group of unitaries in $\mathbf{B}(L^2(M))$.
\begin{lem}\label{compactgamma}
Let $M$, $G$ and $\Gamma$ be as defined above. Further assume that $M$ has a separable predual and that the action of $G$ is compact. Then $\Gamma$ is SOT$^*$-compact.
\end{lem}

\begin{proof}
By hypothesis, $L^2(M)$ is norm separable so fix a dense sequence $\{\xi_i=m_i\Omega\}_{i=1}^\infty$ in the unit ball of $L^2(M)$ with $m_i\in M$. The strong$^*$ topology on the unit ball of $\mathbf{B}(L^2(M))$ is metrizable by the metric
\[d(s,t)=\sum_{n=1}^\infty (\|(s-t)\xi_n\|+\|(s^*-t^*)\xi_n\|)2^{-n},\ \ \ s,t\in \mathbf{B}(L^2(M)).\]
Then $\Gamma$ is a SOT$^*$-closed subset of a separable metric space, so it suffices to show that it is sequentially compact. We will extract an SOT$^*$-convergent subsequence from an arbitrary sequence $\{V_{g_i}:i\geqslant 1\}$.

Relabel this sequence as $\{u_{01},u_{02},u_{03},\ldots\}$, and note that $u_{0j}(m_1\Omega),u_{0j}^*(m_1\Omega)\in {\mathrm{Orb}}(m_1)\Omega$ for $j\geqslant 1$. This orbit has $\|\cdot\|_2$-compact closure, so there is a subsequence $\{u_{11},u_{12},u_{13},\ldots\}$ so that the sequences $\{u_{1j}(m_1\Omega)\}_{j=1}^\infty$
and $\{u_{1j}^*(m_1\Omega)\}_{j=1}^\infty$ are convergent in $\|\cdot\|_2$-norm. Repeating this argument, we obtain successive subsequences $\{u_{ij}\}_{j=1}^\infty$ for $i\geqslant 1$ so that the sequences
$\{u_{ij}(m_i\Omega)\}_{j=1}^\infty$
and $\{u_{ij}^*(m_i\Omega)\}_{j=1}^\infty$ are $\|\cdot\|_2$-convergent  for each $i\geqslant 1$. It is then easy to see that the diagonal subsequence $\{u_{ii}\}_{i=1}^\infty$ converges in the strong$^*$ topology.  Thus $\Gamma$ is SOT$^*$ sequentially compact, and so is compact.
\end{proof}

We now come to the main result of this section.

\begin{thm}\label{RelativeHAPTheorem} Let $G$ be a discrete group and $\alpha$ a trace-preserving, ergodic action of $G$ on a finite von Neumann algebra $(M,\tau)$.  Then the following conditions are equivalent:
\begin{itemize}
    \item [(i)] The action $\alpha$ is compact $($see Definition \ref{basicdef}$)$.
    \medskip
    \item [(ii)] $\mathcal{QN}(L(G) \subseteq M \rtimes_\alpha G)'' = M \rtimes_\alpha G$.
    \medskip
    \item [(iii)] The von Neumann algebra generated by $\mathcal{QN}^{(1)}(L(G) \subseteq M \rtimes_\alpha G)$ is $M \rtimes_\alpha G$.
    \medskip
    \item[(iv)] $M\rtimes_\alpha G$ has property H relative to the subalgebra $L(G)$.
\end{itemize}
\end{thm}

\begin{proof}
\noindent Note that the first three conditions are equivalent by Theorem \ref{QuasinormalizerCompactSystemTheorem} and the discussion following it, and that condition (iv) implies condition (ii),  by Proposition 3.4 in \cite{Po}.  Thus, the proof will be complete when we show that conditions (i)-(iii) imply (iv).  Our strategy will be to construct a net of normal, completely positive maps $\phi_\lambda$ on $M$ which 
\begin{itemize}
    \item approximate the identity map on $M$ pointwise in $\|\cdot\|_2$,
    \item  extend to compact operators on $L^2(M)$, and
    \item  commute with the automorphisms $\alpha_g$, $g \in G$.  
\end{itemize}  
We  assume that conditions (i)-(iii) hold, and we then establish (iv).

By Lemma \ref{Kronecker Subalgebra Lemma}, $M$ is hyperfinite, so can be written as   $M = \overline{(\bigcup_{\lambda\in \Lambda} M_\lambda)}^{w^*}$, where the $M_\lambda$'s form an upwardly directed net of finite-dimensional   $\ast$-subalgebras of $M$. Denote by $E_\lambda$ the trace-preserving conditional expectation of $M$ onto $M_\lambda$.  For each $g \in G$, we have a completely positive map $\phi_{\lambda,g}=\alpha_g \circ E_\lambda \circ \alpha_g^{-1}$ on $M$ which, by uniqueness, is equal to the trace-preserving conditional expectation of $M$ onto $\alpha_g(M_\lambda)$.  We note that these maps can be viewed as contractions on $L^2(M)$; this follows easily from the inequality $T(x)^*T(x)\leqslant T(x^*x)$ for any completely positive contraction $T$ on $M$. 

In order to make use of the earlier lemmas, we impose a temporary requirement that $M$ should have a separable predual. The general case will then be deduced from this special situation.

Let $\Gamma$ be the SOT$^*$-compact group of  Lemma \ref{compactgamma}, which is the SOT$^*$-closure of the set of operators $\{V_g:g\in G\}$ defined in \eqref{defofvg}.  For $\gamma \in \Gamma$, we define 
\[ \phi_{\lambda,\gamma}(x) = \gamma E_\lambda(\gamma^* x \gamma) \gamma^*,\,\,x \in M,\]
noting that this map coincides with $\phi_{\lambda,g}$ when $\gamma=V_g$.
Since $E_\lambda$ is normal and completely positive, these two properties pass to $\phi_{\lambda,\gamma}$.  For each $x \in M$ and vectors $\xi,\eta\in L^2(M)$, the scalar valued map 
$\gamma \mapsto \langle \phi_{\lambda,\gamma}(x)\xi,\eta\rangle$ is easily seen to be SOT$^*$-continuous on $\Gamma$, using the normality of $E_\lambda$. This enables us to define a map  $\phi_\lambda$ on $M$ by 
\begin{equation}\label{intdef}
 \ip{\phi_\lambda(x)\xi}{\eta} = \int_{\Gamma}\ip{\phi_{\lambda,\gamma}(x)\xi}{\eta}\,d\mu(\gamma), \quad \quad x \in M, \ \ \  \xi, \eta \in L^2(M),
 \end{equation}
 where $\mu$ is left Haar measure on the compact group $\Gamma$.
 Then $\phi_\lambda$ maps $M$ into itself since, for $t \in M'$, we have 
\begin{align*}\ip{\phi_\lambda(x) t\xi}{\eta} & = \int_{\Gamma}\ip{t\phi_{\lambda,\gamma}(x)\xi}{\eta}\,d\mu(\gamma) 
= \int_{\Gamma}\ip{\phi_{\lambda,\gamma}(x)\xi}{t^*\eta}\,d\mu(\gamma)\\
&= \ip{\phi_\lambda(x) \xi}{t^* \eta} = \ip{ t \phi_\lambda(x) \xi}{\eta}, \end{align*}
for any such $x \in M,$ $\xi, \eta \in L^2(M)$,
showing that $\phi_\lambda(x)\in M''=M$.  Complete positivity of $\phi_\lambda$ follows from complete positivity of the maps $\phi_{\lambda,\gamma}$. Each $\phi_\lambda$ is trace-preserving and, further, $\phi_\lambda$ is a $\norm{\cdot}_2$-norm contraction, so has a bounded extension $T_{\phi_\lambda}: L^2(M) \rightarrow L^2(M)$.  To see that $\phi_\lambda$ is $w^*$-continuous, it suffices to consider a uniformly bounded net $\{x_\delta\}_{\delta\in \Delta}$ converging to 0 in the $w^*$-topology, by the Krein-Smulian theorem.  We apply Lemma \ref{uniformbound} to the finite-rank map $E_\lambda$ to obtain an arbitrarily small bound on the integrand in \eqref{intdef}, showing that $\lim_\delta\langle \phi_\lambda(x_\delta) \xi,\eta\rangle =0$ for all $\xi,\eta\in L^2(M)$. Since the net $\{\phi_\lambda(x_\delta)\}_{\delta\in \Delta}$ is uniformly bounded, we conclude that $\phi_\lambda$ is  a normal, completely positive map on $M$.  Moreover, translation invariance of $\mu$ implies that $\phi_\lambda \circ \alpha_g = \alpha_g \circ \phi_\lambda$, for any $g \in G$.  

We now show that $\lim_\lambda\norm{\phi_\lambda(x)-x}_2 = 0$ for any $x \in M$.  By the density result of Lemma \ref{Kronecker Subalgebra Lemma}, it suffices to consider an element $x$ that lies in a finite-dimensional $\alpha$-invariant subspace $X$ of $M$. By scaling, we may assume that $\|x\|_2=1$, and we now fix $\eps >0$. Choose $\lambda_0$ so that  $X \subseteq_{\eps/2} M_{\lambda}$ whenever $\lambda \geqslant \lambda_0$, meaning that 
\[ \sup_{y \in X,\,\norm{y}_2\leqslant 1} \mathrm{dist}_{\|\cdot\|_2}(y,M_\lambda)<\eps/2, \ \ \ \ \lambda\geqslant \lambda_0.\]
  Since $X$ is $\alpha$-invariant, it follows that
\[ \sup_{g \in G} \norm{\alpha_g^{-1}(x)-E_{\lambda}(\alpha_g^{-1}(x))}_{2}\leqslant \eps\]
whenever $\lambda \geqslant \lambda_0$.  Then, for all such $\lambda$ and any $g \in G$, we have
\[\norm{x-\phi_{\lambda,g}(x)}_2=\norm{x-\alpha_g \circ E_\lambda \circ \alpha_g^{-1}(x)}_2\leqslant\eps,\]
and the SOT$^*$-density of $\{V_g:g\in G\}$ in $\Gamma$ implies that $\|x-\phi_{\lambda,\gamma}(x)\|_2\leqslant \eps$ for $\lambda\geqslant\lambda_0$ and $\gamma\in \Gamma$.
Averaging over $\Gamma$ gives $\norm{x-\phi_\lambda(x)}_2 \leqslant \eps$ for  $\lambda\geqslant\lambda_0$, and this establishes that   $\lim_\lambda\norm{\phi_\lambda(x)-x}_2 = 0$, as required.  
\vskip 0.07in 
\noindent Next, we show that the associated operator $T_{\phi_\lambda}$ on $L^2(M)$ is compact.  Let $\{z_j\}_{j=1}^\infty$ be a sequence in $L^2(M)$ converging weakly to zero.  By $\|\cdot\|_2$-density of $M$ in $L^2(M)$, we may assume that $z_j\in M$ for $j\geqslant 1$. Then, for any $\lambda$, $\lim_{j\to\infty}\|E_{\lambda}(z_j)\|_2=0$.  Similarly,  $\lim_{j\to\infty}\|E_{\alpha_g(M_\lambda)}(z_j)\|_2=0$  for any $g \in G$.  By the dominated convergence theorem, 
\begin{equation*}
\|T_{\phi_\lambda}(z_j)\|_2=\|\phi_\lambda(z_j)\|_2=\norm{\int_\Gamma \phi_{\lambda,\gamma}(z_j)d\mu(\gamma)}_2 \leqslant \int_\Gamma \norm{\phi_{\lambda,\gamma}(z_j)}_2 d\mu(\gamma) \longrightarrow 0
\end{equation*}
as $j \rightarrow \infty$, and so each $T_{\phi_\lambda}$ is compact.  Thus, we have produced a net $\{\phi_\lambda\}$ of completely positive maps on $M$ which commute with $\alpha_g,$ $g \in G$, approximate the identity map on $M$ in the $\|\cdot\|_2$, and extend to compact operators on $L^2(M)$. 

By Lemma \ref{cp extension}, these extend to a net $\{\Phi_\lambda\}$ of completely positive $L(G)$-bimodule maps on $M\rtimes_\alpha G$, given by $\Phi_\lambda(x) = \sum_{g\in G} \phi_\lambda(x_g) u_g$, for $x = \sum_{g\in G} x_g u_g \in M \rtimes_\alpha G$. 

To complete the proof of the separable predual case, we show that this net satisfies the properties  of Definition \ref{relative HAP}.  The first of these is clear because the maps that we have constructed are all trace-preserving. Since $\phi_\lambda(x) \rightarrow x$ in $\norm{\cdot}_2$ for each $x \in M$, we also have that  $\Phi_\lambda(m u_g) = \phi_\lambda(m)u_g \rightarrow m u_g$ in $\norm{\cdot}_2$ for any $m \in M$ and $g \in G.$ A standard approximation argument then shows that $\norm{\Phi_\lambda(x) - x}_2 \rightarrow 0$ for any $x = \sum_g x_g u_g \in M\rtimes_\alpha G$.  This proves that the $\Phi_\lambda$'s satisfy (ii).  

 We now show that (iii) is satisfied.  Note first that any finite-dimensional $G$-invariant subspace of $M \Omega \subseteq L^2(M)$ may be associated to a finitely generated $L(G)$-module in $L^2(M\rtimes_\alpha G)$, as follows.  Let $X$ be such a subspace, and use Gram-Schmidt  to find an orthonormal basis $(m_1, \ldots , m_n)$ of $X$, with $m_i \in M$, $1 \leqslant i \leqslant n$.  Then the operators $p_i = m_i e_{L(G)} m_i^* \in \ip{M\rtimes_\alpha G}{e_{L(G)}}$ are mutually orthogonal projections, so that $\sum_{i=1}^n m_i e_{L(G)} m_i^* \in \ip{M\rtimes_\alpha G}{e_{L(G)}}$ is a projection, and its range is the right $L(G)$-module $\Hil_X$  in $L^2(M\rtimes_\alpha G)$  generated by the vectors $m_i \Omega$, $1\leqslant i \leqslant n$. 
Denote this projection by $p_{\Hil_X}$, and denote the orthogonal projection of $L^2(M)$ onto $X$ by $p_X$.  Let $U:L^2(M \rtimes_\alpha G) \rightarrow L^2(M) \otimes \ell^2(G)$ be the unitary given by $U(y u_g \Omega) = y \Omega \otimes \delta_g$, for $y \in M$ and $g \in G$.    Then, for any such $y$ and $g$, we have
\begin{align*}(p_X \otimes 1)(y\Omega \otimes \delta_g) & = \sum_{i=1}^n \tau(m_i^* y)m_i \Omega \otimes \delta_g = \sum_{i=1}^n m_i E_{L(G)}(m_i^* y)\Omega \otimes \delta_g\\
                                                        & = U\sum_{i=1}^n m_i E_{L(G)}(m_i^*y)u_g \Omega 
                                                         = U \sum_{i=1}^n m_i e_{L(G)} m_i^* (y u_g \Omega).\end{align*}
Thus, $p_X \otimes 1 = U p_{\Hil_X} U^*.$  Note also that $T_{\phi_\lambda}\otimes 1 = U T_{\Phi_\lambda} U^*$ for each $\lambda$.        
\vskip 0.07in 
\noindent To show that the $T_{\Phi_\lambda}$'s satisfy condition (iii) of Definition \ref{relative HAP}, fix $\lambda$, set $\eps > 0$ and choose a finite-dimensional $G$-invariant subspace $X$ of $M \Omega$ with the property that $\norm{T_{\phi_\lambda}(1-p_X)}<\eps$, possible by compactness of $T_{\phi_\lambda}$ and Lemma \ref{Kronecker Subalgebra Lemma}.  As above, associate to $X$ an $L(G)$-module $\Hil_X = \overline{\sum_{i=1}^n m_i L(G)}$, and a finite-trace projection $p_{\Hil_X}=\sum_{i=1}^n m_i e_{L(G)} m_i^* \in \ip{M \rtimes_\alpha G}{e_{L(G)}}$.
% with $Tr(p_{\Hil_X})<\infty$.  
 From our characterization of $T_{\Phi_\lambda}$, we then have
\begin{align*}\norm{T_{\Phi_\lambda}(1-p_{\Hil_X})} = \norm{U T_{\Phi_\lambda} U^*U(1-p_{\Hil_X})U^*}&=\norm{(T_{\phi_\lambda} \otimes 1)(1 \otimes 1 - p_X \otimes 1)}\\
& =\norm{T_{\phi_\lambda}(1-p_X)}<\eps.\end{align*}
This proves that the net $(T_{\Phi_\lambda})$ satisfies condition (iii) of Definition \ref{relative HAP}, completing the proof of the separable predual case.

We now consider the general case where there is no assumption of a separable predual. We form a net 
$\Lambda=\{(F,\varepsilon): F\subseteq M \text{ is finite and } \varepsilon >0\}$, and we order this by 
\[(F_1,\varepsilon_1)\leqslant (F_2,\varepsilon_2) \text{ if and only if }F_1\subseteq F_2\text{ and } \varepsilon_2\leqslant \varepsilon_1.\]

For $\lambda=(F,\varepsilon)\in \Lambda$, define $\phi_\lambda:M\to M$ as follows. By Lemma \ref{sepsubalg}, there exists an $\alpha$-invariant von Neumann subalgebra $M_F$ satisfying $F\subseteq M_F \subseteq M$ and $M_F$ has a separable predual. Here we use $e_{L(G),F}$ for the Jones projection arising from the inclusion $L(G)\subseteq M_F\rtimes_\alpha G$, to distinguish it from $e_{L(G)}$ for the inclusion
$L(G)\subseteq M\rtimes_\alpha G$.
From our initial case, $L(G)\subseteq M_F\rtimes_\alpha G$ has the Haagerup approximation property, so there exists a normal completely positive $L(G)$-bimodule map $\psi_\lambda: 
M_F\rtimes_\alpha G \to M_F\rtimes_\alpha G$ with the following properties:
\begin{itemize}
\item[1.] $\tau\circ \psi_\lambda= \tau$.
\medskip
\item[2.] $\|\psi_\lambda(z)-z\|_2 <\varepsilon$ for $z\in F$.
\medskip
\item[3.] For $\delta >0$, there exists a projection $p\in \langle M_F\rtimes_\alpha G,e_{L(G),F}\rangle$ such that 
Tr${}_F(p)<\infty$ and $\|T_{\psi_\lambda}(1-p)\|<\delta$.
\medskip
\end{itemize} 
Then set $\phi_\lambda=\psi_\lambda \circ E_F$ where $E_F:M\rtimes_\alpha G\to M_F\rtimes_\alpha G$ is the conditional expectation. Then 1. and 2. hold for $\phi_\lambda$.  Moreover, 
$T_{\phi_\lambda}=T_{\psi_\lambda \circ E_F}=T_{\psi_\lambda}\circ e_{M_F\rtimes_\alpha G}$,
where $e_{M_F\rtimes_\alpha G}$ is the Jones projection.

Given $\delta>0$, choose $p\in \langle
M_F\rtimes_\alpha G,e_{L(G),F}\rangle$ for $T_{\psi_\lambda}$ as above, and consider the inclusion \[L(G)\subseteq M_F\rtimes_\alpha G\subseteq \mathcal{M}:=\overline{\text{span}(M_F\rtimes_\alpha G)e_{L(G)}M_F\rtimes_\alpha G}^{w*}\subseteq \langle M\rtimes_\alpha G, e_{L(G)}\rangle.\] The canonical semifinite trace Tr on $\langle M\rtimes_\alpha G, e_{L(G)}\rangle$ restricts to a semifinite trace on $\mathcal{M}$ such that the hypotheses of \cite[Theorem 4.3.15]{SiSm} are satisfied. It follows that the association $xe_{L(G),F}y\mapsto xe_{L(G)}y$, $x,y\in M_F\rtimes_\alpha G$ extends to a trace-preserving isomorphism $\pi$ from $\langle M_F\rtimes_\alpha G, e_{L(G),F}\rangle$ to its image inside $\langle M\rtimes_\alpha G, e_{L(G)}\rangle$.

Then $\pi(p)\in \langle M\rtimes_\alpha G,e_{L(G)}\rangle$ is a projection such that Tr$(\pi(p))<\infty$ and commutes with $e_{M_F\rtimes_\alpha G}$, so $e_{M_F\rtimes_\alpha G}\pi(p)=e_{M_F\rtimes_\alpha G}\pi(p)e_{M_F\rtimes_\alpha G}=p$. Consequently, 
\[\|T_{\phi_\lambda}(1-\pi(p))\|=\|T_{\psi_\lambda}e_{M_F\rtimes_\alpha G}(1-\pi(p))\|=\| T_{\psi_\lambda}(1-p)\|<\delta.\]
Thus condition 3. also holds, completing the proof.
\end{proof}

$$ $$
\textbf{Acknowledgements}: The authors are grateful to the two anonymous referees for their numerous comments and suggestions which improved greatly the overall exposition of our paper.

\noindent JC acknowledges the support of Simons Collaboration Grant \#319001. His work on this project was also supported by the National Science Foundation while working as a Program Director in the Division of Mathematical Sciences. Any opinions, findings, conclusions
or recommendations expressed in this material are those of the authors and do not necessarily reflect
the views of the National Science Foundation.  IC has been supported by the NSF Grants  FRG-DMS-1854194 and DMS-2154637.  KM acknowledges the support of the grant SB22231267MAETWO008573 (IOE Phase II). RS acknowledges the support of Simons Collaboration Grant \#522375.

%%%%%%%%%%%%%%%%%%%%%%%%%%%%%%%%%%%%%%%%%%%%%%%%%%%%%%%%%%%%%%%%%%%%%%%%%%%%%%%%%%%%%%%%%%%%%%%%%%%%%%%%%%%%%%%%%%%%%%%%%%%%%%%%%
\end{document}